\documentclass[a4paper]{amsart}
\usepackage[utf8]{inputenc}
\usepackage[T1]{fontenc}
\usepackage{lmodern}
\usepackage{amssymb,amsxtra}
\usepackage{nicefrac,mathtools,enumitem}
\setlist[enumerate,1]{label=\textup{(\arabic*)}}

\usepackage{tikz}
\usetikzlibrary{matrix}
\tikzset{cd/.style=matrix of math nodes,row sep=2em,column sep=2em, text height=1.5ex, text depth=0.5ex}
\tikzset{cdar/.style=->,auto}
\tikzset{overar/.style={draw=white,double=black,double distance=.4pt,very thick}}

\usepackage{microtype}
\usepackage[pdftitle={Semidirect products of C*-quantum groups:
  multiplicative unitaries approach},
  pdfauthor={Ralf Meyer, Sutanu Roy and Stanisław Lech Woronowicz},
  pdfsubject={Mathematics; MSC }
]{hyperref}
\usepackage[lite]{amsrefs}

\renewcommand{\PrintDOI}[1]{\href{http://dx.doi.org/\detokenize{#1}}{doi: \detokenize{#1}}}

\BibSpec{book}{%
  +{}  {\PrintPrimary}                {transition}
  +{,} { \textit}                     {title}
  +{.} { }                            {part}
  +{:} { \textit}                     {subtitle}
  +{,} { \PrintEdition}               {edition}
  +{}  { \PrintEditorsB}              {editor}
  +{,} { \PrintTranslatorsC}          {translator}
  +{,} { \PrintContributions}         {contribution}
  +{,} { }                            {series}
  +{,} { \voltext}                    {volume}
  +{,} { }                            {publisher}
  +{,} { }                            {organization}
  +{,} { }                            {address}
  +{,} { \PrintDateB}                 {date}
  +{,} { }                            {status}
  +{}  { \parenthesize}               {language}
  +{}  { \PrintTranslation}           {translation}
  +{;} { \PrintReprint}               {reprint}
  +{.} { }                            {note}
  +{.} {}                             {transition}
  +{} { \PrintDOI}                   {doi}
  +{} { available at \url}            {eprint}
  +{}  {\SentenceSpace \PrintReviews} {review}
}

\numberwithin{equation}{section}

\theoremstyle{plain}
\newtheorem{theorem}[equation]{Theorem}
\newtheorem{lemma}[equation]{Lemma}
\newtheorem{proposition}[equation]{Proposition}

\theoremstyle{definition}
\newtheorem{definition}[equation]{Definition}

\theoremstyle{remark}
\newtheorem{remark}[equation]{Remark}

\newcommand{\tenscorep}{\mathbin{\begin{tikzpicture}[baseline,x=.75ex,y=.75ex] \draw[line width=.2pt] (-0.8,1.15)--(0.8,1.15);\draw[line width=.2pt](0,-0.25)--(0,1.15); \draw[line width=.2pt] (0,0.75) circle [radius = 1];\end{tikzpicture}}}

\newcommand*{\Braiding}[2]{\begin{tikzpicture}[baseline]
    \draw[-] (0,0) -- (1.4ex,1.4ex) node[right,inner sep=0pt] {$\scriptstyle #2$};
    \draw[-,draw=white,line width=2.4pt] (0,1.4ex) -- (1.4ex,0);
    \draw[-] (1.4ex,0) -- (0,1.4ex) node[left,inner sep=0pt] {$\scriptstyle #1$};
  \end{tikzpicture}}
\newcommand*{\Dualbraiding}[2]{\begin{tikzpicture}[baseline]
    \draw[-] (1.4ex,0) -- (0,1.4ex) node[left,inner sep=0pt] {$\scriptstyle #1$};
    \draw[-,draw=white,line width=2.4pt] (0,0) -- (1.4ex,1.4ex);
    \draw[-] (0,0) -- (1.4ex,1.4ex) node[right,inner sep=0pt] {$\scriptstyle #2$};
  \end{tikzpicture}}

\newcommand*{\corep}[1]{\textup{#1}}          
\newcommand*{\Corep}[1]{\mathbb{#1}}          
\newcommand*{\DuCorep}[1]{\hat{\Corep{#1}}}   

\newcommand*{\ket}[1]{\lvert#1\rangle}
\newcommand*{\bra}[1]{\langle#1\rvert}

\newcommand*{\Nhatdrinf}{\hat{\mathcal{N}}}
\newcommand*{\udrinf}{\mathcal U}

\newcommand*{\dom}{\mathcal D}
\newcommand*{\diff}{\textup d}

\newcommand*{\nb}{\nobreakdash}
\newcommand*{\alb}{\hspace{0pt}} 
\newcommand*{\Star}{$^*$\nb-}

\newcommand*{\C}{\mathbb C}
\newcommand*{\Z}{\mathbb Z}
\newcommand*{\R}{\mathbb R}
\newcommand*{\N}{\mathbb N}

\newcommand*{\T}{\mathbb T}

\newcommand*{\G}[1][G]{\mathbb #1}
\newcommand*{\DuG}[1][G]{\hat{\mathbb{#1}}}

\newcommand*{\Comult}[1][]{\Delta_{#1}}
\newcommand*{\DuComult}[1][]{\hat{\Delta}_{#1}}
\newcommand*{\Qgrp}[2]{\mathbb{#1}=(#2,\Comult[#2])}

\newcommand*{\Mod}[1]{\abs{#1}}
\newcommand*{\Ph}[1]{\Phi_{#1}}

\newcommand*{\Bound}{\mathbb B}
\newcommand*{\Comp}{\mathbb K}

\newcommand*{\ima}{\textup i}
\newcommand*{\univ}{\textup u}
\newcommand*{\transpose}{\mathsf T}

\newcommand*{\Contvin}{\textup C_0}
\newcommand*{\Contb}{\textup C_\textup b}
\newcommand*{\Cont}{\textup C}

\newcommand*{\Mor}{\textup{Mor}}
\newcommand*{\Id}{\textup{id}}

\newcommand*{\Multunit}[1][]{\mathbb W^{#1}}
\newcommand*{\multunit}[1][]{\textup W^{#1}}
\newcommand*{\DuMultunit}[1][]{\widehat{\mathbb W}^{#1}}
\newcommand*{\Dumultunit}[1][]{\widehat{\textup W}^{#1}}

\newcommand*{\BrMultunit}{\mathbb F}
\newcommand*{\DuBrMultunit}{\widehat{\BrMultunit}}

\newcommand*{\brmultunit}{\textup F}

\newcommand*{\Bichar}{\mathbb{V}}
\newcommand*{\bichar}{\textup V}

\newcommand*{\ProjBichar}{\mathbb{P}}
\newcommand*{\projbichar}{\textup P}
\newcommand*{\DuProjBichar}{\widehat{\ProjBichar}}
\newcommand*{\Duprojbichar}{\hat{\projbichar}}

\newcommand*{\Flip}{\Sigma}
\newcommand*{\flip}{\sigma}
\newcommand*{\Cst}{\textup C^*}
\newcommand*{\Cred}{\textup C^*_\textup r}

\newcommand*{\Hils}[1][H]{\mathcal{#1}}
\newcommand*{\Mult}{\mathcal M}
\newcommand*{\U}{\mathcal U}

\newcommand*{\defeq}{\mathrel{\vcentcolon=}}

\newcommand*{\abs}[1]{\lvert#1\rvert}

\newcommand*{\conj}[1]{\overline{#1}}
\newcommand*{\cl}[1]{\overline{#1}}

\DeclareMathOperator{\Aut}{Aut}

\begin{document}
\title[Semidirect products of C\textsuperscript{*}-quantum groups]{Semidirect products of C\textsuperscript{*}-quantum groups: multiplicative unitaries approach}

\author{Ralf Meyer}
\email{rmeyer2@uni-goettingen.de}
\address{Mathematisches Institut\\
  Georg-August Universität Göttingen\\
  Bunsenstraße 3--5\\
  37073 Göttingen\\
  Germany}

\author{Sutanu Roy}
\email{sutanu.roy@carleton.ca}
\address{School of Mathematics and Statistics\\
  Carleton University\\
  1125 Colonel By Drive\\
  K1S 5B6 Ottawa\\
  Canada.}

\author{Stanisław Lech Woronowicz}
\email{Stanislaw.Woronowicz@fuw.edu.pl}
\address{Instytut Matematyczny Polskiej Akademii Nauk\\ul.\@ Śniadeckich 8\\00-656 Warszawa\\Poland, and\\Katedra Metod Matematycznych Fizyki, Wydział Fizyki\\Uniwersytet Warszawski\\ul.\@ Pasteura 5\\02-093 Warszawa\\Poland}

\begin{abstract}
  \(\Cst\)\nb-quantum groups with projection are the noncommutative
  analogues of semidirect products of groups.  Radford's Theorem about
  Hopf algebras with projection suggests that any \(\Cst\)\nb-quantum
  group with projection decomposes uniquely into an ordinary
  \(\Cst\)\nb-quantum group and a ``braided'' \(\Cst\)\nb-quantum
  group.  We establish this on the level of manageable multiplicative
  unitaries.
\end{abstract}


\thanks{Supported by the German Research Foundation (Deutsche
  Forschungsgemeinschaft (DFG)) through the Research Training Group
  1493.  The second author was also supported by a Fields--Ontario
  postdoctoral fellowship.  The third author was partially supported
  by the Alexander von Humboldt-Stiftung and the National Science
  Center (NCN), grant 2015/17/B/ST1/00085.}

\maketitle

\section{Introduction}
\label{sec:introduction}

Many important Lie groups like the Poincar\'e group or the group of
motions of Euclidean space are defined as semidirect products of
smaller building blocks.  What is the quantum group analogue of a
semidirect product?  Such a notion should be useful to understand
quantum deformations of semidirect products.

For a semidirect product of groups, we need two groups \(G\)
and~\(H\) and an action of~\(G\) on~\(H\) by group automorphisms.  Since
non-commutative quantum groups cannot act on other quantum groups by
automorphisms, we need a different point of view: semidirect product
groups are the same as groups with a projection.  A semidirect
product of groups~\(G\ltimes H\) comes with a canonical group
homomorphism
\[
p\colon G\ltimes H\to G\ltimes H,\qquad
(g,h)\mapsto (g,1_H),
\]
which is idempotent, that is, \(p^2=p\).
Its kernel and image are \(H\subseteq G\ltimes H\)
and \(G\subseteq G\ltimes H\),
respectively.  The conjugation action of~\(G\)
on~\(H\)
needed for a semidirect product is the restriction of the conjugation
action of~\(G\ltimes H\)
on itself.  Therefore, an idempotent group homomorphism
\(p\colon K\to K\)
on a group~\(K\)
is equivalent to a semidirect product decomposition of~\(K\).

Now consider a quantum group with a projection, that is, with an
idempotent quantum group endomorphism.  What corresponds to the
building blocks \(G\) and~\(H\) in a semidirect product of groups?  If
``quantum group'' means ``Hopf algebra,'' then a theorem by
Radford~\cite{Radford:Hopf_projection} answers this question.  Here we
consider \(\Cst\)\nb-quantum groups, meaning
\(\Cst\)\nb-bialgebras coming from manageable multiplicative unitaries
(see~\cites{Woronowicz:Mult_unit_to_Qgrp,
  Soltan-Woronowicz:Multiplicative_unitaries}).  More precisely, we
work on the level of the multiplicative unitaries themselves to avoid
analytical difficulties.

Let us first recall Radford's Theorem.  It splits a Hopf algebra~\(C\)
with a projection \(p\colon C\to C\) into two pieces \(A\) and~\(B\).
The ``image'' of the projection~\(A\) is a Hopf algebra as well.  The
``kernel'' of the projection~\(B\) is only a Hopf algebra in a certain
braided monoidal category, namely, the category of Yetter--Drinfeld
modules over~\(A\).  The tensor product of two Yetter--Drinfeld
algebras is again a Yetter--Drinfeld algebra, for the diagonal
Yetter--Drinfeld module structure and a certain deformed
multiplication.  The comultiplication on~\(B\) is a homomorphism to
the deformed tensor product \(B\boxtimes B\).

Radford's Theorem contains two constructions.  One puts together \(A\)
and~\(B\) into their ``semidirect product''~\(C\) and describes the
projection~\(p\) on~\(C\).  The other splits~\(C\) into the two
factors \(A\) and~\(B\), with the Hopf algebra structure on~\(A\) and
the \(A\)-Yetter--Drinfeld algebra and braided Hopf algebra structure
on~\(B\).  The first construction is called ``bosonisation'' by
Majid~\cite{Majid:Hopfalg_in_BrdCat}.  The analogue of this
construction for \(\Cst\)\nb-quantum groups is described
in~\cite{Meyer-Roy-Woronowicz:Twisted_tensor_2}, except for the
projection that we expect on this semidirect product.  In particular,
the appropriate analogues of Yetter--Drinfeld algebras and their
deformed tensor product~\(\boxtimes\) are described
in~\cite{Meyer-Roy-Woronowicz:Twisted_tensor_2} for arbitrary
\(\Cst\)\nb-quantum groups.  For regular \(\Cst\)\nb-quantum groups
with Haar weights, this is already done by Nest and
Voigt~\cite{Nest-Voigt:Poincare}.

The ``projections'' on \(\Cst\)\nb-quantum groups that we use are
morphisms as introduced in \cites{Ng:Morph_of_Mult_unit,
  Meyer-Roy-Woronowicz:Homomorphisms}.  That is, a quantum group
morphism from \((C,\Comult[C])\) to \((A,\Comult[A])\) is a
bicharacter in \(\U\Mult(\hat{C}\otimes A)\).  Several equivalent
descriptions of such morphisms are given
in~\cite{Meyer-Roy-Woronowicz:Homomorphisms}, including functors
between the categories of \(\Cst\)\nb-algebra coactions that preserve
the underlying \(\Cst\)\nb-algebra, and Hopf \Star{}homomorphisms
between the associated universal quantum groups.  These are more
general than Hopf \Star{}homomorphisms between the reduced quantum
group \(\Cst\)\nb-algebras.

Thus a \(\Cst\)\nb-quantum group with projection consists of a
\(\Cst\)\nb-quantum group~\((C,\Comult[C])\) with a unitary multiplier
\(\projbichar\in\U\Mult(\hat{C}\otimes C)\) with certain properties.
To express these, we use a manageable multiplicative unitary
\(\Multunit\in\U(\Hils\otimes\Hils)\) that generates~\(C\); in
particular, \(\Multunit\) satisfies the pentagon equation
\begin{equation}
  \label{eq:Multunit_pentagon}
  \Multunit_{23}\Multunit_{12} =
  \Multunit_{12}\Multunit_{13}\Multunit_{23}
  \qquad\text{in }\U(\Hils\otimes\Hils\otimes\Hils).
\end{equation}
Then \(C\) and~\(\hat{C}\) act faithfully on~\(\Hils\).
Write~\(\ProjBichar\) for~\(\projbichar\) viewed as an operator on
\(\Hils\otimes\Hils\).  The condition that~\(\projbichar\) is a
bicharacter is equivalent to
\begin{equation}
  \label{eq:intro_pentagonal}
  \ProjBichar_{23}\Multunit_{12}
  = \Multunit_{12}\ProjBichar_{13}\ProjBichar_{23}
  \quad\text{and}\quad
  \Multunit_{23} \ProjBichar_{12}
  = \ProjBichar_{12}\ProjBichar_{13}\Multunit_{23}
  \qquad\text{in }\U(\Hils\otimes\Hils\otimes\Hils).
\end{equation}
The condition that~\(\projbichar\) is idempotent for the composition
of quantum group homomorphisms is equivalent to the pentagon equation
for~\(\ProjBichar\):
\begin{equation}
  \label{eq:PR_pentagon}
  \ProjBichar_{23}\ProjBichar_{12} =
  \ProjBichar_{12}\ProjBichar_{13}\ProjBichar_{23}
  \qquad\text{in }\U(\Hils\otimes\Hils\otimes\Hils).
\end{equation}
Thus a \(\Cst\)\nb-quantum group with projection is determined by two
unitaries \(\Multunit,\ProjBichar\in
\U(\Hils\otimes\Hils)\) that satisfy
\eqref{eq:Multunit_pentagon}--\eqref{eq:PR_pentagon}; in addition,
\(\Multunit\) must be manageable.  Equation~\eqref{eq:PR_pentagon}
means that~\(\ProjBichar\) is a multiplicative unitary in its own
right.  It is manageable if~\(\Multunit\) is.  The \(\Cst\)\nb-quantum
group~\((A,\Comult[A])\) it generates is the image of the projection.

It is much more difficult to describe the other
factor~\(B\).  As a \(\Cst\)\nb-algebra, it should be the generalised
fixed-point algebra for a canonical coaction of~\((A,\Comult[A])\)
on~\((C,\Comult[C])\).  In the group case, this says that
\(\Cont_0(H)\) is the generalised fixed-point algebra for the left or
right translation action of~\(G\) on~\(\Cont_0(G\ltimes H)\).
Unless~\(G\) is compact, this requires Rieffel's generalisation of
fixed-point algebras to group actions that are ``proper'' in a
suitable sense (see \cites{Rieffel:Integrable_proper,
  Meyer:Generalized_Fixed}).
Buss~\cites{Buss:thesis,Buss:GFPAQuantum} has generalised this theory
to locally compact quantum groups.  We only need the special case of
quantum homogeneous spaces, which is also treated by
Vaes~\cite{Vaes:Induction_Imprimitivity}.  All these approaches need
some regularity assumptions on~\((A,\Comult[A])\) and are technically
difficult.

We may avoid these difficulties by staying on the level of
multiplicative unitaries.  We already described a \(\Cst\)\nb-quantum
group with projection through two multiplicative unitaries
\(\Multunit,\ProjBichar\in \U(\Hils\otimes\Hils)\) on the
same Hilbert space that are linked by the
conditions~\eqref{eq:intro_pentagonal}.  We find that any such pair
comes from a ``braided multiplicative unitary'' over the
\(\Cst\)\nb-quantum group \((A,\Comult[A])\) generated
by~\(\ProjBichar\).

A braided multiplicative unitary is a unitary
\(\BrMultunit\in\U(\Hils[L]\otimes\Hils[L])\) for a Hilbert
space~\(\Hils[L]\) with a Yetter--Drinfeld module structure over
\((A,\Comult[A])\).  That is, \(\Hils[L]\) carries corepresentations
\(\corep{U}\in\U(\Comp(\Hils[L])\otimes A)\) and
\(\corep{V}\in\U(\Comp(\Hils[L])\otimes \hat{A})\) that are linked by
a Yetter--Drinfeld commutation relation.  In addition, \(\BrMultunit\)
is equivariant for the tensor product corepresentations
\(\corep{U}\tenscorep\corep{U}\) and \(\corep{V}\tenscorep\corep{V}\)
on~\(\Hils[L]\otimes\Hils[L]\) and satisfies the \emph{braided
  pentagon equation}:
\begin{equation}
  \label{eq:intro_braided_pentagon}
  \BrMultunit_{23}\BrMultunit_{12}
  = \BrMultunit_{12}(\Braiding{\Hils[L]}{\Hils[L]})_{23}
  \BrMultunit_{12}(\Braiding{\Hils[L]}{\Hils[L]})^*_{23}
  \BrMultunit_{23}
  \qquad\text{in }\U(\Hils[L]\otimes\Hils[L]\otimes\Hils[L]).
\end{equation}
Here~\(\Braiding{\Hils[L]}{\Hils[L]}\) denotes the braiding operator
on the tensor product of the Yetter--Drinfeld Hilbert
space~\(\Hils[L]\) with itself,
see~\cite{Meyer-Roy-Woronowicz:Twisted_tensor_2}.

Since \(A\) and~\(\hat{A}\) are represented faithfully on~\(\Hils\),
the unitaries \(\corep{U}\) and~\(\corep{V}\) are determined by
their images \(\Corep{U}\) and~\(\Corep{V}\)
in~\(\U(\Hils[L]\otimes\Hils)\).  It is convenient to
replace~\(\Corep{V}\) by \(\DuCorep{V} \defeq \Sigma \Corep{V}^*
\Sigma\in\U(\Hils\otimes\Hils[L])\).  We also write~\(\Multunit\)
instead of~\(\ProjBichar\); the multiplicative unitary for the
semidirect product quantum group will be denoted by~\(\Multunit[C]\).

Thus a braided multiplicative unitary is a family of four unitaries
\(\Multunit\in\U(\Hils\otimes\Hils)\),
\(\Corep{U}\in\U(\Hils[L]\otimes\Hils)\),
\(\DuCorep{V}\in\U(\Hils\otimes\Hils[L])\), and
\(\BrMultunit\in\U(\Hils[L]\otimes\Hils[L])\) for two Hilbert spaces
\(\Hils\) and~\(\Hils[L]\); these unitaries satisfy seven conditions:
the pentagon condition for~\(\Multunit\); one corepresentation
condition each for \(\Corep{U}\) and~\(\DuCorep{V}\), which link them
to~\(\Multunit\); the Yetter--Drinfeld condition linking
\(\Corep{U}\) and~\(\DuCorep{V}\); the equivariance of~\(\BrMultunit\)
with respect to \(\Corep{U}\tenscorep\Corep{U}\)
and~\(\DuCorep{V}\tenscorep\DuCorep{V}\); and the braided pentagon
equation for~\(\BrMultunit\).  We show that given these four unitaries
subject to these seven conditions, the unitary
\begin{equation}
  \label{eq:crossed_product_multunit}
  \Multunit[C]_{1234}\defeq
  \Multunit_{13}\Corep{U}_{23}\DuCorep{V}^*_{34}
  \BrMultunit_{24}\DuCorep{V}_{34}
  \qquad\text{in } \U(\Hils\otimes\Hils[L]\otimes\Hils\otimes\Hils[L])
\end{equation}
is multiplicative.  Furthermore, the unitaries \(\Multunit[C]\)
and \(\ProjBichar \defeq \Multunit_{13} \Corep{U}_{23}\)
on \(\Hils\otimes\Hils[L] \otimes\Hils\otimes\Hils[L]\)
satisfy the conditions
\eqref{eq:Multunit_pentagon}--\eqref{eq:PR_pentagon} that characterise
\(\Cst\)\nb-quantum
groups with projection.  The only analytic issue is to prove
that~\(\Multunit[C]\)
is manageable if the braided multiplicative unitary is manageable in a
suitable sense.  Otherwise, the claim is proved by a direct
computation.  This has to be lengthy, however, because all seven
conditions on our four unitaries must play their role.

Conversely, let \(\Multunit[C],\ProjBichar\in\U(\Hils\otimes\Hils)\)
be unitaries satisfying the conditions
\eqref{eq:Multunit_pentagon}--\eqref{eq:PR_pentagon},
with~\(\Multunit[C]\)
manageable.  Then we construct a braided multiplicative unitary based
on the unitary \(\Multunit=\ProjBichar\in\U(\Hils\otimes\Hils)\),
that is, we construct a Hilbert space~\(\Hils[L]\)
and unitaries \(\Corep{U}\in\U(\Hils[L]\otimes\Hils)\),
\(\DuCorep{V}\in\U(\Hils\otimes\Hils[L])\),
and \(\BrMultunit\in\U(\Hils[L]\otimes\Hils[L])\)
satisfying the conditions for a braided multiplicative unitary, and we
check that this braided multiplicative unitary is manageable.  When we
construct a pair \((\Multunit[C],\ProjBichar)\)
out of this data as in~\eqref{eq:crossed_product_multunit}, then we do
not get back the same data we started with because the underlying
Hilbert spaces have changed.  We show, however, that the resulting
\(\Cst\)\nb-quantum
groups with projection are the same.  This isomorphism is also
implemented by a quantum group isomorphism in the category constructed
in~\cite{Meyer-Roy-Woronowicz:Homomorphisms}.

When we start with a manageable braided multiplicative unitary, form
the crossed product as in~\eqref{eq:crossed_product_multunit} and go
back, we also get a different braided multiplicative unitary, which
should be ``equivalent'' to the one we started with.  Since we do not
discuss how a braided multiplicative unitary generates a braided
\(\Cst\)\nb-quantum bialgebra, we cannot yet express this equivalence.

We treat one example of a braided multiplicative unitary in detail,
namely, the one that defines the simplified quantum
\(\textup{E}(2)\) group, a variant of the quantum \(\textup{E}(2)\)
group introduced by Woronowicz
in~\cite{Woronowicz:Qnt_E2_and_Pontr_dual}.  We write down the
braided multiplicative unitary and check that it is manageable.
Similar computations appear in
\cites{Baaj:Regular-Representation-E-2,
  Woronowicz:Qnt_E2_and_Pontr_dual}.

\section{Projections on Quantum Groups}
\label{sec:proj}

A \emph{\(\Cst\)\nb-quantum group} is, by definition, a
\(\Cst\)\nb-bialgebra that is generated by a manageable multiplicative
unitary,
see~\cites{Woronowicz:Mult_unit_to_Qgrp,
  Soltan-Woronowicz:Multiplicative_unitaries}.  We do not assume a
\(\Cst\)\nb-quantum group to have Haar weights.  We fix a
\(\Cst\)\nb-quantum group \(\G[H] = (C,\Comult[C])\) and let
\(\Multunit\in\U(\Hils\otimes\Hils)\) be a manageable multiplicative
unitary on a Hilbert space~\(\Hils\) that generates it.  Let
\(\hat{\G[H]} = (\hat{C},\DuComult[C])\) be the dual quantum group.

A \emph{bialgebra morphism} \((A,\Comult[A])\to (C,\Comult[C])\)
between two \(\Cst\)\nb-bialgebras is a \(\Cst\)\nb-algebra morphism
\(f\colon A\to C\) (that is, a nondegenerate \Star{}homomorphism
\(A\to\Mult(C)\)) making the following diagram commute:
\[
\begin{tikzpicture}
    \matrix(m)[cd]{
    A&A\otimes A\\
    C&C\otimes C\\
  };
  \draw[cdar] (m-1-1) -- node {\(\Comult[A]\)} (m-1-2);
  \draw[cdar] (m-1-1) -- node[swap] {\(f\)} (m-2-1);
  \draw[cdar] (m-2-1) -- node {\(\Comult[C]\)} (m-2-2);
  \draw[cdar] (m-1-2) -- node {\(f\otimes f\)} (m-2-2);
\end{tikzpicture}
\]
This notion of morphism is too restrictive, however, because a group
homomorphism \(G\to H\)
need not induce a morphism \(\Cred(G)\to\Cred(H)\).
When we speak of morphisms of \(\Cst\)\nb-quantum
groups, we will mean those introduced by
Ng~\cite{Ng:Morph_of_Mult_unit}, and we shall use the equivalent
characterisations of these morphisms
in~\cite{Meyer-Roy-Woronowicz:Homomorphisms}.

\begin{definition}
  \label{def:quantum_group_proj}
  A \emph{\(\Cst\)\nb-quantum group with projection} is a
  \(\Cst\)\nb-quantum group with an idempotent quantum group
  endomorphism.
\end{definition}

Before we make this definition explicit, we consider the commutative
case.  It allows us to view \(\Cst\)\nb-quantum groups with projection
as \(\Cst\)\nb-quantum group analogues of semidirect products of
groups.

\begin{proposition}
  \label{pro:commutative_qg_projection}
  Let~\((C,\Comult[C])\)
  be a commutative \(\Cst\)\nb-quantum
  group with projection.  Then \(C\cong \Cont_0(G\ltimes H)\)
  for a semidirect product group, with the corresponding
  comultiplication, and the projection on~\(C\)
  comes from the group homomorphism \(G\ltimes H\to G\ltimes H\),
  \((g,h)\mapsto (g,1_H)\);
  here \(G\)
  and~\(H\)
  are locally compact groups and~\(G\)
  acts continuously on~\(H\)
  by automorphisms.  Conversely, any semidirect product group gives a
  commutative \(\Cst\)\nb-quantum group with projection in this way.
\end{proposition}

\begin{proof}
  Since~\(C\) is commutative, \(C\cong\Cont_0(K)\) for a locally
  compact group~\(K\).  A quantum group homomorphism from~\(C\) to
  itself is equivalent to a group homomorphism \(K\to K\), and the
  composition of quantum group homomorphisms also corresponds to the
  composition of group homomorphisms.  Thus a projection on~\(C\)
  corresponds to a group homomorphism \(p\colon K\to K\) with
  \(p\circ p=p\).  Let \(G\subseteq K\) and \(H\subseteq K\) be the
  image and kernel of~\(p\), respectively; these are locally compact
  groups as well.  Since~\(H\) is a normal subgroup, conjugation
  in~\(K\) lets \(G\subseteq K\) act continuously on~\(H\) by
  automorphisms.  The continuous maps \(m\colon G\times H\to K\),
  \((g,h)\mapsto g\cdot h\), and \(n\colon K\to G\times H\),
  \(k\mapsto (p(k),p(k^{-1})k)\), are inverse to each other and
  hence homeomorphisms.  The multiplication is given by
  \[
  m(g_1,h_1)\cdot m(g_2,h_2)
  = g_1h_1g_2h_2
  = g_1g_2 (g_2^{-1}h_1g_2)h_2
  = m(g_1g_2, (g_2^{-1}h_1g_2)h_2).
  \]
  Thus the homeomorphism~\(m\) is also a group isomorphism \(K\cong
  G\ltimes H\).  The converse assertion is routine to check.
\end{proof}

Now we make Definition~\ref{def:quantum_group_proj} explicit in
several different ways, corresponding to some of the equivalent
characterisations of quantum group morphisms
in~\cite{Meyer-Roy-Woronowicz:Homomorphisms}.  First we use unitaries
satisfying pentagon equations.

\begin{proposition}
  \label{pro:qg_projection_mu}
  A \(\Cst\)\nb-quantum group with projection is given by a Hilbert
  space~\(\Hils\) and two unitaries
  \(\ProjBichar,\Multunit\in\U(\Hils\otimes\Hils)\) that satisfy
  \begin{align*}
    \Multunit_{23}\Multunit_{12}
    &= \Multunit_{12}\Multunit_{13}\Multunit_{23},\\
    \ProjBichar_{23}\Multunit_{12}
    &= \Multunit_{12}\ProjBichar_{13}\ProjBichar_{23},\\
    \Multunit_{23} \ProjBichar_{12}
    &= \ProjBichar_{12}\ProjBichar_{13}\Multunit_{23},\\
    \ProjBichar_{23}\ProjBichar_{12}
    &= \ProjBichar_{12}\ProjBichar_{13}\ProjBichar_{23}
    \qquad\quad\text{in }\U(\Hils\otimes\Hils\otimes\Hils).
  \end{align*}
  In addition, \(\Multunit\) is manageable as a multiplicative
  unitary.
\end{proposition}

All four equations in Proposition~\ref{pro:qg_projection_mu} are variants
of the pentagon equation.

\begin{proof}
  \cite{Meyer-Roy-Woronowicz:Homomorphisms}*{Lemma 3.2} describes a
  quantum group morphism from~\(\G[H]\) to itself by a unitary
  \(\ProjBichar\in\U(\Hils\otimes\Hils)\) on the same Hilbert
  space~\(\Hils\) on which the manageable multiplicative
  unitary~\(\Multunit\) lives, subject to the two
  conditions~\eqref{eq:intro_pentagonal}, which are the second and
  third equation in our statement.  The first equation is the pentagon
  equation for~\(\Multunit\).  The fourth equation says that the
  quantum group endomorphism associated to~\(\ProjBichar\) is
  idempotent by \cite{Meyer-Roy-Woronowicz:Homomorphisms}*{Definition
    3.5}.
\end{proof}

Our first goal is to prove the following structural result:

\begin{proposition}
  \label{pro:idempotents_split}
  Any idempotent endomorphism \(p\colon \G[H]\to\G[H]\)
  of a \(\Cst\)\nb-quantum
  group~\(\G[H]\)
  splits.  That is, there are a \(\Cst\)\nb-quantum
  group~\(\G\)
  and quantum group morphisms \(a\colon \G\to\G[H]\),
  \(b\colon \G[H]\to\G\)
  with \(a\circ b=p\) and \(b\circ a=\Id_{\G}\).
\end{proposition}

The \(\Cst\)\nb-quantum
group~\(\G\)
is called the \emph{image} of the idempotent endomorphism~\(p\).
We first construct this image, then we describe \(a\)
and~\(b\)
and then we prove \(a\circ b=p\)
and \(b\circ a=\Id_{\G}\).
The proof of Proposition~\ref{pro:idempotents_split} will be finished
by Lemma~\ref{lem:projection_from_ij}.

The fourth equation in Proposition~\ref{pro:qg_projection_mu} says
that~\(\ProjBichar\) is a multiplicative unitary.

\begin{proposition}
  \label{pro:ProjBichar_manageable}
  The multiplicative unitary \(\ProjBichar\in\U(\Hils\otimes\Hils)\)
  is manageable.
\end{proposition}

\begin{proof}
  The multiplicative unitary~\(\Multunit\) is manageable by
  assumption.  This requires the existence of certain auxiliary
  operators \(Q\) and~\(\widetilde{\Multunit}\).  We use the same
  operator~\(Q\) for~\(\ProjBichar\).
  \cite{Woronowicz:Mult_unit_to_Qgrp}*{Theorem
    1.6} gives a unitary
  \(\widetilde{\ProjBichar}\in\U(\conj{\Hils}\otimes\Hils)\) with
  \[
  \bigl(x\otimes u\mid \ProjBichar\mid z\otimes y\bigr)
  = \bigl(\conj{z}\otimes Q u\mid \widetilde{\ProjBichar}\mid
  \conj{x}\otimes Q^{-1}y\bigr)
  \]
  for all \(x,z\in\Hils\), \(u\in\dom(Q)\) and \(y\in\dom(Q^{-1})\).
  Lemma~\ref{lem:V_tilde_commute_Q} shows that~\(\ProjBichar\)
  commutes with \(Q\otimes Q\).  So \(\widetilde{\ProjBichar}\)
  and~\(Q\) witness the manageability of the multiplicative
  unitary~\(\ProjBichar\) (see
  \cite{Woronowicz:Mult_unit_to_Qgrp}*{Definition 1.2}).
\end{proof}

Proposition~\ref{pro:ProjBichar_manageable} shows that~\(\ProjBichar\)
generates a \(\Cst\)\nb-quantum
group \(\G=(A,\Comult[A])\),
which is called the \emph{image of~\(\ProjBichar\)}.
Let \(\hat{\G}=(\hat{A},\DuComult[A])\) be its dual.

The unitary~\(\ProjBichar\) is the image of a unitary multiplier
\(\projbichar\in\U(\hat{C}\otimes C)\) by
\cite{Meyer-Roy-Woronowicz:Homomorphisms}*{Lemma 3.2}.  Hence slices
of~\(\ProjBichar\) are multipliers of
\(\hat{C}\subseteq\Bound(\Hils)\) and \(C\subseteq\Bound(\Hils)\),
respectively.  These slices generate \(\hat{A}\) and~\(A\),
respectively, so \(A\subseteq\Mult(C)\) and
\(\hat{A}\subseteq\Mult(\hat{C})\).

\begin{lemma}
  \label{lem:A_into_C_Hopf}
  The embeddings \(i\colon A\to\Mult(C)\) and \(j\colon
  \hat{A}\to\Mult(\hat{C})\) are bialgebra morphisms \(\G\to\G[H]\)
  and \(\hat{\G}\to\hat{\G[H]}\).
\end{lemma}

\begin{proof}
  First we claim that \(i\) and~\(j\) are \(\Cst\)\nb-algebra morphisms,
  that is, \(i(A)\cdot C = C\) and \(j(\hat{A})\cdot \hat{C} = \hat{C}\).
  The third condition in Proposition~\ref{pro:qg_projection_mu} is
  equivalent to
  \[
  \ProjBichar_{12}^*\Multunit_{23}\ProjBichar_{12}
  =\ProjBichar_{13}\Multunit_{23}
  \qquad\text{in }\U(\Hils\otimes\Hils\otimes\Hils).
  \]
  When we slice the first two legs on both sides
  by~\(\omega_1\otimes\omega_2\) for
  \(\omega_1, \omega_2\in\Bound(\Hils)_*\)
  and close in norm, we get \(C=A\cdot C\).
  The same argument works for~\(j\).

  The conditions in Proposition~\ref{pro:qg_projection_mu} also imply
  \begin{gather*}
    \Multunit_{23} \ProjBichar_{12}\Multunit[*]_{23}
    = \ProjBichar_{12}\ProjBichar_{13}
    = \ProjBichar_{23}\ProjBichar_{12} \ProjBichar_{23}^*,\\
    \Multunit[*]_{12} \ProjBichar_{23}\Multunit_{12}
    = \ProjBichar_{13}\ProjBichar_{23}
    = \ProjBichar_{12}^* \ProjBichar_{23}\ProjBichar_{12}.
  \end{gather*}
  Since \((\Id_D\otimes \Comult[C])(x) = \Multunit_{23}
  x_{12}\Multunit[*]_{23}\) for all \(x\in D\otimes C\) and
  \((\Id_D\otimes \Comult[A])(x) = \ProjBichar_{23}
  x_{12}\ProjBichar^*_{23}\) for all \(x\in D\otimes A\), the first
  equation says that \(\Id_{\hat{A}}\otimes\Comult[C]\) and
  \(\Id_{\hat{A}}\otimes\Comult[A]\) agree on~\(\ProjBichar\).  Since
  slices of~\(\ProjBichar\) generate~\(A\), this implies
  \(\Comult[C]|_A=\Comult[A]\), that is, \(i\) is a bialgebra
  morphism.  So is~\(j\) by a similar argument.
\end{proof}

The bialgebra morphisms \(i\) and~\(j\) give quantum group morphisms
\begin{alignat*}{2}
  V_i&=(\Id\otimes i)(\multunit[A]) \in\U(\hat{A} \otimes C)&\qquad
  &\text{from }A\text{ to }C,\\
  \hat{V}_j &=(j\otimes \Id)(\multunit[A]) \in\U(\hat{C} \otimes A)&\qquad
  &\text{from }C\text{ to }A.
\end{alignat*}
The quantum groups \(\G\)
and~\(\G[H]\)
may be generated by the multiplicative unitaries \(\ProjBichar\)
and~\(\Multunit\)
on the same Hilbert space~\(\Hils\).
Then the unitaries \(V_i\)
and~\(\hat{V}_j\)
are both represented by the same unitary~\(\ProjBichar\)
on~\(\Hils\otimes\Hils\);
the conditions in Proposition~\ref{pro:qg_projection_mu} allow us to
view~\(\ProjBichar\)
as a quantum group homomorphism \(\G\to\G[H]\),
\(\G[H]\to \G\),
\(\G[H]\to\G[H]\),
or as the identity quantum group homomorphism on~\(\G\).

\begin{lemma}
  \label{lem:projection_from_ij}
  The composite quantum group homomorphism \(V_i\circ \hat{V}_j\colon
  \G[H]\to\G\to\G[H]\) is the given projection
  \(\projbichar\in\U(\hat{C}\otimes C)\) on~\(\G[H]\).  The other composite
  \(\G\to\G[H]\to\G\) is the identity on~\(\G\).
\end{lemma}

\begin{proof}
  The composition of quantum group homomorphisms is described
  in~\cite{Meyer-Roy-Woronowicz:Homomorphisms} by a
  pentagon-like equation.  The two claims in the lemma are both
  equivalent to the pentagon equation for~\(\ProjBichar\).
\end{proof}

The description of a projection on a \(\Cst\)\nb-quantum group by a pair of bialgebra morphisms \((i,j)\) is unwieldy because it mixes quantum groups and their duals and because the composition \(\G\to\G[H]\to\G\) is computed only indirectly.

The quantum group morphism \(\G[H]\to\G\) is usually not representable by a bialgebra morphism \(C\to A\).  We may, however, also represent the quantum group morphism~\(j\) by a
bialgebra morphism \(\hat{j}^\univ\colon C^\univ\to A^\univ\)
between the universal quantum groups, see
\cite{Meyer-Roy-Woronowicz:Homomorphisms}*{Theorem 4.8}.  Similarly,
\(i\) lifts to a bialgebra morphism \(i^\univ\colon A^\univ\to
C^\univ\).  A \(\Cst\)\nb-quantum group with projection is
equivalent to a \(\Cst\)\nb-quantum group~\(\G[H]\) with a bialgebra
morphism \(p\colon C^\univ\to C^\univ\) satisfying \(p\circ p=p\) by
\cite{Meyer-Roy-Woronowicz:Homomorphisms}*{Theorem 4.8}.  Our
analysis above shows that for any such~\(p\) there are a \(\Cst\)\nb-quantum group~\((A,\Comult[A])\) and bialgebra morphisms
\(\hat{j}^\univ\colon C^\univ\to A^\univ\) and \(i^\univ\colon A^\univ\to
C^\univ\) with \(p=i^\univ\circ \hat{j}^\univ\) and \(\hat{j}^\univ \circ i^\univ = \Id_A\).  Thus a quantum group with projection is equivalent to two \(\Cst\)\nb-quantum groups with bialgebra morphisms
\(\hat{j}^\univ\colon C^\univ\to A^\univ\) and \(i^\univ\colon A^\univ\to
C^\univ\) with \(\hat{j}^\univ \circ i^\univ = \Id_A\).

Next we replace~\(\hat{j}\) by right and left quantum group morphisms:

\begin{proposition}
  \label{pro:projection_via_action}
  A \(\Cst\)\nb-quantum group with projection is equivalent to two
  \(\Cst\)\nb-\alb{}quantum groups \(\G[H]=(C,\Comult[C])\)
  and \(\G=(A,\Comult[A])\) with morphisms \(i\colon A\to C\) and
  \(\Delta_R\colon C\to C\otimes A\) such that the following
  diagrams commute:
  \[
  \begin{tikzpicture}
    \matrix(m)[cd]{
      A&A\otimes A\\
      C&C\otimes C\\
    };
    \draw[cdar] (m-1-1) -- node {\(\Comult[A]\)} (m-1-2);
    \draw[cdar] (m-1-1) -- node[swap] {\(i\)} (m-2-1);
    \draw[cdar] (m-2-1) -- node {\(\Comult[C]\)} (m-2-2);
    \draw[cdar] (m-1-2) -- node {\(i\otimes i\)} (m-2-2);
  \end{tikzpicture}\qquad
  \begin{tikzpicture}
    \matrix(m)[cd,column sep=4.5em]{
      C&C\otimes A\\
      C\otimes C&C\otimes C\otimes A\\
    };
    \draw[cdar] (m-1-1) -- node {\(\Delta_R\)} (m-1-2);
    \draw[cdar] (m-1-1) -- node[swap] {\(\Comult[C]\)} (m-2-1);
    \draw[cdar] (m-2-1) -- node {\(\Id_C\otimes\Delta_R\)} (m-2-2);
    \draw[cdar] (m-1-2) -- node {\(\Comult[C]\otimes\Id_A\)} (m-2-2);
  \end{tikzpicture}
  \]
  \[
  \begin{tikzpicture}
    \matrix(m)[cd,column sep=4.5em]{
      C&C\otimes A\\
      C\otimes A&C\otimes A\otimes A\\
    };
    \draw[cdar] (m-1-1) -- node {\(\Delta_R\)} (m-1-2);
    \draw[cdar] (m-1-1) -- node[swap] {\(\Delta_R\)} (m-2-1);
    \draw[cdar] (m-2-1) -- node {\(\Delta_R\otimes\Id_A\)} (m-2-2);
    \draw[cdar] (m-1-2) -- node {\(\Id_C\otimes\Comult[A]\)} (m-2-2);
  \end{tikzpicture}\qquad
  \begin{tikzpicture}
    \matrix(m)[cd]{
      A&A\otimes A\\
      C&C\otimes A\\
    };
    \draw[cdar] (m-1-1) -- node {\(\Comult[A]\)} (m-1-2);
    \draw[cdar] (m-1-1) -- node[swap] {\(i\)} (m-2-1);
    \draw[cdar] (m-2-1) -- node {\(\Delta_R\)} (m-2-2);
    \draw[cdar] (m-1-2) -- node {\(i\otimes\Id_A\)} (m-2-2);
  \end{tikzpicture}
  \]
  Another equivalent set of data is a pair of morphisms \(i\colon
  A\to C\) and \(\Delta_L\colon C\to A\otimes C\) with commutative diagrams
  \[
  \begin{tikzpicture}
    \matrix(m)[cd]{
      A&A\otimes A\\
      C&C\otimes C\\
    };
    \draw[cdar] (m-1-1) -- node {\(\Comult[A]\)} (m-1-2);
    \draw[cdar] (m-1-1) -- node[swap] {\(i\)} (m-2-1);
    \draw[cdar] (m-2-1) -- node {\(\Comult[C]\)} (m-2-2);
    \draw[cdar] (m-1-2) -- node {\(i\otimes i\)} (m-2-2);
  \end{tikzpicture}\qquad
  \begin{tikzpicture}
    \matrix(m)[cd,column sep=4.5em]{
      C&A\otimes C\\
      C\otimes C&A\otimes C\otimes C\\
    };
    \draw[cdar] (m-1-1) -- node {\(\Delta_L\)} (m-1-2);
    \draw[cdar] (m-1-1) -- node[swap] {\(\Comult[C]\)} (m-2-1);
    \draw[cdar] (m-2-1) -- node {\(\Delta_L\otimes\Id_C\)} (m-2-2);
    \draw[cdar] (m-1-2) -- node {\(\Id_A\otimes\Comult[C]\)} (m-2-2);
  \end{tikzpicture}
  \]
  \[
  \begin{tikzpicture}
    \matrix(m)[cd,column sep=4.5em]{
      C&A\otimes C\\
      A\otimes C&A\otimes A\otimes C\\
    };
    \draw[cdar] (m-1-1) -- node {\(\Delta_L\)} (m-1-2);
    \draw[cdar] (m-1-1) -- node[swap] {\(\Delta_L\)} (m-2-1);
    \draw[cdar] (m-2-1) -- node {\(\Id_A\otimes\Delta_L\)} (m-2-2);
    \draw[cdar] (m-1-2) -- node {\(\Comult[A]\otimes\Id_C\)} (m-2-2);
  \end{tikzpicture}\qquad
  \begin{tikzpicture}
    \matrix(m)[cd]{
      A&A\otimes A\\
      C&A\otimes C\\
    };
    \draw[cdar] (m-1-1) -- node {\(\Comult[A]\)} (m-1-2);
    \draw[cdar] (m-1-1) -- node[swap] {\(i\)} (m-2-1);
    \draw[cdar] (m-2-1) -- node {\(\Delta_L\)} (m-2-2);
    \draw[cdar] (m-1-2) -- node {\(\Id_A\otimes i\)} (m-2-2);
  \end{tikzpicture}
  \]
  Finally, the quantum group with projection is equivalent to a
  triple of morphisms \(i\colon A\to C\), \(\Delta_R\colon C\to
  C\otimes A\) and \(\Delta_L\colon C\to A\otimes C\) satisfying all
  the above conditions and, in addition,
  \[
  \begin{tikzpicture}
    \matrix(m)[cd,column sep=4.5em]{
      C&C\otimes C\\
      C\otimes C&C\otimes A\otimes C\\
    };
    \draw[cdar] (m-1-1) -- node {\(\Comult[C]\)} (m-1-2);
    \draw[cdar] (m-1-1) -- node[swap] {\(\Comult[C]\)} (m-2-1);
    \draw[cdar] (m-2-1) -- node {\(\Delta_R\otimes \Id_C\)} (m-2-2);
    \draw[cdar] (m-1-2) -- node {\(\Id_C\otimes \Delta_L\)} (m-2-2);
  \end{tikzpicture}
  \]
  Then the following diagram also commutes:
  \[
  \begin{tikzpicture}
    \matrix(m)[cd,column sep=4.5em]{
      C&A\otimes C\\
      C\otimes A&A\otimes C\otimes A\\
    };
    \draw[cdar] (m-1-1) -- node {\(\Delta_L\)} (m-1-2);
    \draw[cdar] (m-1-1) -- node[swap] {\(\Delta_R\)} (m-2-1);
    \draw[cdar] (m-2-1) -- node {\(\Delta_L\otimes \Id_A\)} (m-2-2);
    \draw[cdar] (m-1-2) -- node {\(\Id_A\otimes \Delta_R\)} (m-2-2);
  \end{tikzpicture}
  \]
\end{proposition}

\begin{proof}
  We have already seen that any projection on a \(\Cst\)\nb-quantum
  group~\(\G[H]\) has an image~\(\G\) and that there are a bialgebra
  morphism \(i\colon A\to C\) and a quantum group morphism
  \(\hat{j}\colon \G[H]\to\G\) with \(\hat{j}\circ i=\Id_{\G}\) and
  \(i\circ\hat{j} = p\), where~\(p\) denotes the given projection
  on~\(\G[H]\).  Now we describe~\(\hat{j}\) by a right quantum
  group morphism~\(\Delta_R\) as in
  \cite{Meyer-Roy-Woronowicz:Homomorphisms}*{Definition 5.1}.

  The first diagram above says that~\(i\)
  is a bialgebra morphism.  The second and third diagram together say
  that~\(\Delta_R\)
  is a right quantum group homomorphism from~\(C\)
  to~\(A\).
  The fourth diagram says that the composite \(A\to C\to A\)
  of these quantum group morphisms is the identity map.  Therefore,
  the other composite \(C\to A\to C\)
  is idempotent, hence a projection.  Thus \(i\)
  and~\(\Delta_R\)
  give a projection on~\(\G[H]\)
  with image~\(\G\).
  Conversely, any projection on a \(\Cst\)\nb-quantum
  group~\(\G[H]\)
  has an image by Proposition~\ref{pro:idempotents_split}, which
  gives~\(i\) and~\(\Delta_R\) as above.

  Replacing right by left quantum group morphisms shows that
  pairs~\((i,\Delta_L)\)
  as above are also equivalent to \(\Cst\)\nb-quantum
  groups with projection.  Of the two diagrams that relate
  \(\Delta_R\)
  and~\(\Delta_L\),
  the first one characterises when the right and left quantum group
  homomorphisms \(\Delta_R\)
  and~\(\Delta_L\)
  describe the same quantum group morphism, and the second one
  commutes automatically, see
  \cite{Meyer-Roy-Woronowicz:Homomorphisms}*{Lemma 5.7}.
\end{proof}

Let \(A\) and~\(B\) be~\(\Cst\)\nb-algebras and \(T\in\U(A\otimes
B)\).  Then~\(B\) is \emph{generated by~\(T\)} in the sense of
\cite{Woronowicz:Cstar_generated}*{Definition 4.1} if, for any
representation \(\xi\colon B\to\Bound(\Hils)\) and any
\(\Cst\)\nb-algebra \(C\subset\Bound(\Hils)\), the condition
\((\Id_A\otimes\xi)T\in\Mult(A\otimes C)\) implies that
\(\xi\in\Mor(B,C)\).

\begin{definition}[\cite{Daws-Kasprzak-Skalski-Soltan:Closed_qnt_subgrps}*{Definition
    3.2}]
  \label{def:closed_qnt_sb_grp}
  Let \(\Qgrp{I}{C}\) and \(\Qgrp{G}{A}\) be quantum groups.  We
  call~\(\G\) a \emph{closed quantum subgroup of\/~\(\G[I]\) in the
    sense of Woronowicz} if there is a bicharacter
  \(\bichar\in\U(\hat{C}\otimes A)\) that generates~\(\G\).
\end{definition}

In the situation of Proposition~\ref{pro:projection_via_action},
\((A,\Comult[A])\)
is indeed a closed quantum subgroup of~\((C,\Comult[C])\)
because the bicharacter
\((j\otimes\Id_A)(\multunit[A])\in \U(\hat{C}\otimes A)\)
generates~\(A\).
This is to be expected because \((A,\Comult[A])\)
is even a retract of~\((C,\Comult[C])\)
in the category of quantum group morphisms.

\subsection{Semidirect products}
\label{sec:semidirect}

In this section, we are going to show that the semidirect product
construction in \cite{Meyer-Roy-Woronowicz:Twisted_tensor_2}*{Section
  6} gives examples of \(\Cst\)\nb-quantum
groups with projection.  Since we do not use this construction in the
rest of the article, we do not recall the notation and setup
from~\cite{Meyer-Roy-Woronowicz:Twisted_tensor_2}.  Readers unfamiliar
with the semidirect product construction
in~\cite{Meyer-Roy-Woronowicz:Twisted_tensor_2} may skip this
section.

Let \(\G=(A,\Comult[A])\)
be a \(\Cst\)\nb-quantum
group.  Let~\((B,\beta,\hat\beta)\)
be an \(A\)\nb-Yetter--Drinfeld
algebra, that is, \(\beta\colon B\to B\otimes A\)
and \(\hat\beta\colon B\to B\otimes \hat{A}\)
are continuous coactions of \(A\)
and~\(\hat{A}\)
that satisfy the compatibility condition in
\cite{Meyer-Roy-Woronowicz:Twisted_tensor_2}*{Definition 5.11}.  The
twisted tensor product \(B\boxtimes B = B\boxtimes_{\multunit} B\)
is defined in \cite{Meyer-Roy-Woronowicz:Twisted_tensor_2}.  We also
require a coassociative comultiplication
\(\Comult[B]\colon B\to B\boxtimes B\).
Then \cite{Meyer-Roy-Woronowicz:Twisted_tensor_2}*{Theorem 6.8}
describes a coassociative comultiplication~\(\Comult[C]\)
on \(C\defeq A\boxtimes B\)
and shows that the \(\Cst\)\nb-bialgebra
\(\G[H]= (C,\Comult[C])\)
is bisimplifiable if~\((B,\Comult[B])\)
is bisimplifiable.  Furthermore, \(\Comult[C]\)
is injective if and only if~\(\Comult[B]\)
is injective.  It is not studied
in~\cite{Meyer-Roy-Woronowicz:Twisted_tensor_2} when
\((C,\Comult[C])\)
is a \(\Cst\)\nb-quantum
group: by our definition, this would require a multiplicative unitary
that generates it.  If~\(C\)
is unital, then this automatically exists and we are dealing with a
compact quantum group.  In the non-compact case, we need some sort of
multiplicative unitary for~\(B\) to get one for~\(C\).

For now, we disregard this issue.  We want to describe a projection
on~\(\G[H]\)
with image~\(\G\),
and the description of projections in
Proposition~\ref{pro:projection_via_action} makes sense in our
situation.  Thus we are going to define morphisms
\[
i\colon A\to C,\qquad
\Delta_R\colon C\to C\otimes A,\qquad
\Delta_L\colon C\to A\otimes C
\]
with the properties listed in
Proposition~\ref{pro:projection_via_action}.  If we know for some
reason that~\(\G[H]\)
is a \(\Cst\)\nb-quantum
group, that is, comes from a manageable multiplicative unitary, then
\((i,\Delta_L,\Delta_R)\)
as in Proposition~\ref{pro:projection_via_action} give a projection
on~\(\G[H]\)
with image~\(\G\).
Actually, we only need either \(\Delta_L\)
or~\(\Delta_R\)
for this purpose.  We provide both, however, and check all conditions in
Proposition~\ref{pro:projection_via_action}.

The morphism \(i\colon A\to A\boxtimes B = C\) is the canonical
embedding from the twisted tensor product, which is denoted~\(j_1\)
or~\(\iota_A\) in~\cite{Meyer-Roy-Woronowicz:Twisted_tensor_2}.  The
right coaction \(\Delta_R\colon C\to C\otimes A\) is the one
constructed in \cite{Meyer-Roy-Woronowicz:Twisted_tensor_2}*{Lemma
  6.5}.  It is the unique one for which the embeddings
\(i=\iota_A\colon A\to C\) and \(\iota_B\colon B\to C\) are
equivariant; that is,
\[
\Delta_R(\iota_A(a)\cdot \iota_B(b))
= (\iota_A\otimes\Id_A)(\Comult[A](a))
\cdot (\iota_B\otimes\Id_A)(\beta(b)).
\]
To construct~\(\Delta_L\), we equip \(A\otimes A\) with the right
\(A\)\nb-coaction \(\Id_A\otimes\Comult[A]\) on the second tensor
factor; this is a continuous \(A\)\nb-coaction, and
\(\Comult[A]\colon A\to A\otimes A\) is an \(A\)\nb-equivariant
morphism.  Therefore, there is an \(A\)\nb-equivariant morphism
\(\Comult[A]\boxtimes \Id_B\colon A\boxtimes B \to (A\otimes
A)\boxtimes B\).  We let~\(\Delta_L\) be the composite of
\(\Comult[A]\boxtimes \Id_B\) with the isomorphism \((A\otimes
A)\boxtimes B \cong A\otimes (A\boxtimes B) = A\otimes C\) from
\cite{Meyer-Roy-Woronowicz:Twisted_tensor_2}*{Lemma 3.14}.
We may also rewrite
\[
A\boxtimes_{\multunit} B \cong B\boxtimes_{\Dumultunit} A
\]
by \cite{Meyer-Roy-Woronowicz:Twisted_tensor_2}*{Proposition 5.1}.
This is exactly the reduced crossed product for the
\(\hat{A}\)\nb-coaction
on~\(B\)
by \cite{Meyer-Roy-Woronowicz:Twisted_tensor_2}*{Section 6.3}.  After
this identification, \(\Delta_L\)
becomes the dual coaction on the reduced crossed product as described
in \cite{Meyer-Roy-Woronowicz:Twisted_tensor_2}*{Section 6.3}.

\begin{proposition}
  \label{pro:projection_on_semidirect_product}
  The morphisms \(i\), \(\Delta_R\) and~\(\Delta_L\) constructed
  above make all the diagrams in
  Proposition~\textup{\ref{pro:projection_via_action}} commute.
\end{proposition}

\begin{proof}
  Of the ten diagrams in
  Proposition~\ref{pro:projection_via_action}, the last one commutes
  automatically if the others do, and the first and fifth one
  are the same.  So we have to check eight commuting diagrams.  The
  maps \(\Comult[C]\), \(\Delta_R\) and~\(\Delta_L\) are defined to
  have certain composites with \(\iota_A\) and~\(\iota_B\):
  \begin{alignat*}{2}
    \Comult[C]\circ\iota_A &= (\iota_A\otimes\iota_A)\Comult[A],&\qquad
    \Comult[C]\circ\iota_B &= \Psi_{23}\circ\Comult[B],\\
    \Delta_R  \circ\iota_A &= (\iota_A\otimes\Id_A)  \Comult[A],&\qquad
    \Delta_R  \circ\iota_B &= (\iota_B\otimes\Id_A)  \beta,\\
    \Delta_L  \circ\iota_A &= (\Id_A  \otimes\iota_A)\Comult[A],&\quad
    \Delta_L  \circ\iota_B &= 1_A \otimes\iota_B,
  \end{alignat*}
  where \(\Psi_{23}\colon B\boxtimes B\to C\otimes C\) is the
  restriction of the map~\(\Psi\) in
  \cite{Meyer-Roy-Woronowicz:Twisted_tensor_2}*{Proposition 6.6} to
  the second two legs; that is, \(\Psi_{23}j_1(b) =
  (\iota_B\otimes\iota_A)\beta(b)\) and \(\Psi_{23}j_2(b) =
  (1\otimes\iota_B)(b)\) for all \(b\in B\).

  In particular, \((\iota_A\otimes\iota_A)\circ \Comult[A] =
  \Comult[C]\circ\iota_A\) says that the first and fifth diagram
  commute, \(\Delta_R \circ\iota_A = (\iota_A\otimes\Id_A)
  \Comult[A]\) says that the fourth diagram commutes, and \(\Delta_L
  \circ\iota_A = (\Id_A \otimes\iota_A)\Comult[A]\) says that the
  eighth diagram commutes.

  The remaining diagrams in
  Proposition~\ref{pro:projection_via_action} involve equalities of
  two maps defined on~\(C\).  Two maps \(f,f'\) defined on~\(C\) are
  equal if and only if \(f\circ\iota_A = f'\circ\iota_A\) and
  \(f\circ\iota_B = f'\circ\iota_B\).  For all remaining diagrams, it is
  trivial to check that they commute after composing
  with~\(\iota_A\) because of the explicit formulas above.  The
  third and seventh diagram do not involve~\(\Comult[C]\), so the
  composites with~\(\iota_B\) are also given explicitly, which makes
  them trivial to check; in fact, they say simply that \(\Delta_R\)
  and~\(\Delta_L\) are a right and a left coaction, respectively, which
  is already checked in~\cite{Meyer-Roy-Woronowicz:Twisted_tensor_2}.

  The condition on~\(B\) for the sixth diagram is also trivial
  because~\(\Delta_L\) only does something complicated
  on~\(\iota_A(A)\) and \(\Comult[C]\) maps~\(\iota_B(B)\) into
  \(\iota_B(B)\otimes C\).

  For the second diagram, we must check
  \((\Comult[C]\otimes\Id_A)\Delta_R\iota_B =
  (\Id_C\otimes\Delta_R)\Comult[C]\iota_B\).  Since~\(\Comult[B]\)
  is \(A\)\nb-equivariant, \((\Comult[B]\otimes \Id_A)\circ \beta =
  (\beta\bowtie\beta)\circ\Comult[B]\).  Using the definition
  of~\(\Comult[C]\), we may rewrite our goal as
  \((\Psi_{23}\otimes\Id_A)(\beta\bowtie\beta)\Comult[B] =
  (\Id_C\otimes\Delta_R)\Psi_{23}\Comult[B]\).  From this, we may
  cancel the factor~\(\Comult[B]\), so it suffices to check that
  \[
  (\Psi_{23}\otimes\Id_A)(\beta\bowtie\beta)
  = (\Id_C\otimes\Delta_R)\Psi_{23}.
  \]
  This is an equality of maps \(B\boxtimes B\to C\otimes C\otimes A\),
  which we may check on both legs separately.  On the first leg, this
  reduces to the condition
  \((\Id_B\otimes\Comult[A])\beta= (\beta\otimes\Id_A)\beta\)
  that says that~\(\beta\)
  is a coaction, and on the second leg this is trivial.  This finishes
  the proof that the second diagram commutes

  In the condition from the ninth diagram on~\(B\), we may cancel
  the factor~\(\Comult[B]\) from~\(\Comult[C]\), so it suffices to
  check that \((\Id_C\otimes\Delta_L)\Psi_{23} =
  (\Delta_R\otimes\Id_C)\Psi_{23}\) as maps \(B\boxtimes B\to
  C\otimes A\otimes C\).  This is once again checked separately on
  the two factors~\(B\).  So we must check that the maps
  \(\Id_C\otimes\Delta_L\) and \(\Delta_R\otimes\Id_C\) take the
  same values both on \((\iota_B\otimes\iota_A)\beta(b)\) and on
  \(1\otimes\iota_B(b)\) for all \(b\in B\).  This reduces to the
  coaction condition for~\(\beta\) on
  \((\iota_B\otimes\iota_A)\beta(b)\) and is trivial on
  \(1\otimes\iota_B(b)\).
\end{proof}

\section{Braided Multiplicative Unitaries}
\label{sec:braided_mu}

The definition of a braided multiplicative unitary is as complicated
as the definition of a braided \(\Cst\)\nb-quantum
group.  Recall that the latter is relative to a \(\Cst\)\nb-quantum
group \(\G=(A,\Comult[A])\)
which generates the braiding.  The underlying
\(\Cst\)\nb-algebra~\(B\)
of a braided \(\Cst\)\nb-quantum
group carries continuous coactions \(\beta\)
and~\(\hat{\beta}\)
of \(\G\)
and~\(\DuG\),
respectively, which satisfy the Yetter--Drinfeld compatibility
condition which characterises coactions of the quantum codouble
of~\(\G\).
Finally, there is the comultiplication
\(\Comult[B]\colon B\to B\boxtimes B\),
which is equivariant with respect to \(\beta\)
and~\(\hat{\beta}\)
and coassociative.  Thus a braided \(\Cst\)\nb-quantum
group contains four coactions or comultiplications \(\Comult[A]\),
\(\beta\),
\(\hat{\beta}\),
\(\Comult[B]\), which must satisfy seven algebraic conditions:
\begin{enumerate}
\item \(\Comult[A]\) is coassociative;
\item \(\beta\) is a coaction of \((A,\Comult[A])\);
\item \(\hat{\beta}\) is a coaction of \((\hat{A},\DuComult[A])\);
\item \(\beta\) and~\(\hat{\beta}\) satisfy the Drinfeld commutation
  relation, so that they give a coaction of the quantum codouble;
\item \(\Comult[B]\) is equivariant with respect to the
  coaction~\(\beta\);
\item \(\Comult[B]\) is equivariant with respect to the
  coaction~\(\hat{\beta}\);
\item \(\Comult[B]\) is coassociative.
\end{enumerate}
The tensor product~\(\boxtimes\) is not symmetric unless~\(\G\) is
trivial.  Thus \(X \boxtimes' Y \defeq Y\boxtimes X\) gives another
equally reasonable tensor product.  We may also consider braided
quantum groups where the comultiplication takes values in
\(B\boxtimes' B\) instead of \(B\boxtimes B\).  Actually, these
\(\Cst\)\nb-algebras are canonically isomorphic through the flip map,
which interchanges the two factors~\(B\).  Thus there are two kinds of
braided \(\Cst\)-quantum group, and taking the ``coopposite,'' that
is, composing~\(\Comult[B]\) with the flip map~\(\Sigma\) and leaving
everything else the same, gives a bijection between the two types.

\begin{remark}
  \label{rem:quasitriangular_simplify}
  The definition above simplifies somewhat if~\(\G\) is
  quasitriangular.  Then a corepresentation~\(\beta\) determines a
  corepresentation~\(\hat\beta\) so as to form a coaction of the
  quantum codouble.  Since~\(\hat\beta\) is a coaction constructed
  naturally from~\(\beta\), the conditions (3), (4) and~(6) above are
  redundant.  A similar simplification occurs for braided
  multiplicative unitaries.  Since we are concerned with the general
  theory here, we do not explore this situation any further.
\end{remark}

When we turn to multiplicative unitaries, we replace
\(\Cst\)\nb-algebras
by Hilbert spaces on which they act faithfully; comultiplications and
coactions are replaced by unitaries on appropriate tensor product
Hilbert spaces that implement the coactions through conjugation.  So
to specify a braided multiplicative unitary, we need two Hilbert
spaces and four unitaries that satisfy seven conditions, which
correspond to the seven conditions for the comultiplications and
coactions listed above.  Moreover, there are two slightly different
kinds of braided multiplicative unitaries, depending on whether we use
the ``standard'' braiding or its opposite; which braiding
is standard and which is opposite is, of course, a mere convention.
The following
definition contains the details:

\begin{definition}
  \label{def:braided_multiplicative_unitary}
  Let \(\Hils\) and~\(\Hils[L]\) be Hilbert spaces and let
  \(\Multunit \in \U(\Hils\otimes\Hils)\) be a manageable
  \emph{multiplicative unitary}; in particular, \(\Multunit\)
  satisfies the \emph{pentagon equation}
  \begin{equation}
    \label{eq:pentagon}
    \Multunit_{23}\Multunit_{12}
    = \Multunit_{12}\Multunit_{13}\Multunit_{23}.
  \end{equation}
  A \emph{top-braided multiplicative unitary on~\(\Hils[L]\) relative
    to~\(\Multunit\)} is given by unitaries
  \[
  \Corep{U}\in\U(\Hils[L]\otimes\Hils),\qquad
  \DuCorep{V}\in\U(\Hils\otimes\Hils[L]),\qquad
  \BrMultunit\in\U(\Hils[L]\otimes\Hils[L])
  \]
  which satisfy the following conditions:
  \begin{itemize}
  \item \(\Corep{U}\) is a \emph{right corepresentation}
    of~\(\Multunit\):
    \begin{equation}
      \label{eq:U_corep}
      \Multunit_{23} \Corep{U}_{12}
      = \Corep{U}_{12} \Corep{U}_{13} \Multunit_{23}
      \quad\text{in }\U(\Hils[L]\otimes\Hils\otimes\Hils);
    \end{equation}
  \item \(\DuCorep{V}\) is a \emph{left corepresentation}
    of~\(\Multunit\):
    \begin{equation}
      \label{eq:V_corep}
      \DuCorep{V}_{23} \Multunit_{12}
      = \Multunit_{12} \DuCorep{V}_{13} \DuCorep{V}_{23}
      \quad\text{in }\U(\Hils\otimes\Hils\otimes\Hils[L]);
    \end{equation}
  \item the corepresentations \(\Corep{U}\) and~\(\DuCorep{V}\) are
    \emph{Drinfeld compatible}:
    \begin{equation}
      \label{eq:U_V_compatible}
      \Corep{U}_{23} \Multunit_{13} \DuCorep{V}_{12}
      = \DuCorep{V}_{12} \Multunit_{13} \Corep{U}_{23}
      \quad\text{in }\U(\Hils\otimes\Hils[L]\otimes\Hils);
    \end{equation}
  \item \(\BrMultunit\) is \emph{invariant} with respect to the
    right corepresentation \(\Corep{U} \tenscorep \Corep{U} \defeq
    \Corep{U}_{13}\Corep{U}_{23}\) of~\(\Multunit\)
    on~\(\Hils[L]\otimes\Hils[L]\):
    \begin{equation}
      \label{eq:F_U-invariant}
      \Corep{U}_{13} \Corep{U}_{23} \BrMultunit_{12}
      = \BrMultunit_{12} \Corep{U}_{13} \Corep{U}_{23}
      \quad\text{in }\U(\Hils[L]\otimes\Hils[L]\otimes \Hils);
    \end{equation}
  \item \(\BrMultunit\) is \emph{invariant} with respect to the left
    corepresentation \(\DuCorep{V} \tenscorep \DuCorep{V} \defeq
    \DuCorep{V}_{13}\DuCorep{V}_{12}\) of~\(\Multunit\)
    on~\(\Hils[L]\otimes\Hils[L]\):
    \begin{equation}
      \label{eq:F_V-invariant}
      \DuCorep{V}_{13} \DuCorep{V}_{12} \BrMultunit_{23}
      = \BrMultunit_{23} \DuCorep{V}_{13} \DuCorep{V}_{12}
      \quad\text{in }\U(\Hils\otimes\Hils[L]\otimes\Hils[L]);
    \end{equation}
  \item \(\BrMultunit\) satisfies the \emph{top-braided pentagon
      equation}
    \begin{equation}
      \label{eq:top-braided_pentagon}
      \BrMultunit_{23} \BrMultunit_{12}
      = \BrMultunit_{12} (\Braiding{\Hils[L]}{\Hils[L]})_{23}
      \BrMultunit_{12} (\Dualbraiding{\Hils[L]}{\Hils[L]})_{23}
      \BrMultunit_{23}
      \quad\text{in }\U(\Hils[L]\otimes\Hils[L]\otimes\Hils[L]);
    \end{equation}
    here the braiding \(\Braiding{\Hils[L]}{\Hils[L]} \in
    \U(\Hils[L]\otimes \Hils[L])\) and
    \(\Dualbraiding{\Hils[L]}{\Hils[L]} =
    (\Braiding{\Hils[L]}{\Hils[L]})^*\) are defined as
    \(\Braiding{\Hils[L]}{\Hils[L]} = Z \Flip\) for the
    flip~\(\Flip\), \(x\otimes y\mapsto y\otimes x\), and the unique
    unitary \(Z \in \U(\Hils[L]\otimes \Hils[L])\) that satisfies
    \begin{equation}
      \label{eq:braiding}
      Z_{13} = \DuCorep{V}_{23} \Corep{U}_{12}^*
      \DuCorep{V}_{23}^* \Corep{U}_{12}
      \quad\text{in }\U(\Hils[L]\otimes\Hils\otimes\Hils[L]).
    \end{equation}
  \end{itemize}
  A \emph{bottom-braided multiplicative unitary on~\(\Hils[L]\)
    relative to~\(\Multunit\)} is given by the same unitaries
  \(\Corep{U}\), \(\DuCorep{V}\), \(\BrMultunit\) satisfying
  \eqref{eq:U_corep}--\eqref{eq:F_V-invariant} and the
  \emph{bottom-braided pentagon equation}
  \begin{equation}
    \label{eq:bottom-braided_pentagon}
    \BrMultunit_{23} \BrMultunit_{12}
    = \BrMultunit_{12} (\Dualbraiding{\Hils[L]}{\Hils[L]})_{23}
    \BrMultunit_{12} (\Braiding{\Hils[L]}{\Hils[L]})_{23}
    \BrMultunit_{23}
    \quad\text{in }\U(\Hils[L]\otimes\Hils[L]\otimes\Hils[L]).
  \end{equation}
\end{definition}

Two corepresentations \(\Corep{U}\) and~\(\DuCorep{V}\) on a Hilbert
space~\(\Hils[L]\) satisfying~\eqref{eq:U_V_compatible} are equivalent
to a corepresentation of the quantum codouble of the quantum group
associated to~\(\Multunit\).  It is shown
in~\cite{Meyer-Roy-Woronowicz:Twisted_tensor_2} that these
corepresentations form a braided monoidal category.  Our conventions
differ from those in~\cite{Meyer-Roy-Woronowicz:Twisted_tensor_2}
because we use a left corepresentation~\(\DuCorep{V}\) instead of the
corresponding right corepresentation \(\Corep{V} \defeq \Sigma
\DuCorep{V}^* \Sigma\).  The compatibility
condition~\eqref{eq:U_V_compatible} and the definition of the braiding
operator above are equivalent to those
in~\cite{Meyer-Roy-Woronowicz:Twisted_tensor_2} up to this change of
notation.  The operator~\(Z\) in~\eqref{eq:braiding} exists
because~\(\Multunit\) is manageable.  It is shown
in~\cite{Meyer-Roy-Woronowicz:Twisted_tensor_2} that the
operators~\(\Braiding{\Hils[L]_1}{\Hils[L]_2}\) defined as above form
a braiding on the tensor category of triples
\((\Hils[L],\Corep{U},\DuCorep{V})\); the
operators~\(\Dualbraiding{\Hils[L]_1}{\Hils[L]_2}\) give the opposite
braiding.

In a braided monoidal category, the leg numbering notation should use
the braiding operators.  This explains why we
replace~\(\BrMultunit_{13}\) by \((\Braiding{\Hils[L]}{\Hils[L]})_{23}
\BrMultunit_{12} (\Dualbraiding{\Hils[L]}{\Hils[L]})_{23}\) or
\((\Dualbraiding{\Hils[L]}{\Hils[L]})_{23} \BrMultunit_{12}
(\Braiding{\Hils[L]}{\Hils[L]})_{23}\) in the two braided pentagon
equations \eqref{eq:top-braided_pentagon}
and~\eqref{eq:bottom-braided_pentagon}.  We should also have
replaced~\(\BrMultunit_{23}\) by
\(\Braiding{\Hils[L]}{\Hils[L]\otimes \Hils[L]} \BrMultunit_{12}
\Dualbraiding{\Hils[L]}{\Hils[L] \otimes \Hils[L]}\); the braiding
operator~\(\Braiding{\Hils[L]}{\Hils[L]\otimes \Hils[L]}\) is defined
as \(Z' \Sigma^{\Hils[L],\Hils[L]\otimes\Hils[L]}\), where~\(Z'\) is
the unique operator on \((\Hils[L]\otimes\Hils[L])\otimes\Hils[L]\)
with
\[
Z'_{134} = (\DuCorep{V} \tenscorep \DuCorep{V})_{234} \Corep{U}_{12}^*
(\DuCorep{V} \tenscorep \DuCorep{V})_{234} ^* \Corep{U}_{12}
\quad\text{in }
\U(\Hils[L]\otimes\Hils\otimes\Hils[L]\otimes\Hils[L]).
\]
Since we are dealing with a braided monoidal category, we also have
\[
\Braiding{\Hils[L]}{\Hils[L]\otimes \Hils[L]}
= \Braiding{\Hils[L]}{\Hils[L]}_{23}
\Braiding{\Hils[L]}{\Hils[L]}_{12},\qquad
\Braiding{\Hils[L]\otimes \Hils[L]}{\Hils[L]}
= \Braiding{\Hils[L]}{\Hils[L]}_{12}
\Braiding{\Hils[L]}{\Hils[L]}_{23}.
\]

Since~\(\BrMultunit\) is invariant with respect to both
corepresentations, it commutes with any operator that is constructed
in a natural way out of them, such as~\(Z'\).  This implies
\[
\BrMultunit_{23}
= \Braiding{\Hils[L]}{\Hils[L]\otimes \Hils[L]}
\BrMultunit_{12} \Dualbraiding{\Hils[L]}{\Hils[L] \otimes \Hils[L]}
= \Dualbraiding{\Hils[L]}{\Hils[L]\otimes \Hils[L]}
\BrMultunit_{12} \Braiding{\Hils[L]}{\Hils[L] \otimes \Hils[L]},
\]
so here the braiding has no effect.  This also implies
\[
\Braiding{\Hils[L]}{\Hils[L]}_{23} \BrMultunit_{12}
\Dualbraiding{\Hils[L]}{\Hils[L]}_{23} =
\Dualbraiding{\Hils[L]}{\Hils[L]}_{12} \BrMultunit_{23}
\Braiding{\Hils[L]}{\Hils[L]}_{12},\qquad
\Dualbraiding{\Hils[L]}{\Hils[L]}_{23} \BrMultunit_{12}
\Braiding{\Hils[L]}{\Hils[L]}_{23} =
\Braiding{\Hils[L]}{\Hils[L]}_{12} \BrMultunit_{23}
\Dualbraiding{\Hils[L]}{\Hils[L]}_{12}.
\]
Such equations are easier to digest as pictures:
\[
\begin{tikzpicture}[baseline=(current bounding box.west),scale=.4]
  \draw (0,4)--(0,3.1);
  \draw (1,4) to[out=315, in=90] (2,2.5) to[out=270,, in=45] (1,1);
  \draw[overar] (2,4)--(1,3.1);
  \draw[overar] (1,1.9)--(2,1);
  \draw (0,1.9)--(0,1);
  \draw (-.1,3.1)--(1.1,3.1)--(1.1,1.9)--(-.1,1.9)--(-.1,3.1);
  \draw (0.5,2.4) node {$\BrMultunit$};
  \draw (1.5,0) node {$\Braiding{\Hils[L]}{\Hils[L]}_{23} \BrMultunit_{12} \Dualbraiding{\Hils[L]}{\Hils[L]}_{23}$};
\end{tikzpicture}
=
\begin{tikzpicture}[baseline=(current bounding box.west),scale=.4]
  \draw (1,4) to[out=225, in=90] (0,2.5) to[out=270,, in=135] (1,1);
  \draw (2,4)--(2,3.1);
  \draw[overar] (1,1.9)--(0,1);
  \draw[overar] (0,4)--(1,3.1);
  \draw (2,1.9)--(2,1);
  \draw (.9,3.1)--(2.1,3.1)--(2.1,1.9)--(.9,1.9)--(.9,3.1);
  \draw (1.5,2.4) node {$\BrMultunit$};
  \draw (1.5,0) node {$\Dualbraiding{\Hils[L]}{\Hils[L]}_{12} \BrMultunit_{23} \Braiding{\Hils[L]}{\Hils[L]}_{12}$};
\end{tikzpicture}
\qquad
\begin{tikzpicture}[baseline=(current bounding box.west),scale=.4]
  \draw (0,4)--(0,3.1);
  \draw (2,4)--(1,3.1);
  \draw (1,1.9)--(2,1);
  \draw (0,1.9)--(0,1);
  \draw[overar] (1,4) to[out=315, in=90] (2,2.5) to[out=270,, in=45] (1,1);
  \draw (-.1,3.1)--(1.1,3.1)--(1.1,1.9)--(-.1,1.9)--(-.1,3.1);
  \draw (0.5,2.4) node {$\BrMultunit$};
  \draw (1.5,0) node {$\Dualbraiding{\Hils[L]}{\Hils[L]}_{23} \BrMultunit_{12} \Braiding{\Hils[L]}{\Hils[L]}_{23}$};
\end{tikzpicture}
=
\begin{tikzpicture}[baseline=(current bounding box.west),scale=.4]
  \draw (2,4)--(2,3.1);
  \draw (1,1.9)--(0,1);
  \draw (0,4)--(1,3.1);
  \draw (2,1.9)--(2,1);
  \draw[overar] (1,4) to[out=225, in=90] (0,2.5) to[out=270,, in=135] (1,1);
  \draw (.9,3.1)--(2.1,3.1)--(2.1,1.9)--(.9,1.9)--(.9,3.1);
  \draw (1.5,2.4) node {$\BrMultunit$};
  \draw (1.5,0) node {$\Braiding{\Hils[L]}{\Hils[L]}_{12} \BrMultunit_{23} \Dualbraiding{\Hils[L]}{\Hils[L]}_{12}$};
\end{tikzpicture}
\]
The top-braided pentagon equation~\eqref{eq:top-braided_pentagon}
uses the version of~\(\BrMultunit_{13}\) where~\(\BrMultunit\) acts
on the two top strands, whereas the bottom-braided pentagon
equation~\eqref{eq:bottom-braided_pentagon} uses the version
of~\(\BrMultunit_{13}\) where~\(\BrMultunit\) acts on the two bottom
strands; this explains our notation.

The braided pentagon equation is the usual pentagon equation if and
only if~\(\BrMultunit\) commutes with \(\Sigma Z\Sigma\).  Sufficient
conditions for this are \(Z=1\), \(\Corep{U}=1\) or \(\DuCorep{V}=1\).

From now on, we restrict attention to top-braided multiplicative
unitaries, so \emph{braided multiplicative unitary} means
\emph{top-braided multiplicative unitary}.

\begin{definition}
  \label{def:dual_brmult}
  The \emph{dual} of a braided multiplicative unitary
  \((\Corep{U},\DuCorep{V},\BrMultunit)\) over~\(\Multunit\) is
  \((\Corep{V},\DuCorep{U},\DuBrMultunit)\) over~\(\DuMultunit\),
  where \(\DuMultunit\defeq \Sigma \Multunit[*] \Sigma\), \(\Corep{V}
  \defeq \Sigma \DuCorep{V}^* \Sigma\), \(\DuCorep{U} \defeq \Sigma
  \Corep{U}^* \Sigma\), and
  \[
  \DuBrMultunit\defeq \Dualbraiding{\Hils[L]}{\Hils[L]}
  \BrMultunit^*\Braiding{\Hils[L]}{\Hils[L]} \in
  \U(\Hils[L]\otimes\Hils[L]).
  \]
\end{definition}

The braiding operator~\(\Braiding{\Hils[L]}{\Hils[L]}\) for
\((\DuMultunit,\Corep{V},\DuCorep{U})\) is the opposite
braiding~\(\Dualbraiding{\Hils[L]}{\Hils[L]}\) for
\((\Multunit,\Corep{U},\DuCorep{V})\).  Therefore, the dual of the
dual is the braided multiplicative unitary that we started with,
even if the braiding is not symmetric.

\begin{proposition}
  \label{pro:dual_brmult}
  Let \((\Corep{U},\DuCorep{V},\BrMultunit)\) be a top-braided
  multiplicative unitary over~\(\Multunit\).  Its dual
  \((\Corep{V},\DuCorep{U},\DuBrMultunit)\) is a top-braided
  multiplicative unitary over \(\DuMultunit \defeq \Sigma
  \Multunit[*] \Sigma\).
\end{proposition}

\begin{proof}
  It is well-known that the dual~\(\DuMultunit\) is again a
  multiplicative unitary, that~\(\Corep{U}\) is a right
  corepresentation of~\(\Multunit\) if and only if~\(\DuCorep{U}\)
  is a left corepresentation of~\(\DuMultunit\), and
  that~\(\DuCorep{V}\) is a left corepresentation of~\(\Multunit\)
  if and only if~\(\Corep{V}\) is a right corepresentation
  of~\(\DuMultunit\).  Routine computations show that the Drinfeld
  compatibility condition and the invariance conditions are also
  preserved by the duality.  The top-braided (or bottom-braided)
  pentagon equation for the dual is equivalent to the top-braided
  (or bottom-braided) pentagon equation for the original braided
  multiplicative unitary because the duality replaces the braiding
  by the opposite braiding.
\end{proof}

Now we define when a braided multiplicative unitary
\((\Multunit,\Corep{U},\DuCorep{V},\BrMultunit)\)
is manageable.  This requires~\(\Multunit\)
to be manageable, that is, there are a strictly positive
operator~\(Q\)
on~\(\Hils\)
and a unitary
\(\widetilde{\Multunit}\in \U(\conj{\Hils}\otimes\Hils)\)
with \(\Multunit[*](Q\otimes Q)\Multunit = Q\otimes Q\) and
\begin{equation}
  \label{eq:Multunit_manageable}
  \bigl(x\otimes u\mid \Multunit\mid z\otimes y\bigr)
  = \bigl(\conj{z}\otimes Q u\mid \widetilde{\Multunit} \mid
  \conj{x}\otimes Q^{-1}y\bigr)
\end{equation}
for all \(x,z\in\Hils\),
\(u\in\dom(Q)\)
and \(y\in\dom(Q^{-1})\)
(see \cite{Woronowicz:Mult_unit_to_Qgrp}*{Definition 1.2}).
Here~\(\conj{\Hils}\)
is the conjugate Hilbert space, and an operator is \emph{strictly
  positive} if it is positive and self-adjoint with trivial kernel.
The condition \(\Multunit[*](Q\otimes Q)\Multunit = Q\otimes Q\)
means that the unitary~\(\Multunit\)
commutes with the unbounded operator~\(Q\otimes Q\).

\begin{definition}
  \label{def:braided_manageable}
  Let \(\Multunit\in\U(\Hils\otimes\Hils)\) be a manageable
  multiplicative unitary and let \(Z\) and~\(Q\) be as above.  A
  braided multiplicative unitary
  \((\Corep{U},\DuCorep{V},\BrMultunit)\) over~\(\Multunit\) is
  \emph{manageable} if there are a strictly positive
  operator~\(Q_{\Hils[L]}\) on~\(\Hils[L]\) and a unitary
  \(\widetilde{\BrMultunit} \widetilde{Z}{}^*\in
  \U(\conj{\Hils[L]}\otimes\Hils[L])\) such that
  \begin{align}
    \label{eq:br_manag_commute_U}
    \Corep{U}(Q_{\Hils[L]}\otimes Q)\Corep{U}^* &= Q_{\Hils[L]}\otimes Q,\\
    \label{eq:br_manag_commute_V}
    \DuCorep{V}(Q\otimes Q_{\Hils[L]})\DuCorep{V}^* &= Q\otimes Q_{\Hils[L]},\\
    \label{eq:br_manag_commute_F}
    \BrMultunit(Q_{\Hils[L]}\otimes Q_{\Hils[L]})\BrMultunit^* &= Q_{\Hils[L]}\otimes Q_{\Hils[L]},\\
    \label{eq:br_manag}
    (x\otimes u\mid Z^* \BrMultunit \mid y\otimes v) &=
    (\conj{y}\otimes Q_{\Hils[L]}(u) \mid \widetilde{\BrMultunit} \widetilde{Z}{}^*
    \mid \conj{x}\otimes Q_{\Hils[L]}^{-1}(v))
  \end{align}
  for all \(x,y\in\Hils[L]\), \(u\in\dom(Q_{\Hils[L]})\) and
  \(v\in\dom(Q_{\Hils[L]}^{-1})\).
\end{definition}

We have written~\(\widetilde{\BrMultunit} \widetilde{Z}{}^*\) and
not~\(\widetilde{\BrMultunit}\) in~\eqref{eq:br_manag} to
make the formula more symmetric and to clarify the manageability of
the dual of a braided multiplicative unitary.

We now describe the operator~\(\widetilde{Z}\) that we want to use.  The
corepresentation~\(\Corep{U}\) of~\(\Multunit\) on~\(\Hils[L]\)
induces a contragradient corepresentation on~\(\conj{\Hils[L]}\).
This is of the form~\(\widetilde{\Corep{U}}{}^*\), where
\(\widetilde{\Corep{U}}\in \U(\conj{\Hils[L]}\otimes\Hils)\)
satisfies a variant of~\eqref{eq:Multunit_manageable}, see
\cite{Woronowicz:Mult_unit_to_Qgrp}*{Theorem 1.6} and
\cite{Soltan-Woronowicz:Multiplicative_unitaries}*{Proposition 10}.
Since~\(\widetilde{\Corep{U}}{}^*\) is a right corepresentation
of~\(\Multunit\) on~\(\conj{\Hils[L]}\), there is a unique unitary
\(\widetilde{Z}\in\U(\conj{\Hils[L]}\otimes \Hils[L])\) that
satisfies
\begin{equation}
  \label{eq:braiding-manag}
  \widetilde{Z}_{13} = \DuCorep{V}_{23}\widetilde{\Corep{U}}_{12}
  \DuCorep{V}{}_{23}^*\widetilde{\Corep{U}}{}_{12}^*
  \qquad
  \text{in }\U(\conj{\Hils[L]}\otimes\Hils\otimes\Hils[L]).
\end{equation}
We use this unitary in~\eqref{eq:br_manag}.  Of course, it does not
matter which unitary~\(\widetilde{Z}\) we use because we may absorb it
in~\(\widetilde{\BrMultunit}\).

\begin{proposition}
  \label{prop:br_dual_manag}
  The dual of a manageable braided multiplicative unitary is again
  manageable.
\end{proposition}

\begin{proof}
  Let \((\Corep{U},\DuCorep{V},\BrMultunit)\) be a manageable
  top-braided multiplicative unitary over~\(\Multunit\), let \(Z\)
  and~\(\widetilde{Z}\) be as in \eqref{eq:braiding}
  and~\eqref{eq:braiding-manag}.  Let \(\widetilde{\Multunit},Q\)
  witness the manageability of~\(\Multunit\) and let
  \(\widetilde{\BrMultunit}\) and~\(Q_{\Hils[L]}\) witness the manageability
  of \((\Corep{U},\DuCorep{V},\BrMultunit)\).

  On \(\Hils[L]\otimes\Hils\otimes\Hils[L]\), both
  \(\Corep{U}_{12}\) and~\(\DuCorep{V}_{23}\) commute with
  \(Q_{\Hils[L]}\otimes Q\otimes Q_{\Hils[L]}\) by \eqref{eq:br_manag_commute_U}
  and~\eqref{eq:br_manag_commute_V}.  Hence so does~\(Z\)
  by~\eqref{eq:braiding}.  Thus
  \begin{equation}
    \label{eq:Z_comm_Q}
    Z(Q_{\Hils[L]}\otimes Q_{\Hils[L]})Z^* = Q_{\Hils[L]}\otimes Q_{\Hils[L]}.
  \end{equation}
  Together with~\eqref{eq:br_manag_commute_F}, this implies
  that~\(Z^* \BrMultunit\) commutes with~\(Q_{\Hils[L]}\otimes
  Q_{\Hils[L]}\).  This together with~\eqref{eq:br_manag} implies
  that~\(\widetilde{\BrMultunit} \widetilde{Z}{}^*\) commutes
  with~\(Q_{\Hils[L]}^\transpose\otimes Q_{\Hils[L]}^{-1}\), compare
  the proof of Lemma~\ref{lem:V_tilde_commute_Q} or
  \cite{Woronowicz:Mult_unit_to_Qgrp}*{Proposition 1.4.(1)}.

  The unitary~\(\widetilde{\Corep{U}}\) commutes
  with~\(Q_{\Hils[L]}^\transpose\otimes Q^{-1}\), compare
  Lemma~\ref{lem:V_tilde_commute_Q} or
  \cite{Woronowicz:Mult_unit_to_Qgrp}*{Proposition 1.4.(1)}.  This
  together with \eqref{eq:br_manag_commute_V}
  and~\eqref{eq:braiding-manag} implies
  \begin{equation}
    \label{eq:tilde-Z-W-comm_Q1}
    \widetilde{Z}(Q_{\Hils[L]}^\transpose\otimes Q_{\Hils[L]}^{-1})\widetilde{Z}{}^*
    = Q_{\Hils[L]}^\transpose\otimes Q_{\Hils[L]}^{-1},
  \end{equation}
  compare the proof of~\eqref{eq:Z_comm_Q}.  Hence
  \begin{equation}
    \label{eq:tilde-Z-W-comm_Q2}
    \widetilde{\BrMultunit}(Q_{\Hils[L]}^\transpose\otimes Q_{\Hils[L]}^{-1})
    \widetilde{\BrMultunit}{}^*
    = Q_{\Hils[L]}^\transpose\otimes Q_{\Hils[L]}^{-1}
  \end{equation}
  because~\(\widetilde{\BrMultunit} \widetilde{Z}{}^*\) commutes
  with~\(Q_{\Hils[L]}^\transpose\otimes Q_{\Hils[L]}^{-1}\) as well.

  If \(y\in\dom(Q_{\Hils[L]})\), \(x\in\dom(Q_{\Hils[L]}^{-1})\),
  and \(u,v\in\Hils[L]\), then
  \begin{equation}
    \label{eq:equiv_br_manag_mod}
    (x\otimes u\mid Z^* \BrMultunit \mid y\otimes v) =
    (\conj{Q_{\Hils[L]}(y)}\otimes u \mid \widetilde{\BrMultunit}\widetilde{Z}{}^*
    \mid \conj{Q_{\Hils[L]}^{-1}(x)}\otimes v);
  \end{equation}
  this is proved like
  \cite{Woronowicz:Mult_unit_to_Qgrp}*{Proposition 1.4 (2)}.  We
  rewrite this using the unitaries \(\widetilde{\hat{Z}},
  \widetilde{\DuBrMultunit}\in \U(\conj{\Hils[L]}\otimes\Hils[L])\)
  defined by
  \[
  \widetilde{\hat{Z}}\defeq
  \bigl(\Flip\widetilde{Z}{}^*\Flip\bigr)^{\transpose\otimes\transpose},
  \qquad
  \widetilde{\DuBrMultunit}\defeq
  \bigl(\Flip\widetilde{\BrMultunit}{}^*\Flip\bigr)^{\transpose\otimes\transpose}.
  \]
  By definition, \(\hat{Z}{}^* \DuBrMultunit = \Flip\BrMultunit{}^*
  Z\Flip\) and \(\widetilde{\DuBrMultunit}\widetilde{\hat{Z}}{}^* =
  (\Flip\widetilde{Z}\widetilde{\BrMultunit}{}^*
  \Flip)^{\transpose\otimes\transpose}\).
  Thus~\eqref{eq:equiv_br_manag_mod} gives
  \begin{align*}
    (x\otimes u\mid \hat{Z}{}^* \DuBrMultunit \mid y\otimes v) &=
    (x\otimes u\mid \Flip\BrMultunit{}^*Z\Flip \mid y\otimes v)\\
    &=\overline{ (y\otimes v\mid \Flip Z^*\BrMultunit\Flip \mid x\otimes u)}\\
    &=\overline{(v\otimes y\mid Z^*\BrMultunit \mid u\otimes x)}\\
    &=\overline{(\conj{Q_{\Hils[L]}(u)}\otimes y\mid
      \widetilde{\BrMultunit} \widetilde{Z}{}^*\mid
      \conj{Q_{\Hils[L]}^{-1}(v)} \otimes x)}\\
    &=(\conj{Q_{\Hils[L]}^{-1}(v)}\otimes x\mid
    \widetilde{Z}\widetilde{\BrMultunit}{}^*\mid
    \conj{Q_{\Hils[L]}(u)}\otimes y)\\
    &= (x\otimes\conj{Q_{\Hils[L]}^{-1}(v)}\mid
    \Flip\widetilde{Z}\widetilde{\BrMultunit}{}^*\Flip\mid
    y\otimes\conj{Q_{\Hils[L]}(u)})\\
    &= (\conj{y}\otimes Q_{\Hils[L]}(u)\mid
    (\Flip\widetilde{Z}\widetilde{\BrMultunit}{}^*\Flip)^{\transpose\otimes\transpose}
    \mid \conj{x}\otimes Q_{\Hils[L]}^{-1}(v))\\
    &= (\conj{y}\otimes Q_{\Hils[L]}(u) \mid
    \widetilde{\DuBrMultunit}\widetilde{\hat{Z}}{}^*
    \mid \conj{x}\otimes Q_{\Hils[L]}^{-1}(v)).
  \end{align*}
  Since the unitary~\(Z\) for the dual braided multiplicative
  unitary becomes~\(\hat{Z}\), the operators \(Q_{\Hils[L]}\)
  and~\(\widetilde{\DuBrMultunit}\) witness the manageability
  of~\(\DuBrMultunit\).
\end{proof}

\subsection{Semidirect product multiplicative unitaries}
\label{sec:semidirect_product_mu}

In this section, we construct a semidirect product multiplicative
unitary~\(\Multunit[C]\)
and a projection~\(\ProjBichar\)
out of a braided multiplicative unitary
\((\Corep{U},\DuCorep{V},\BrMultunit)\)
over a multiplicative unitary~\(\Multunit\).
We show that the semidirect product multiplicative
unitary~\(\Multunit[C]\)
is manageable if the braided multiplicative unitary
\((\Corep{U},\DuCorep{V},\BrMultunit)\) is manageable.

The formulas and proofs below are explicit but lengthy because all
four unitaries \(\Multunit\),
\(\Corep{U}\),
\(\DuCorep{V}\),
\(\BrMultunit\)
must enter in the definitions of \(\Multunit[C]\)
and~\(\ProjBichar\)
and all seven conditions on them must be used in the proofs.

\begin{theorem}
  \label{the:standard_mult_from_braided_and_standard}
  Let \((\Corep{U},\DuCorep{V},\BrMultunit)\) be a braided
  multiplicative unitary over a multiplicative unitary~\(\Multunit\).
  Define \(\Multunit[C]_{1234},\ProjBichar \in
  \U(\Hils\otimes\Hils[L]\otimes\Hils\otimes\Hils[L])\) by
  \begin{align}
    \label{eq:Stand_multunit_frm_brd}
    \Multunit[C]_{1234} &\defeq
    \Multunit_{13} \Corep{U}_{23} \DuCorep{V}^*_{34}
    \BrMultunit_{24} \DuCorep{V}_{34},\\
    \ProjBichar_{1234} &\defeq \Multunit_{13} \Corep{U}_{23}.
  \end{align}
  Then \(\Multunit[C]\) and~\(\ProjBichar\) satisfy the four
  pentagon-like equations in
  Proposition~\textup{\ref{pro:qg_projection_mu}}.  Thus they give a
  \(\Cst\)\nb-quantum group with projection when~\(\Multunit[C]\) is
  manageable.
\end{theorem}

\begin{proof}
  We first verify the pentagon equation~\eqref{eq:pentagon}
  for~\(\Multunit[C]_{1234}\).  Let
  \[
  XXX = \Multunit[C]_{3456} \Multunit[C]_{1234}
  (\Multunit[C])^*_{3456}.
  \]
  We will rewrite this in several steps using the conditions in
  Definition~\ref{def:braided_multiplicative_unitary}.  We use
  \(\{\ldots\}\) to highlight which part of the formula we are
  modifying in the following step.
  Definition~\eqref{eq:Stand_multunit_frm_brd} gives
  \[
  XXX
  = \Multunit_{35} \{\Corep{U}_{45} \DuCorep{V}^*_{56}
  \BrMultunit_{46} \DuCorep{V}_{56}\} \{\Multunit_{13}
  \Corep{U}_{23}\} \DuCorep{V}^*_{34} \BrMultunit_{24}
  \DuCorep{V}_{34} \DuCorep{V}^*_{56} \BrMultunit_{46}^*
  \DuCorep{V}_{56} \Corep{U}_{45}^* \Multunit[*]_{35}.
  \]
  Since \(\Corep{U}_{45} \DuCorep{V}^*_{56} \BrMultunit_{46}
  \DuCorep{V}_{56}\) and \(\Multunit_{13} \Corep{U}_{23}\) commute,
  \[
  XXX =
  \{\Multunit_{35} \Multunit_{13}\} \Corep{U}_{23} \Corep{U}_{45}
  \DuCorep{V}^*_{56} \BrMultunit_{46} \{\DuCorep{V}_{56}\}
  \DuCorep{V}^*_{34} \BrMultunit_{24} \DuCorep{V}_{34}
  \{\DuCorep{V}^*_{56}\} \BrMultunit_{46}^*
  \DuCorep{V}_{56} \Corep{U}_{45}^* \Multunit[*]_{35}.
  \]
  Now we use the pentagon equation~\eqref{eq:pentagon}
  for~\(\Multunit\)
  and commute \(\DuCorep{V}_{56}\)
  with \(\DuCorep{V}^*_{34} \BrMultunit_{24} \DuCorep{V}_{34}\):
  \[
  XXX =
  \Multunit_{13} \Multunit_{15} \{\Multunit_{35} \Corep{U}_{23}\}
  \Corep{U}_{45} \DuCorep{V}^*_{56}
  \{\BrMultunit_{46} \DuCorep{V}^*_{34}\} \BrMultunit_{24}
  \{\DuCorep{V}_{34} \BrMultunit_{46}^*\}
  \DuCorep{V}_{56} \Corep{U}_{45}^* \Multunit[*]_{35}.
  \]
  Equations \eqref{eq:U_corep} and~\eqref{eq:F_V-invariant} turn
  this into
  \[
  \Multunit_{13} \Multunit_{15} \Corep{U}_{23} \Corep{U}_{25}
  \Multunit_{35} \Corep{U}_{45} \{\DuCorep{V}^*_{56}
  \DuCorep{V}^*_{34}\} \DuCorep{V}^*_{36} \BrMultunit_{46}
  \{\DuCorep{V}_{36} \BrMultunit_{24} \DuCorep{V}^*_{36}\}
  \BrMultunit_{46}^* \DuCorep{V}_{36} \{\DuCorep{V}_{34}
  \DuCorep{V}_{56}\} \Corep{U}_{45}^* \Multunit[*]_{35}.
  \]
  Commuting \(\DuCorep{V}_{56}\) with~\(\DuCorep{V}_{34}\) and
  \(\DuCorep{V}_{36}\) with~\(\BrMultunit_{24}\) gives
  \[
  XXX =
  \Multunit_{13} \Multunit_{15} \Corep{U}_{23} \Corep{U}_{25}
  \{\Multunit_{35} \Corep{U}_{45}
  \DuCorep{V}^*_{34}\} \DuCorep{V}^*_{56} \DuCorep{V}^*_{36}
  \BrMultunit_{46} \BrMultunit_{24}
  \BrMultunit_{46}^* \DuCorep{V}_{36} \DuCorep{V}_{56}
  \{\DuCorep{V}_{34} \Corep{U}_{45}^* \Multunit[*]_{35}\}.
  \]
  Now~\eqref{eq:U_V_compatible} gives
  \[
  XXX =
  \Multunit_{13} \Multunit_{15} \Corep{U}_{23} \Corep{U}_{25}
  \DuCorep{V}^*_{34} \Corep{U}_{45} \{\Multunit_{35}
  \DuCorep{V}^*_{56} \DuCorep{V}^*_{36}\}
  \BrMultunit_{46} \BrMultunit_{24}
  \BrMultunit_{46}^* \{\DuCorep{V}_{36} \DuCorep{V}_{56}
  \Multunit[*]_{35}\} \Corep{U}_{45}^* \DuCorep{V}_{34}.
  \]
  We transform this using~\eqref{eq:V_corep}:
  \[
  XXX =
  \Multunit_{13} \Multunit_{15} \Corep{U}_{23} \Corep{U}_{25}
  \DuCorep{V}^*_{34} \Corep{U}_{45} \DuCorep{V}^*_{56}
  \{\Multunit_{35}\} \BrMultunit_{46} \BrMultunit_{24}
  \BrMultunit_{46}^* \{\Multunit[*]_{35}\} \DuCorep{V}_{56}
  \Corep{U}_{45}^* \DuCorep{V}_{34}.
  \]
  We commute~\(\Multunit_{35}\) with~\(\BrMultunit_{46}
  \BrMultunit_{24} \BrMultunit_{46}^*\):
  \[
  XXX =
  \Multunit_{13} \Multunit_{15} \Corep{U}_{23} \Corep{U}_{25}
  \DuCorep{V}^*_{34} \Corep{U}_{45} \DuCorep{V}^*_{56}
  \{\BrMultunit_{46} \BrMultunit_{24} \BrMultunit_{46}^*\}
  \DuCorep{V}_{56} \Corep{U}_{45}^* \DuCorep{V}_{34}.
  \]
  Now we use the braided pentagon
  equation~\eqref{eq:top-braided_pentagon} and the definition of the
  braiding through~\(Z\):
  \[
  XXX =
  \Multunit_{13} \Multunit_{15} \Corep{U}_{23} \{\Corep{U}_{25}
  \DuCorep{V}^*_{34}\} \Corep{U}_{45}
  \{\DuCorep{V}_{56}^* \BrMultunit_{24}\} Z_{46} \BrMultunit_{26} Z^*_{46}
  \DuCorep{V}_{56} \Corep{U}_{45}^* \DuCorep{V}_{34}.
  \]
  Now we commute \(\Corep{U}_{25}\) with~\(\DuCorep{V}^*_{34}\),
  \(\DuCorep{V}_{56}^*\) with~\(\BrMultunit_{24}\):
  \[
  XXX =
  \Multunit_{13} \Multunit_{15} \Corep{U}_{23} \DuCorep{V}^*_{34}
  \Corep{U}_{25} \Corep{U}_{45} \BrMultunit_{24}
  \{\DuCorep{V}_{56}^* Z_{46}\} \BrMultunit_{26} \{Z^*_{46}
  \DuCorep{V}_{56} \Corep{U}_{45}^*\} \DuCorep{V}_{34}.
  \]
  Equation~\eqref{eq:braiding} implies \(\Corep{U}_{45}
  \DuCorep{V}_{56}^* Z_{46} = \DuCorep{V}_{56}^* \Corep{U}_{45}\),
  so this becomes
  \[
  XXX =
  \Multunit_{13} \Multunit_{15} \Corep{U}_{23} \DuCorep{V}^*_{34}
  \{\Corep{U}_{25} \Corep{U}_{45}
  \BrMultunit_{24} \Corep{U}^*_{45}\} \DuCorep{V}^*_{56}
  \{\Corep{U}_{45} \BrMultunit_{26}\Corep{U}^*_{45}\}
  \DuCorep{V}_{56} \DuCorep{V}_{34}.
  \]
  Now we use~\eqref{eq:F_U-invariant} and
  commute~\(\BrMultunit_{26}\) with~\(\Corep{U}_{45}\):
  \[
  XXX =
  \Multunit_{13} \{\Multunit_{15}\} \{\Corep{U}_{23} \DuCorep{V}^*_{34}
  \BrMultunit_{24}\} \{\Corep{U}_{25} \DuCorep{V}^*_{56} \BrMultunit_{26}
  \DuCorep{V}_{56}\} \{\DuCorep{V}_{34}\}.
  \]
  Finally, we commute \(\Multunit_{15}\) with~\(\Corep{U}_{23}
  \DuCorep{V}^*_{34} \BrMultunit_{24}\) and \(\Multunit_{15}
  \Corep{U}_{25} \DuCorep{V}^*_{56} \BrMultunit_{26}
  \DuCorep{V}_{56}\) with~\(\DuCorep{V}_{34}\) to get
  \[
  XXX
  = \{\Multunit_{13} \Corep{U}_{23} \DuCorep{V}^*_{34}
  \BrMultunit_{24} \DuCorep{V}_{34}\} \{\Multunit_{15}
  \Corep{U}_{25} \DuCorep{V}^*_{56} \BrMultunit_{26} \DuCorep{V}_{56}\}
  = \Multunit[C]_{1234} \Multunit[C]_{1256}.
  \]
  This is the desired pentagon equation for~\(\Multunit[C]_{1234}\).

  Next we show that~\(\ProjBichar\) satisfies the pentagon equation:
  \begin{align*}
    \ProjBichar_{3456} \ProjBichar_{1234} \ProjBichar_{3456}
    &= \Multunit_{35} \Corep{U}_{45} \Multunit_{13} \Corep{U}_{23}
    \Corep{U}_{45}^* \Multunit[*]_{35}
    = \Multunit_{35} \Multunit_{13} \Corep{U}_{23} \Multunit[*]_{35}
    \\&= \Multunit_{13} \Multunit_{15} \Multunit_{35} \Corep{U}_{23}
    \Multunit[*]_{35}
    = \Multunit_{13} \Multunit_{15} \Corep{U}_{23} \Corep{U}_{25}
    = \Multunit_{13} \Corep{U}_{23} \Multunit_{15} \Corep{U}_{25}
    \\&= \ProjBichar_{1234} \ProjBichar_{1256}.
  \end{align*}
  The first and last equalities are the definition
  of~\(\ProjBichar\); the second step commutes \(\Corep{U}_{45}\)
  with~\(\Multunit_{13} \Corep{U}_{23}\); the third step uses the
  pentagon equation~\eqref{eq:pentagon} for~\(\Multunit\); the
  fourth step uses~\eqref{eq:U_corep}; the fifth step commutes
  \(\Multunit_{15}\) with~\(\Corep{U}_{23}\).

  Next we prove
  \(\ProjBichar_{3456} \Multunit[C]_{1234} = \Multunit[C]_{1234}
  \ProjBichar_{1256} \ProjBichar_{3456}\)
  or, equivalently,
  \(\ProjBichar_{3456} \Multunit[C]_{1234} \ProjBichar_{3456}^* =
  \Multunit[C]_{1234} \ProjBichar_{1256}\):
  \begin{align*}
    \ProjBichar_{3456} \Multunit[C]_{1234} \ProjBichar_{3456}^*
    &= \Multunit_{35} \Corep{U}_{45} \Multunit_{13}
    \Corep{U}_{23} \DuCorep{V}^*_{34} \BrMultunit_{24}
    \DuCorep{V}_{34} \Corep{U}_{45}^* \Multunit[*]_{35}
    \\&= \Multunit_{35} \Multunit_{13} \Corep{U}_{23}
    \Corep{U}_{45} \DuCorep{V}^*_{34} \BrMultunit_{24}
    \DuCorep{V}_{34} \Corep{U}_{45}^* \Multunit[*]_{35}
    \\&= \Multunit_{13} \Multunit_{15} \Multunit_{35}
    \Corep{U}_{23} \Corep{U}_{45} \DuCorep{V}^*_{34}
    \BrMultunit_{24} \DuCorep{V}_{34}
    \Corep{U}_{45}^* \Multunit[*]_{35}
    \\&= \Multunit_{13} \Multunit_{15} \Corep{U}_{23}
    \Corep{U}_{25} \Multunit_{35} \Corep{U}_{45}
    \DuCorep{V}^*_{34} \BrMultunit_{24} \DuCorep{V}_{34}
    \Corep{U}_{45}^* \Multunit[*]_{35}
    \\&= \Multunit_{13} \Corep{U}_{23} \Multunit_{15}
    \Corep{U}_{25} \DuCorep{V}^*_{34} \Corep{U}_{45}
    \Multunit_{35} \BrMultunit_{24} \Multunit[*]_{35} \Corep{U}_{45}^*
    \DuCorep{V}_{34}
    \\&= \Multunit_{13} \Corep{U}_{23} \DuCorep{V}^*_{34}
    \Multunit_{15} \Corep{U}_{25} \Corep{U}_{45} \BrMultunit_{24}
    \Corep{U}_{45}^* \DuCorep{V}_{34}
    \\&= \Multunit_{13} \Corep{U}_{23} \DuCorep{V}^*_{34}
    \Multunit_{15} \BrMultunit_{24} \Corep{U}_{25}
    \DuCorep{V}_{34}
    \\&= \Multunit_{13} \Corep{U}_{23} \DuCorep{V}^*_{34}
    \BrMultunit_{24} \DuCorep{V}_{34} \Multunit_{15}
    \Corep{U}_{25}
    = \Multunit[C]_{1234} \ProjBichar_{1256}.
  \end{align*}
  The first and last equalities are the definitions of
  \(\Multunit[C]\) and~\(\ProjBichar\); the second, sixth, and
  eighth steps commute unitaries in different legs; the third step
  uses the pentagon equation~\eqref{eq:pentagon} for~\(\Multunit\);
  the fourth step uses~\eqref{eq:U_corep}; the fifth step
  uses~\eqref{eq:U_V_compatible} and commutes unitaries in different
  legs; the seventh step uses~\eqref{eq:F_U-invariant}.

  Finally, we prove \(\Multunit[C]_{3456} \ProjBichar_{1234} =
  \ProjBichar_{1234} \ProjBichar_{1256} \Multunit[C]_{3456}\) by
  computing \(\ProjBichar_{1234}^* \Multunit[C]_{3456}
  \ProjBichar_{1234}\):
  \begin{align*}
    \ProjBichar_{1234}^* \Multunit[C]_{3456} \ProjBichar_{1234}
    &=  \Corep{U}_{23}^* \Multunit[*]_{13} \Multunit_{35}
    \Corep{U}_{45} \DuCorep{V}^*_{56} \BrMultunit_{46}
    \DuCorep{V}_{56} \Multunit_{13} \Corep{U}_{23}
    \\&=\Corep{U}_{23}^* \Multunit[*]_{13} \Multunit_{35}\Multunit_{13}
        \Corep{U}_{45} \DuCorep{V}^*_{56} \BrMultunit_{46}
        \DuCorep{V}_{56}\Corep{U}_{23}
    \\&= \Corep{U}_{23}^* \Multunit_{15}\Multunit_{35}
    \Corep{U}_{45} \DuCorep{V}^*_{56} \BrMultunit_{46}
    \DuCorep{V}_{56}\Corep{U}_{23}
    \\&=\Multunit_{15}\Corep{U}_{23}^* \Multunit_{35} \Corep{U}_{23}
    \Corep{U}_{45} \DuCorep{V}^*_{56} \BrMultunit_{46}
    \DuCorep{V}_{56}
    \\&=\Multunit_{15}\Corep{U}_{25}\Multunit_{35}
    \Corep{U}_{45} \DuCorep{V}^*_{56} \BrMultunit_{46}
    \DuCorep{V}_{56}=\ProjBichar_{1256}\Multunit[C]_{3456}.
  \end{align*}
  The first and last steps are the definitions of \(\Multunit[C]\)
  and~\(\ProjBichar\); the second and fourth steps commute unitaries
  in different legs; the third step uses the pentagon
  equation~\eqref{eq:pentagon} for~\(\Multunit\); the fifth step
  uses~\eqref{eq:U_corep}.
\end{proof}

\begin{theorem}
  \label{the:mang_mult_from_braided_and_standard}
  Let \(\Multunit\) be a manageable multiplicative unitary and let
  \((\Corep{U},\DuCorep{V},\BrMultunit)\) be a manageable braided
  multiplicative unitary over~\(\Multunit\).  Then the multiplicative
  unitaries \(\Multunit[C] \defeq \Multunit_{13} \Corep{U}_{23}
  \DuCorep{V}^*_{34} \BrMultunit_{24} \DuCorep{V}_{34}\) and
  \(\ProjBichar \defeq \Multunit_{13} \Corep{U}_{23}\) are manageable.
\end{theorem}

\begin{proof}
  Let \(\widetilde{\Multunit}\)
  and~\(Q\)
  witness the manageability of~\(\Multunit\),
  and let~\(\widetilde{\BrMultunit}\)
  and~\(Q_{\Hils[L]}\)
  witness the manageability of~\(\BrMultunit\).
  The construction of the unitary~\(\widetilde{\Corep{U}}\)
  in~\eqref{eq:second_leg_bichar_adapted} works for any right
  corepresentation of~\(\Multunit\)
  by the same argument; in particular, it works for~\(\Corep{U}\),
  so we get
  \(\widetilde{\Corep{U}} \in \U(\conj{\Hils[L]}\otimes\Hils)\) with
  \begin{equation}
    \label{eq:Corep_U_tilde}
    \bigl<x\otimes u\big| \Corep{U}\big| z\otimes y\bigr>
    = \bigl<\conj{z}\otimes Q u\big| \widetilde{\Corep{U}}\big|
    \conj{x}\otimes Q^{-1}y\bigr>
  \end{equation}
  for all \(x,z\in\Hils[L]\), \(u\in\dom(Q)\) and
  \(y\in\dom(Q^{-1})\).

  Let \(Q^C \defeq Q\otimes Q_{\Hils[L]}\in\U(\Hils\otimes\Hils[L])\) and
  \begin{equation}
    \label{eq:manageable_W_1234}
    \widetilde{\Multunit}{}^C_{1234} \defeq
    \widetilde{\Multunit}_{13} \DuCorep{V}^*_{34}
    \widetilde{\BrMultunit}_{24} \widetilde{\Corep{U}}_{23}
    \DuCorep{V}_{34}
    \in \U(\conj{\Hils}\otimes\conj{\Hils[L]}\otimes
    \Hils\otimes\Hils[L]).
  \end{equation}
  We claim that these operators witness the manageability
  of~\(\Multunit[C]\).  It is clear that~\(Q^C\) is strictly positive.

  The operators \(\Multunit_{13}\), \(\DuCorep{V}_{34}\),
  \(\Corep{U}_{23}\) and~\(\BrMultunit_{24}\) all commute with
  \(Q^C\otimes Q^C = Q\otimes Q_{\Hils[L]}\otimes Q\otimes
  Q_{\Hils[L]}\) by the manageability assumptions.
  Hence~\(\Multunit[C]\) commutes with~\(Q^C\otimes Q^C\).  It
  remains to check~\eqref{eq:Multunit_manageable} for
  \(\Multunit[C]\), \(\widetilde{\Multunit}{}^C\) and~\(Q^C\).  We
  relegate this technical computation to
  Lemma~\ref{lem:mang_mult_from_braided_and_standard} in the
  appendix.  This finishes the proof that~\(\Multunit[C]\) is
  manageable.  Now Proposition~\ref{pro:ProjBichar_manageable} shows
  that~\(\ProjBichar\) is manageable as well.
\end{proof}

\subsection{Analysis of a quantum group with projection}
\label{sec:Qnt_grp_with_proj}

In this section, we construct a braided multiplicative unitary from a
quantum group with projection.  Our starting point is a Hilbert
space~\(\Hils\) with two unitaries
\(\Multunit[C],\ProjBichar\in\U(\Hils\otimes\Hils)\) satisfying the
conditions in Proposition~\ref{pro:qg_projection_mu}.  We must
construct another Hilbert space~\(\Hils[L]\) with operators
\(\Corep{U}\in\U(\Hils[L]\otimes\Hils)\),
\(\DuCorep{V}\in\U(\Hils\otimes\Hils[L])\) and
\(\BrMultunit\in\U(\Hils[L]\otimes\Hils[L])\) as in
Definition~\ref{def:braided_multiplicative_unitary}.

In particular, the corepresentations \(\Corep{U}\)
and~\(\DuCorep{V}\) form a Drinfeld pair for the multiplicative
unitary~\(\ProjBichar\).  The simplest
general construction of such a Drinfeld pair lives on the tensor
product Hilbert space \(\Hils[L] \defeq \conj{\Hils} \otimes
\Hils\), where~\(\conj{\Hils}\) denotes the conjugate Hilbert space
of~\(\Hils\).  Therefore, we will use this rather large Hilbert
space.

Let \(\G = (A,\Comult[A])\) be the \(\Cst\)\nb-quantum group generated
by~\(\ProjBichar\), which is manageable by
Proposition~\ref{pro:ProjBichar_manageable}.
Let \(\G[H] = (C,\Comult[C])\)
be the \(\Cst\)\nb-quantum
group generated by the manageable multiplicative
unitary~\(\Multunit[C]\).
By construction, we have inclusion maps
\(\iota \colon C\to \Bound(\Hils)\)
and \(\hat{\iota} \colon C\to \Bound(\Hils)\),
which are non-degenerate \Star{}homomorphisms.  The reduced
bicharacter is the unique unitary
\(\multunit[C]\in\U(\hat{C}\otimes C)\)
with \(\Multunit[C] = (\hat{\iota}\otimes \iota)(\multunit[C])\)
or, briefly, \(\Multunit[C]= \multunit[C]_{\hat{\iota}\iota}\).
By construction, \(A\subseteq \Mult(C)\)
and \(\hat{A}\subseteq \Mult(\hat{C})\)
as \(\Cst\)\nb-subalgebras of \(\Bound(\Hils)\).

The representations \((\iota,\hat{\iota})\)
form a Heisenberg pair for the quantum group~\((C,\Comult[C])\)
in the notation of~\cite{Meyer-Roy-Woronowicz:Twisted_tensor}.  This
Heisenberg pair generates an anti-Heisenberg pair
\(\alpha\colon C\to\Bound(\conj{\Hils})\),
\(\hat\alpha\colon \hat{C} \to \Bound(\conj{\Hils})\)
by \cite{Meyer-Roy-Woronowicz:Twisted_tensor}*{Lemma 3.6}.  Thus
\begin{equation}
  \label{eq:anti-Heisenberg_MultunitC}
  \multunit[C]_{1 \alpha} \multunit[C]_{\hat\alpha 3}
  =  \multunit[C]_{\hat\alpha 3} \multunit[C]_{1 3} \multunit[C]_{1 \alpha}
  \qquad\text{in }\U(\hat{C} \otimes \Comp(\conj{\Hils}) \otimes \hat{C}).
\end{equation}
The restriction of a Heisenberg or anti-Heisenberg pair for~\(\G[H]\)
to~\(\G\)
remains a Heisenberg or anti-Heisenberg pair, respectively.  Thus
\begin{equation}
  \label{eq:anti-Heisenberg_MultunitA}
  \projbichar_{1 \alpha}\projbichar_{\hat\alpha 3}
  = \projbichar_{\hat\alpha 3} \projbichar_{1 3} \projbichar_{1 \alpha}
  \qquad\text{in } \U(\hat{A} \otimes \Comp(\conj{\Hils}) \otimes \hat{A}).
\end{equation}
To make computations shorter, we shall use leg numbering notation such
as
\(\projbichar_{ij}, \multunit[C]_{ij}\in \U(\conj{\Hils} \otimes \Hils
\otimes \conj{\Hils} \otimes\Hils)\)
for \(1\le i<j\le 4\).
This means the unitary acting on the \(i\)th
and \(j\)th
tensor factor by applying the appropriate representations of \(C\)
or~\(\hat{C}\)
to the two legs of \(\projbichar\)
or~\(\multunit[C]\),
respectively.  For instance,
\(\projbichar_{12} = (\hat\alpha\otimes \iota)(\projbichar) \otimes
1_{\conj{\Hils} \otimes \Hils}\).
This notation is not ambiguous if we also specify the Hilbert space on
which the operator acts because we have given one representation of
\(C\) and~\(\hat{C}\) on \(\Hils\) and~\(\conj{\Hils}\) each.  We let
\begin{alignat}{2}
  \label{eq:U_from_projection}
  \Corep{U} &\defeq \projbichar_{23} \projbichar_{13}
  \defeq (\hat{\iota}\otimes\iota)\projbichar_{23} \cdot
  (\hat{\alpha}\otimes\iota)\projbichar_{13}
  &\qquad &\text{in } \U(\conj{\Hils}  \otimes \Hils \otimes \Hils),\\
  \label{eq:V_from_projection}
  \DuCorep{V} &\defeq \projbichar_{12} \projbichar_{13}
  \defeq (\hat{\iota}\otimes\alpha)\projbichar_{12} \cdot
  (\hat{\iota}\otimes\iota)\projbichar_{13}
  &\qquad &\text{in } \U(\Hils \otimes \conj{\Hils}  \otimes \Hils),\\
  \label{eq:F_from_projection}
  \BrMultunit &\defeq \projbichar_{14}^* \projbichar_{24}^*
  \multunit[C]_{24} \multunit[C]_{14}
  &\qquad &\text{in } \U(\conj{\Hils} \otimes \Hils \otimes
  \conj{\Hils}  \otimes \Hils).
\end{alignat}

\begin{theorem}
  \label{the:analysis_qg_projection}
  The unitaries \(\ProjBichar\in\U(\Hils\otimes\Hils)\),
  \(\Corep{U}\in\U(\Hils[L]\otimes\Hils)\),
  \(\DuCorep{V}\in\U(\Hils\otimes\Hils[L])\),
  and \(\BrMultunit\in\U(\Hils[L]\otimes\Hils[L])\)
  form a braided multiplicative unitary.
\end{theorem}

The proof of Theorem~\ref{the:analysis_qg_projection} will take some
work.  The precise formulas for \(\alpha\) and~\(\hat{\alpha}\) will
only matter in the end when we check the manageability of our braided
multiplicative unitary.  Thus our construction really uses
that~\(\Multunit[C]\) is manageable (or at least modular).  The
pentagon equation~\eqref{eq:pentagon} for~\(\ProjBichar\) holds by
assumption.  Equations \eqref{eq:U_corep} and~\eqref{eq:V_corep},
which say that \(\Corep{U}\) and~\(\DuCorep{V}\) are
corepresentations, amount to
\begin{alignat*}{2}
  \projbichar_{34} \projbichar_{23} \projbichar_{13}
  &= \projbichar_{23} \projbichar_{13} \projbichar_{24} \projbichar_{14}
  \projbichar_{34}
  &\qquad&\text{in }\U(\conj{\Hils}  \otimes \Hils \otimes \Hils\otimes
  \Hils),\\
  \projbichar_{23} \projbichar_{24} \projbichar_{12}
  &= \projbichar_{12} \projbichar_{13} \projbichar_{14} \projbichar_{23}
  \projbichar_{24}
  &\qquad&\text{in }\U(\Hils \otimes \Hils\otimes \conj{\Hils} \otimes
  \Hils).
\end{alignat*}
We get both equations using the pentagon equation for~\(\ProjBichar\)
twice, in legs where the representations of \(\G[H]\)
and hence of~\(\G\)
form a Heisenberg pair.  Thus \(\Corep{U}\)
and~\(\DuCorep{V}\) are corepresentations of~\(\ProjBichar\).  The
Drinfeld compatibility condition~\eqref{eq:U_V_compatible} becomes
\[
\projbichar_{34} \projbichar_{24} \projbichar_{14}
\projbichar_{12} \projbichar_{13}
= \projbichar_{12} \projbichar_{13} \projbichar_{14} \projbichar_{34}
\projbichar_{24}
\qquad\text{in }\U(\Hils\otimes \conj{\Hils}  \otimes \Hils \otimes
\Hils\otimes \Hils).
\]
The anti-Heisenberg and Heisenberg properties of our representations
on the second and third leg give
\(\projbichar_{24} \projbichar_{14} \projbichar_{12} =
\projbichar_{12} \projbichar_{24}\)
and
\(\projbichar_{13} \projbichar_{14} \projbichar_{34} =
\projbichar_{34} \projbichar_{13}\).  Hence
\[
\projbichar_{34} \projbichar_{24} \projbichar_{14}
\projbichar_{12} \projbichar_{13}
= \projbichar_{34} \projbichar_{12} \projbichar_{24} \projbichar_{13}
= \projbichar_{12} \projbichar_{34} \projbichar_{13} \projbichar_{24}
= \projbichar_{12} \projbichar_{13} \projbichar_{14} \projbichar_{34}
\projbichar_{24}
\]
as needed.  Thus \(\Corep{U}\)
and~\(\DuCorep{V}\)
are Drinfeld compatible.  The equivariance of~\(\BrMultunit\)
with respect to~\(\Corep{U}\) in~\eqref{eq:F_U-invariant} amounts to
\begin{equation}
  \label{eq:F_U-invariant_analysis}
  \projbichar_{25} \projbichar_{15} \projbichar_{45} \projbichar_{35}
  \projbichar_{14}^* \projbichar_{24}^* \multunit[C]_{24} \multunit[C]_{14}
  = \projbichar_{14}^* \projbichar_{24}^* \multunit[C]_{24} \multunit[C]_{14}
  \projbichar_{25} \projbichar_{15} \projbichar_{45} \projbichar_{35}
\end{equation}
in \(\U(\conj{\Hils} \otimes \Hils \otimes \conj{\Hils} \otimes
\Hils\otimes \Hils)\).
Since we have a Heisenberg pair on the fourth leg, we may use the
conditions in Proposition~\ref{pro:qg_projection_mu} to simplify
\[
\multunit[C]_{24} \multunit[C]_{14}
\projbichar_{25} \projbichar_{15} \projbichar_{45}
= \multunit[C]_{24} \projbichar_{25} (\multunit[C]_{14} \projbichar_{15}
\projbichar_{45})
= \multunit[C]_{24} \projbichar_{25} \projbichar_{45} \multunit[C]_{14}
= \projbichar_{45} \multunit[C]_{24} \multunit[C]_{14}.
\]
Hence the right side in~\eqref{eq:F_U-invariant_analysis} becomes
\[
\projbichar_{14}^* \projbichar_{24}^* \multunit[C]_{24} \multunit[C]_{14}
\projbichar_{25} \projbichar_{15} \projbichar_{45} \projbichar_{35}
= \projbichar_{14}^* \projbichar_{24}^* \projbichar_{45}
\multunit[C]_{24} \multunit[C]_{14} \projbichar_{35}.
\]
Plugging this in and cancelling
\(\multunit[C]_{24} \multunit[C]_{14} \projbichar_{35}\),
we see that~\eqref{eq:F_U-invariant_analysis} is equivalent to
\[
\projbichar_{25} \projbichar_{15} \projbichar_{45} \projbichar_{14}^*
\projbichar_{24}^*
= \projbichar_{14}^* \projbichar_{24}^* \projbichar_{45}
\quad\text{or}\quad
\projbichar_{24} \projbichar_{25} \projbichar_{14} \projbichar_{15}
\projbichar_{45} = \projbichar_{45} \projbichar_{24} \projbichar_{14}.
\]
Since we have Heisenberg pairs on the second and fourth legs, this
follows by two applications of the pentagon equation
for~\(\projbichar\).

The condition~\eqref{eq:F_V-invariant} about~\(\BrMultunit\)
being equivariant with respect to~\(\DuCorep{V}\) becomes
\begin{equation}
  \label{eq:F_V-invariant_analysis}
  \projbichar_{14} \projbichar_{15} \projbichar_{12} \projbichar_{13}
  \projbichar_{25}^* \projbichar_{35}^* \multunit[C]_{35} \multunit[C]_{25}
  = \projbichar_{25}^* \projbichar_{35}^* \multunit[C]_{35} \multunit[C]_{25}
  \projbichar_{14} \projbichar_{15} \projbichar_{12} \projbichar_{13}
\end{equation}
in
\(\U(\Hils\otimes \conj{\Hils} \otimes \Hils \otimes \conj{\Hils}
\otimes \Hils)\).
Since we have an anti-Heisenberg pair on the second and a Heisenberg
pair on the third leg, we get
\(\multunit[C]_{25} \projbichar_{15} \projbichar_{12} =
\projbichar_{12} \multunit[C]_{25}\)
and
\(\projbichar_{13} \projbichar_{15} \multunit[C]_{35} =
\multunit[C]_{35} \projbichar_{13}\).
Thus the right side of~\eqref{eq:F_V-invariant_analysis} becomes
\begin{align*}
  &\phantom{{}={}}\projbichar_{25}^* \projbichar_{35}^* \multunit[C]_{35}
  \multunit[C]_{25} \projbichar_{14} \projbichar_{15} \projbichar_{12}
  \projbichar_{13}\\
  &= \projbichar_{14} \projbichar_{25}^* \projbichar_{35}^*
  \multunit[C]_{35} (\multunit[C]_{25} \projbichar_{15}
  \projbichar_{12}) \projbichar_{13}
  = \projbichar_{14} \projbichar_{25}^* \projbichar_{35}^*
  \multunit[C]_{35} \projbichar_{12}
  \multunit[C]_{25} \projbichar_{13}
  \\&= \projbichar_{14} \projbichar_{25}^* \projbichar_{35}^* \projbichar_{12}
  (\multunit[C]_{35} \projbichar_{13}) \multunit[C]_{25}
  = \projbichar_{14} \projbichar_{25}^* \projbichar_{35}^* \projbichar_{12}
  \projbichar_{13} \projbichar_{15} \multunit[C]_{35} \multunit[C]_{25}.
\end{align*}
Now we may cancel \(\projbichar_{14}\)
on the left and \(\multunit[C]_{35} \multunit[C]_{25}\)
on the right to transform our condition into
\(\projbichar_{15} \projbichar_{12} \projbichar_{13}
\projbichar_{25}^* \projbichar_{35}^* = \projbichar_{25}^*
\projbichar_{35}^* \projbichar_{12} \projbichar_{13}
\projbichar_{15}\) or, equivalently,
\[
\projbichar_{35} \projbichar_{25} \projbichar_{15} \projbichar_{12} \projbichar_{13}
= \projbichar_{12} \projbichar_{13} \projbichar_{15} \projbichar_{35} \projbichar_{25}
\]
in
\(\U(\Hils\otimes \conj{\Hils} \otimes \Hils \otimes \conj{\Hils}
\otimes \Hils)\).
The anti-Heisenberg pair condition on the second leg gives
\(\projbichar_{25} \projbichar_{15} \projbichar_{12} =
\projbichar_{12} \projbichar_{25}\),
the Heisenberg pair on the third leg gives
\(\projbichar_{13} \projbichar_{15} \projbichar_{35} =
\projbichar_{35} \projbichar_{13}\).
Plugging this in, our condition becomes
\[
\projbichar_{35} \projbichar_{12} \projbichar_{25} \projbichar_{13}
= \projbichar_{12} \projbichar_{35} \projbichar_{13}\projbichar_{25},
\]
which is manifestly true.  Thus our operators
satisfy~\eqref{eq:F_V-invariant} as well.

Checking the braided pentagon equation~\eqref{eq:top-braided_pentagon}
is a long computation.  We may omit it because of the following
trick.  In the proof of
Theorem~\ref{the:standard_mult_from_braided_and_standard}, the braided
pentagon equation is used exactly once.  Therefore, if all the other
conditions in Definition~\ref{def:braided_multiplicative_unitary}
hold, then the braided pentagon
equation~\eqref{eq:top-braided_pentagon} is both sufficient and
\emph{necessary} for the usual pentagon equation for the
unitary~\(\Multunit[D]\)
constructed in
Theorem~\ref{the:standard_mult_from_braided_and_standard}.  Thus the
proof of Theorem~\ref{the:analysis_qg_projection} is finished up to
the braided pentagon equation, and this follows from the proof of the
following theorem.

\begin{theorem}
  \label{the:there_and_back_again}
  Let \(\Multunit[C],\ProjBichar\in\U(\Hils\otimes\Hils)\)
  define a \(\Cst\)\nb-quantum
  group with projection as in
  Proposition~\textup{\ref{pro:qg_projection_mu}}.  Construct a
  braided multiplicative unitary
  \((\ProjBichar,\Corep{U},\DuCorep{V},\BrMultunit)\)
  on the Hilbert space \(\Hils[L]= \conj{\Hils}\otimes\Hils\)
  with the unitaries defined in
  \eqref{eq:U_from_projection}--\eqref{eq:F_from_projection}.  From
  this, construct a multiplicative unitary~\(\Multunit[D]\)
  with a projection~\(\ProjBichar^D\) on the Hilbert space
  \[
  \Hils\otimes\Hils[L]\otimes\Hils\otimes\Hils[L] \cong
  \Hils\otimes\conj{\Hils}\otimes\Hils
  \otimes\Hils\otimes\conj{\Hils} \otimes\Hils
  \]
  by
  Theorem~\textup{\ref{the:standard_mult_from_braided_and_standard}}.

  The braided multiplicative unitary
  \((\ProjBichar,\Corep{U},\DuCorep{V},\BrMultunit)\)
  and the multiplicative unitary~\(\Multunit[D]\)
  are manageable.  And~\(\Multunit[D]\)
  generates the same \(\Cst\)\nb-quantum
  group as~\(\Multunit[C]\).
  The isomorphism between these \(\Cst\)\nb-quantum
  groups maps~\(\ProjBichar\) to~\(\ProjBichar^D\).
\end{theorem}

Roughly speaking, going from a \(\Cst\)\nb-quantum
group with projection to a braided \(\Cst\)\nb-quantum
group and back gives an isomorphic \(\Cst\)\nb-quantum
group with projection.

The definitions in
Theorem~\ref{the:standard_mult_from_braided_and_standard} amount to
\begin{align*}
  \Multunit[D] {}\defeq{} &
  \projbichar_{14} \projbichar_{34} \projbichar_{24}
  \projbichar_{46}^* \projbichar_{45}^* \projbichar_{26}^*
  \projbichar_{36}^* \multunit[C]_{36} \multunit[C]_{26}
  \projbichar_{45} \projbichar_{46}\\
  {}={}& \projbichar_{14} \projbichar_{34} \projbichar_{24}
  \projbichar_{46}^* \projbichar_{26}^*
  \projbichar_{36}^* \multunit[C]_{36} \multunit[C]_{26}
  \projbichar_{46},\\
  \ProjBichar^D {}\defeq{}&
  \projbichar_{14} \projbichar_{34} \projbichar_{24}.
\end{align*}

Our first task is to construct representations \(\pi\)
and~\(\hat{\pi}\) of \(C\) and~\(\hat{C}\) that form a Heisenberg pair
and that satisfy \((\hat{\pi}\otimes\pi)\multunit[C]= \Multunit[D]\).
This implies that~\(\Multunit[D]\) satisfies the pentagon equation.
As we remarked above, this implies the braided pentagon
equation~\eqref{eq:top-braided_pentagon} for~\(\BrMultunit\), which
still remained to be proven.

\begin{lemma}
  \label{lemm:aux-C-heis}
  There is a representation
  \(\pi'\colon C\to\Bound(\Hils\otimes\conj{\Hils}\otimes\Hils)\)
  such that
  \begin{alignat}{2}
    \label{eq:pi_on_WC}
    (\Id_{\hat{C}}\otimes\pi')\multunit[C]
    &=\projbichar_{12}\multunit[C]_{14}
    &\qquad\text{in \(\U(\hat{C}\otimes\Comp(\Hils\otimes\conj{\Hils}\otimes\Hils))\),}\\
    \label{eq:pi_on_P}
    (\Id_{\hat{C}}\otimes\pi')\projbichar
    &=\projbichar_{12}\projbichar_{14}
    &\qquad\text{in \(\U(\hat{C}\otimes\Comp(\Hils\otimes\conj{\Hils}\otimes\Hils))\).}
  \end{alignat}
\end{lemma}

\begin{proof}
  Let
  \(\Duprojbichar = \sigma(\projbichar)^* \in \U(C\otimes \hat{C})\)
  for the flip~\(\sigma\).
  Define \(\varphi\colon C\to\Mult(C\otimes\Comp(\conj{\Hils}))\)
  by
  \(\varphi(c)\defeq
  \Duprojbichar_{1\hat{\alpha}} (1\otimes\alpha(c))
  \Duprojbichar{}_{1\hat{\alpha}}^*\).
  Then
  \[
  (\Id_{\hat{C}}\otimes\varphi) \multunit[C]
  =\Duprojbichar_{2\hat{\alpha}}\multunit[C]_{1\alpha}
  \Duprojbichar{}_{2\hat{\alpha}}^*
  =\flip_{23}(\projbichar_{\hat{\alpha}3}^* \multunit[C]_{1\alpha}
  \projbichar_{\hat{\alpha}3})
  \qquad\text{in }\Mult(\hat{C}\otimes C\otimes\Comp(\conj{\Hils})).
  \]
  The second condition in Proposition~\ref{pro:qg_projection_mu} is
  equivalent to
  \(\multunit[C]_{1\alpha}\projbichar_{\hat{\alpha}3} =
  \projbichar_{\hat{\alpha}3} \projbichar_{13}\multunit[C]_{1
    \alpha}\)
  in \(\U(\hat{C}\otimes\Comp(\conj{\Hils})\otimes C)\)
  because we have an anti\nb-Heisenberg pair on the second leg.  Thus
  \[
  (\Id_{\hat{C}}\otimes\varphi)\multunit[C]
  = \flip_{23}(\projbichar_{13}\multunit[C]_{1\alpha})
  =\projbichar_{12}\multunit[C]_{1\alpha}
  \qquad\text{in }\Mult(\hat{C}\otimes C\otimes\Comp(\conj{\Hils})).
  \]
  We may define \(\pi'(c)\defeq
  \bigl((\hat{\iota}\otimes\iota\circ\alpha^{-1})\varphi(c)\bigr)_{13}\)
  because~\(\alpha\) is automatically injective (see
  \cite{Roy:Codoubles}*{Proposition 3.7}).  This is the unique
  representation that satisfies~\eqref{eq:pi_on_WC}.
  Replacing~\(\multunit[C]\) by~\(\projbichar\) in the above
  computations gives~\eqref{eq:pi_on_P}.
\end{proof}

\begin{lemma}
  \label{lem:aux-C-Heis}
  Let~\(\pi'\)
  be as in the previous lemma.  The pair of representations
  \((\pi,\hat{\pi})\)
  of \(C\)
  and~\(\hat{C}\)
  on \(\Hils\otimes\conj{\Hils}\otimes\Hils\) defined by
  \[
  \pi(c)\defeq \projbichar_{13}^*\pi'(c)\projbichar_{13},
  \qquad
  \hat{\pi}(\hat{c})\defeq \projbichar_{13}^*
  ((\hat{\alpha}\otimes\hat{\iota}) \DuComult[C](\hat{c}))_{23}
  \projbichar_{13}
  \]
  is an \(\G[H]\)\nb-Heisenberg pair.
\end{lemma}

\begin{proof}
  Let
  \(\hat{\pi}'(\hat{c})\defeq ((\hat{\alpha}\otimes\hat{\iota})
  \DuComult[C](\hat{c}))_{23}\).
  The lemma is equivalent to \((\pi',\hat{\pi}')\)
  being \(\G[H]\)\nb-Heisenberg.
  Recall that
  \((\DuComult[C]\otimes\Id_{C})\multunit[C] =
  \multunit[C]_{23}\multunit[C]_{13}\).
  Lemma~\ref{lemm:aux-C-heis} gives
  \[
  \multunit[C]_{\hat{\pi}'5}\multunit[C]_{1\pi'}
  =\multunit[C]_{45}\multunit[C]_{35}\projbichar_{12}\multunit[C]_{14}
  =\projbichar_{12}\multunit[C]_{45}\multunit[C]_{14}\multunit[C]_{35}
  \]
  in
  \(\U(\hat{C}\otimes
  \Comp(\Hils\otimes\conj{\Hils}\otimes\Hils)\otimes C)\).
  Since we have a Heisenberg pair on the fourth leg, the pentagon
  equation~\eqref{eq:pentagon} gives
  \[
  \projbichar_{12} \multunit[C]_{45} \multunit[C]_{14} \multunit[C]_{35}
  = \projbichar_{12} \multunit[C]_{14} \multunit[C]_{15} \multunit[C]_{45}
  \multunit[C]_{35},
  \]
  which is equivalent to
  \(\multunit[C]_{\hat{\pi}'5}\multunit[C]_{1\pi'} =
  \multunit[C]_{1\pi'}\multunit[C]_{15}\multunit[C]_{\hat{\pi}'5}\)
  in
  \(\U(\hat{C}\otimes\Comp(\Hils\otimes\conj{\Hils}\otimes\Hils)\otimes
  C)\).
\end{proof}

\begin{lemma}
  \label{lem:analysis-multunit}
  \(\Multunit[D] = (\hat{\pi}\otimes\pi)\multunit[C]\)
  and \(\ProjBichar^D = (\hat{\pi}\otimes\pi)\projbichar\)
  in \(\U(\Hils\otimes\Hils[L]\otimes\Hils\otimes \Hils[L])\).
\end{lemma}

\begin{proof}
  Cancelling~\(\projbichar_{45}\)
  gives
  \(\Multunit[D]= \projbichar_{14} \projbichar_{34} \projbichar_{24}
  \projbichar_{46}^* \projbichar_{26}^* \projbichar_{36}^*
  \multunit[C]_{36} \multunit[C]_{26} \projbichar_{46}\).
  Computing as in the proof of Lemma~\ref{lem:aux-C-Heis}, we get
  \[
  (\hat{\pi}\otimes\pi)\multunit[C]
  = \projbichar_{13}^* \projbichar_{46}^* \projbichar_{34}
  \multunit[C]_{36} \projbichar_{24} \multunit[C]_{26}
  \projbichar_{13}\projbichar_{46}
  = \projbichar_{46}^* \projbichar_{13}^* \projbichar_{34}
  \projbichar_{24} \multunit[C]_{36} \projbichar_{13}
  \multunit[C]_{26} \projbichar_{46}.
  \]
  Since we have a Heisenberg pair on the third leg,
  \(\multunit[C]_{36} \projbichar_{13} = \projbichar_{13}
  \projbichar_{16} \multunit[C]_{36}\).  Hence
  \[
  (\hat{\pi}\otimes\pi)\multunit[C]
  = \projbichar_{46}^* \projbichar_{13}^* \projbichar_{34}
  \projbichar_{24} \projbichar_{13} \multunit[C]_{36} \multunit[C]_{26}
  \projbichar_{46}.
  \]
  Since we have Heisenberg pairs on the third and fourth legs,
  \(\projbichar_{34}\projbichar_{13} =
  \projbichar_{13}\projbichar_{14}\projbichar_{34}\)
  and
  \(\projbichar_{46}\projbichar_{i4} = \projbichar_{i 6}\projbichar_{i
    6} \projbichar_{4 6}\)
  for all \(i=1,2,3\).
  Using these identities (for the expressions within brackets in the
  computation below) we get
  \begin{multline*}
    \projbichar_{46}^* \projbichar_{13}^* \projbichar_{34}
    \projbichar_{24} \projbichar_{13} \projbichar_{16}
    = \projbichar_{46}^* (\projbichar_{13}^* \projbichar_{34}
    \projbichar_{13}) \projbichar_{24}  \projbichar_{16}
    = \projbichar_{46}^* \projbichar_{14} \projbichar_{34}
    \projbichar_{24}  \projbichar_{16} \\
    = (\projbichar_{46}^* \projbichar_{14} \projbichar_{16})
    \projbichar_{34} \projbichar_{24}
    = \projbichar_{14} (\projbichar_{46}^* \projbichar_{34})
    \projbichar_{24}
    = \projbichar_{14} \projbichar_{34} \projbichar_{46}^*
    \projbichar_{36}^* \projbichar_{24}\\
    = \projbichar_{14} \projbichar_{34} (\projbichar_{46}^*
    \projbichar_{24}) \projbichar_{36}^*
    = \projbichar_{14} \projbichar_{34} \projbichar_{24}
    \projbichar_{46}^* \projbichar_{26}^* \projbichar_{36}^* .
  \end{multline*}
  The last two computations together give
  \[
  (\hat{\pi}\otimes\pi)\multunit[C]
  =\projbichar_{14} \projbichar_{34} \projbichar_{24}
  \projbichar_{46}^* \projbichar_{26}^* \projbichar_{36}^*
  \multunit[C]_{36}\multunit[C]_{26}\projbichar_{46}
  =\Multunit[D].
  \]
  Equation~\eqref{eq:pi_on_P} allows a similar computation
  with~\(\projbichar\) instead of~\(\multunit[C]\).  This gives
  \[
  (\hat{\pi}\otimes\pi)\projbichar
  =\projbichar_{14} \projbichar_{34} \projbichar_{24}
  \projbichar_{46}^* \projbichar_{26}^* \projbichar_{36}^*
  \projbichar_{36}\projbichar_{26}\projbichar_{46}
  =\ProjBichar^D.\qedhere
  \]
\end{proof}

The following remarks apply to any Heisenberg pair \((\pi,\hat{\pi})\)
for a \(\Cst\)\nb-quantum
group \(\G[H]= (C,\Comult[C])\)
on a Hilbert space~\(\Hils'\).
Being a Heisenberg pair means that
\((\hat{\pi}\otimes\pi)\multunit[C]\)
is a multiplicative unitary.  It is unclear, in general, whether this
multiplicative unitary is manageable.  If it is manageable, then we
claim that the \(\Cst\)\nb-quantum
group that it generates is isomorphic to the one we started with.  The
representations in a Heisenberg pair are automatically faithful by
\cite{Roy:Codoubles}*{Proposition 3.7}.  Hence we may view \(C\)
and~\(\hat{C}\)
as subalgebras of \(\Bound(\Hils')\),
and \((\hat{\pi}\otimes\pi)\multunit[C]\)
is a unitary multiplier of
\(\hat{C}\otimes C \subseteq \Bound(\Hils'\otimes\Hils')\).
It makes no difference whether we take slices on the first leg with
elements of \(\Bound(\Hils')_*\)
or~\(\hat{C}^*\):
both generate the same \(\Cst\)\nb-subalgebra
of \(\Bound(\Hils')\),
namely, \(\pi(C)\).
The comultiplication on the quantum group generated
by~\((\hat{\pi}\otimes\pi)\multunit[C]\)
is defined so that the isomorphism~\(\pi\)
is a Hopf \Star{}homomorphism.
Thus the \(\Cst\)\nb-quantum group generated
by~\((\hat{\pi}\otimes\pi)\multunit[C]\) is isomorphic to~\(\G[H]\)
for any Heisenberg pair for
which~\((\hat{\pi}\otimes\pi)\multunit[C]\) is manageable.
Furthermore, Lemma~\ref{lem:analysis-multunit}
shows that this Hopf \Star{}isomorphism maps~\(\projbichar\)
to~\(\ProjBichar^D\),
so we also get the same projection on our \(\Cst\)\nb-quantum group.

Thus the proof of Theorem~\ref{the:there_and_back_again} will be
finished once we show that~\(\Multunit[D]\)
and the braided multiplicative unitary
\((\Corep{U},\DuCorep{V},\BrMultunit)\)
are manageable.  By
Theorem~\ref{the:mang_mult_from_braided_and_standard},
\(\Multunit[D]\)
is manageable once \((\Corep{U},\DuCorep{V},\BrMultunit)\)
is manageable.  So it remains to prove this.

The braiding on \(\Hils[L]\otimes\Hils[L]\)
comes from the unique unitary~\(Z\)
that verifies~\eqref{eq:braiding}.  A simple computation shows that
\(Z = \projbichar_{14}^* \projbichar_{24}^* \projbichar_{13}^*
\projbichar_{23}^*\)
in \(\U(\conj{\Hils}\otimes\Hils\otimes\conj{\Hils}\otimes\Hils)\)
does the job.  This gives
\[
Z^*\BrMultunit=\projbichar_{23}\projbichar_{13}\projbichar_{24}\projbichar_{14}
\projbichar_{14}^* \projbichar_{24}^* \multunit[C]_{24}\multunit[C]_{14}
=  \projbichar_{23}\projbichar_{13} \multunit[C]_{24}\multunit[C]_{14}.
\]
Now we use that \((\iota,\hat{\iota})\)
is the standard Heisenberg pair, generated by~\(\Multunit[C]\),
and that the anti-Heisenberg pair \((\alpha,\hat{\alpha})\)
is constructed as in \cite{Meyer-Roy-Woronowicz:Twisted_tensor}*{Lemma
  3.6}; that is,
\(\hat{\alpha}(\hat{a}) \defeq
\hat{a}^{\transpose\circ\hat{\iota}\circ\hat{R}_C}\)
and \(\alpha(a)\defeq a^{\transpose\circ\iota\circ R_C}\).  Thus
\[
Z^*\BrMultunit =
\projbichar_{23}^{\hat{\iota}\otimes\transpose\circ\iota\circ R_{C}}
\projbichar_{13}^{\transpose\circ\hat{\iota}\circ\hat{R}_{C}\otimes\transpose\circ\iota\circ R_{C}}
(\multunit[C]_{24})^{\hat{\iota}\otimes\iota}
(\multunit[C]_{14})^{\transpose\circ\hat{\iota}\circ\hat{R}_{C}\otimes\iota}.
\]

Let \(Q_C\)
and \(\widetilde{\Multunit[C]} \in \U(\conj{\Hils}\otimes\Hils)\)
witness the manageability of
\(\Multunit[C] = (\hat{\iota}\otimes\iota)
\multunit[C]\in\U(\Hils\otimes\Hils)\),
see Appendix~\ref{app_sec:mang_bichar}.  Since~\(\ProjBichar\)
is manageable by Proposition~\ref{pro:ProjBichar_manageable}, so is
the dual \(\DuProjBichar = \Sigma \ProjBichar^* \Sigma\).
This is witnessed by a certain unitary
\(\widetilde{\DuProjBichar} \in \U(\conj{\Hils}\otimes\Hils)\).
We have
\((\multunit[C])^{\transpose \hat\iota R_{\hat{C}} \otimes \iota} =
(\multunit[C])^{\transpose \hat\iota \otimes \iota R_C} =
(\widetilde{\Multunit}{}^C)^*\)
and
\(\projbichar^{\transpose \hat\iota R_{\hat{C}} \otimes \transpose
  \iota R_C} = \projbichar^{\transpose \otimes \transpose}\)
by \cite{Woronowicz:Mult_unit_to_Qgrp}*{Theorem 1.6 (5)}
and~\cite{Meyer-Roy-Woronowicz:Homomorphisms}*{(19)}.  Similarly,
\((\Sigma \projbichar^{\hat{\iota}\otimes\transpose\circ\iota\circ
  R_{C}} \Sigma)^* = \Duprojbichar{}^{\transpose\circ\iota\circ
  R_C\otimes \hat{\iota}} = \widetilde{\DuProjBichar}\).  Thus
\[
Z^*\BrMultunit = \Flip_{23}\widetilde{\DuProjBichar}_{23}\Flip_{23}
\ProjBichar^{\transpose\otimes\transpose}_{13} \Multunit[C]_{24}
(\widetilde{\Multunit}{}^{C}_{14})^*
\qquad\text{in }\U(\conj{\Hils}\otimes\Hils\otimes\conj{\Hils}\otimes\Hils).
\]
Let \(Q\defeq Q_C^\transpose\otimes Q_C\).
Then \(Q\otimes Q\)
commutes with~\(\BrMultunit\),
\(Q\otimes Q_C\)
commutes with~\(\Corep{U}\),
and \(Q_C\otimes Q\)
commutes with~\(\DuCorep{V}\).
Define
\(\widetilde{\BrMultunit} \in \U(\conj{\Hils[L]}\otimes\Hils[L])\) by
\begin{equation}
  \label{eq:mang-br-analysis}
  \widetilde{\BrMultunit} \defeq
  \widetilde{\Multunit}{}^C_{24}(\Multunit[C]_{14})^*
  (\ProjBichar^*)^{\transpose\otimes\transpose}_{23}
  \widetilde{\ProjBichar}^{\transpose\otimes\transpose}_{13}
  \qquad\text{in }\U(\Hils\otimes\conj{\Hils}\otimes\conj{\Hils}\otimes\Hils).
\end{equation}
This unitary and~\(Q\) witness the
manageability of the braided multiplicative unitary
\((\Corep{U},\DuCorep{V},\BrMultunit)\).  The rather technical proof
of this fact is relegated to the appendix, see
Lemma~\ref{lemm:brmanag_analysis}.

\section{Examples of Quantum Groups with Projections}
\label{sec:Ex_Qunt_grp_proj}

The simplest examples of semidirect products of connected Lie groups
are \(\textup{E}(2) = \R^2\rtimes \T\) and the real and complex
\(ax+b\)-groups \(\R\rtimes \R_{>0}^\times\) and \(\C\rtimes
\C^\times\), where the second, multiplicative factor acts by multiplication
on the first, additive factor.  The group~\(\textup{E}(2)\) is the group of
isometries of the plane.  Another very important example is the
Poincar\'e group, the semidirect product of the Lorentz group
with~\(\R^4\).

When quantising such groups, one may try to preserve the semidirect
product structure, that is, construct \(\Cst\)\nb-quantum groups
with projection.  For instance, the quantum \(\textup{E}(2)\) groups by
Woronowicz~\cite{Woronowicz:Qnt_E2_and_Pontr_dual} have obvious
morphisms to the circle group~\(\T\) and back that compose to the
identity on~\(\T\).  The quantum \(az+b\) groups introduced by
Woronowicz~\cite{Woronowicz:Quantum_azb} and
Sołtan~\cite{Soltan:New_az_plus_b} -- which deform
\(\C\rtimes\C^\times\) -- have obvious morphisms to the group
\(\C_q^\times = q^{\Z+\ima\R}\subseteq \C^\times\) (with
multiplication as group structure) and back, which compose to the
identity on~\(\C_q^\times\); see also
\cite{Kasprzak-Soltan:Quantum_projection}*{Example 3.7}.  The
quantum \(ax+b\) group by Woronowicz and
Zakrzewski~\cite{Woronowicz-Zakrzewski:Quantum_axb} has an obvious
projection onto the group \(\R_{>0}^\times \cong \R\).

There are also quantum versions of semidirect product groups that
appear to have no such projection.  This includes the \(az+b\)
groups by Baaj and Skandalis, see
\cite{Vaes-Vainerman:Extension_of_lcqg}*{Section 5.3}, the \(ax+b\)
groups by Stachura~\cite{Stachura:ax_plus_b} and the
\(\kappa\)\nb-Poincaré groups by
Stachura~\cite{Stachura:Kappa-Poincare}.  These examples are all
constructed using the formalism of quantum group extensions
of~\cite{Vaes-Vainerman:Extension_of_lcqg}.  Quantum group
extensions are compared with quantum groups with projection
in~\cite{Kasprzak-Soltan:Extension_vs_Qgrp_proj}.

As an example of our theory, we are going to construct a braided
multiplicative unitary that generates ``simplified quantum
\(\textup{E}(2)\),'' a variant of quantum \(\textup{E}(2)\) also due
to Woronowicz (unpublished); whereas the quantum \(\textup{E}(2)\)
groups in~\cite{Woronowicz:Qnt_E2_and_Pontr_dual} deform a double
cover of~\(\textup{E}(2)\), the simplified variants deform
\(\textup{E}(2)\) itself.  A common feature of simplified quantum
\(\textup{E}(2)\) and the quantum groups with projection mentioned
above is that the image of the projection is a classical, Abelian
group.  This is to be expected when deforming semidirect products by
Abelian groups because these cannot be deformed to quantum groups in
interesting ways.  We begin by observing some common features of
braided multiplicative unitaries in case~\(\Multunit\) generates an
Abelian group~\(G\).

Let~\(\hat{G}\)
be the dual group.  The corepresentations \(\Corep{U}\)
and~\(\DuCorep{V}\)
in a braided multiplicative unitary are equivalent to representations
of \(G\)
and~\(\hat{G}\)
on the Hilbert space~\(\Hils[L]\),
respectively.  The compatibility
condition~\eqref{eq:U_V_compatible} for \(\Corep{U}\)
and~\(\DuCorep{V}\)
says here that the representations of \(G\)
and~\(\hat{G}\)
commute.  Thus we may combine them to one representation of
\(\hat{G}\times G\)
on~\(\Hils[L]\).

We can further normalise this representation because the left
regular representation of any quantum group absorbs every other
representation.  The operator \(\BrMultunit\otimes 1\) on
\(\Hils[L]\otimes L^2(\hat{G} \times G)\) is a braided
multiplicative unitary if and only if~\(\BrMultunit\) is, and it
generates an equivalent semidirect product quantum group.  Thus we
may assume without loss of generality that our representation is a
multiple of the left regular representation:
\[
\Hils[L]= L^2(\hat{G}\times G) \otimes \Hils[L]_0
\]
for some separable Hilbert space~\(\Hils[L]_0\), with \(\Cst(\hat{G}
\times G)\) acting only on the first tensor factor, by the regular
representation.  We may identify \(\Cont_0(G\times\hat{G}) \cong
\Cst(\hat{G}\times G)\) and \(L^2(\hat{G}\times G) \cong L^2(G\times
\hat{G})\) by the Fourier transform, and the regular representation
of \(\Cst(\hat{G}\times G)\) on \(L^2(\hat{G}\times G) \cong
L^2(G\times \hat{G})\) becomes the standard representation of
\(\Cont_0(G\times\hat{G})\) on \(L^2(G\times\hat{G})\) by pointwise
multiplication.

For some examples, a variant of the above is useful: if the
representation of~\(\hat{G}\times G\) on~\(\Hils[L]\) factors
through an Abelian locally compact group~\(H\), then we may
use~\(H\) instead of \(\hat{G}\times G\) in the above
simplification.  That is, we seek a braided multiplicative unitary
on the Hilbert space \(L^2(H)\otimes\Hils[L]_0\) with
\(\hat{G}\times G\) acting only on the first tensor factor, through
the regular representation of~\(H\) and the given homomorphism
\(\hat{G}\times G\to H\).

For instance, the compact quantum group \(\textup{U}_q(2)\) is a
semidirect product of the braided quantum group \(\textup{SU}_q(2)\)
by the circle~\(\T\)
(see~\cite{Kasprzak-Meyer-Roy-Woronowicz:Braided_SU2}), and the
relevant representation of \(\Z\times\T\) factors through a
homomorphism \(\Z\times\T\to\T\), \((n,z)\mapsto \lambda^n\cdot z\)
for some \(\lambda\in\T\).  For the quantum \(az+b\) groups,
\(G=\C_q\) is self-dual, and the relevant representation of
\(\hat{G}\times G\) factors through the map \(\hat{G}\times G \cong
G\times G \to G\), where the second map is the multiplication map
\((x,y)\mapsto x\cdot y\).

When we have simplified~\(\Hils[L]\) to \(L^2(G\times\hat{G})
\otimes \Hils[L]_0\) with \(\Cont_0(G\times\hat{G})\) acting by
pointwise multiplication, the braiding
operator~\(\Braiding{\Hils[L]}{\Hils[L]}\) is the operator of
pointwise multiplication with the circle-valued function
\begin{equation}
  \label{eq:braiding_Abelian}
  (g_1,\chi_1,g_2,\chi_2)\mapsto \chi_1(g_2).
\end{equation}
The conditions \eqref{eq:F_U-invariant} and~\eqref{eq:F_V-invariant}
for~\(\BrMultunit\) mean that
\(\BrMultunit\in\U(\Hils[L]\otimes\Hils[L])\) is a \(\hat{G}\times
G\)-equivariant operator with respect to the tensor product
representation of \(\hat{G}\times G\) on
\(\Hils[L]\otimes\Hils[L]\).  In terms of the above spectral
analysis, \(f\in \Cont_0(G\times\hat{G})\) acts on
\(\Hils[L]\otimes\Hils[L]\) by pointwise multiplication on each
fibre with the function
\[
\Delta^* f(g_1,\chi_1,g_2,\chi_2) = f(g_1g_2,\chi_1\chi_2)
\]
for all \((g_1,\chi_1,g_2,\chi_2)\in G\times\hat{G}\times
G\times\hat{G}\).  An operator on \(\Hils[L]\otimes\Hils[L]\) is
\(\hat{G}\times G\)-equivariant if and only if it commutes with the
operators of pointwise multiplication by~\(\Delta^* f\) for \(f\in
\Cont_0(G\times\hat{G})\).

Summing up, it suffices to look for braided multiplicative unitaries
over an Abelian group~\(G\) on the Hilbert space \(\Hils[L]=
L^2(G\times\hat{G}, \ell^2(\N))\).  Such a braided multiplicative
unitary \(\BrMultunit\in\U(\Hils[L]\otimes\Hils[L])\) must commute
with the operators of pointwise multiplication by functions in
\(\Delta^*\Contb(G\times\hat{G}) \subseteq
\Contb(G\times\hat{G}\times G\times\hat{G})\), and it must satisfy
the braided pentagon equation~\eqref{eq:top-braided_pentagon}, where
the braiding is given by pointwise multiplication with the function
in~\eqref{eq:braiding_Abelian}.

\subsection{Simplified quantum E(2) groups}
\label{sec:qnt_E_2}

Now we specialise to simplified quantum E(2) groups.
They were already treated in~\cite{Roy:Qgrp_with_proj} except
for the manageability of the braided multiplicative unitary in
question.  A \(\Cst\)\nb-algebraic version of this construction
appears in~\cite{Roy:Braided_Cstar}.

Here the image of the projection is the circle group \(G=\T\), so
\(\hat{G}=\Z\).  The analysis above suggests to construct a braided
multiplicative unitary for quantum E(2) on a Hilbert space of the form
\(L^2(\T\times\Z) \otimes \Hils[L]_0\).  Actually, we shall not
need~\(\Hils[L]_0\) and work on \(L^2(\T\times\Z)\) itself.  First
we describe the standard multiplicative unitary~\(\Multunit\)
generating~\(\T\).

Let \(\Hils\defeq \ell^2(\Z)\) and let~\(\{e_p\}_{p\in\Z}\) be an
orthonormal basis of~\(\Hils\).  Define
\[
u e_p\defeq e_{p+1} \quad\text{ and }\quad
\hat{N}e_p\defeq p e_p.
\]
The shift~\(u\) is unitary and generates the regular representation
of~\(\Z\) on~\(\Hils\).  The operator~\(\hat{N}\) is self-adjoint
with spectrum~\(\Z\), and the resulting representation of
\(\Cont_0(\Z) \cong \Cst(\T)\) is the Fourier transform of the
regular representation of~\(\T\) on~\(L^2(\T)\).  These operators
generate a representation of the crossed product
\(\Cont_0(\Z)\rtimes\Z \cong \Comp(\ell^2\Z)\).  That is, they
satisfy the commutation relation
\begin{equation}
  \label{eq:comm_u_hat_N}
  u^*\hat{N}u=\hat{N}+1.
\end{equation}

The multiplicative unitary generating~\(\T\) is
\[
\Multunit\defeq(1\otimes u)^{\hat{N}\otimes 1}
= \int_{\Z\times\T} z^s \, \diff E_{\hat{N}}(s)\otimes\diff E_{u}(z),
\qquad e_k\otimes e_l \mapsto e_k \otimes e_{l+k},
\]
where \(\diff E_{\hat{N}}\) and~\(\diff E_{u}\) are the spectral
measures of \(\hat{N}\) and~\(u\), respectively.  The commutation
relation~\eqref{eq:comm_u_hat_N} implies the pentagon equation
\[
(1\otimes 1\otimes u)^{1\otimes\hat{N}\otimes 1}
(1\otimes u\otimes 1)^{\hat{N}\otimes 1\otimes 1}
(1\otimes 1\otimes u^*)^{1\otimes\hat{N}\otimes 1}
= (1\otimes u\otimes u)^{\hat{N}\otimes 1\otimes 1}.
\]
for~\(\Multunit\), that is, \(\Multunit\) is a multiplicative
unitary.  Simple computations show that \(Q=1\) and
\(\widetilde{\Multunit}\defeq (1\otimes u^*)^{\hat{N}\otimes 1}\)
witness the manageability of~\(\Multunit\).  Slices of~\(\Multunit\)
in the first and second leg clearly generate the
\(\Cst\)\nb-algebras \(\Cst(\hat{N})\cong \Cont_0(\Z)\) and
\(\Cst(u)\cong \Cont_0(\T)\).  The comultiplications defined
by~\(\Multunit\) satisfy \(\Comult[\Cont(\T)](u)\defeq u\otimes u\)
and \(\Comult[\Contvin(\Z)](\hat{N})\defeq \hat{N}\otimes 1\dotplus
1\otimes\hat{N}\), where \(\hat{N}\otimes 1\dotplus
1\otimes\hat{N}\) means the unbounded affiliated element of
\(\Contvin(\Z\times\Z)\) given by the closure of the essentially
self-adjoint operator \(\hat{N}\otimes 1+ 1\otimes\hat{N}\).

Let \(0<q<1\).  We identify \(\Z\times\T \cong \C_q^\times =
q^{\Z+\ima \R}\subseteq \C^\times\) by mapping \((n,z)\mapsto
q^n\cdot z\).  As suggested above, we are going to construct a
multiplicative unitary on the Hilbert space \(L^2(\Z\times\T) =
L^2(\C_q^\times)\) with the regular representation of
\(\Z\times\T\).  We choose the orthonormal basis \(e_{k,l} \defeq
\delta_k\otimes z^l\) for \(k,l\in\Z\) in~\(L^2(\C_q^\times)\) and
thus identify \(\Hils[L] \cong \Hils\otimes \Hils\).  In our chosen
basis, \(\Z\) acts by \(\hat\alpha_n(e_{k,l}) = e_{k+n,l}\) for
\(n,k,l\in\Z\) and~\(\T\) acts by \(\alpha_\zeta(e_{k,l}) = \zeta^l
\cdot e_{k,l}\) for \(k,l\in\Z\), \(\zeta\in\T\).  Thus the right
and left corepresentations \(\Corep{U}\in\U(\Hils[L]\otimes\Hils)\)
and \(\DuCorep{V}\in\U(\Hils\otimes\Hils[L])\) and the resulting
braiding operator \(\Braiding{\Hils[L]}{\Hils[L]} \in
\U(\Hils[L]\otimes\Hils[L])\) are
\begin{alignat*}{2}
  \Corep{U} &= \Multunit_{23},&\qquad
  e_{k,l,m} &\mapsto e_{k,l,m+l}\\
  \DuCorep{V} &= \Multunit_{12},&\qquad
  e_{m,k,l} &\mapsto e_{m,k+m,l}\\
  \Braiding{\Hils[L]}{\Hils[L]} = Z\Sigma
  &= \Multunit[*]_{23} \Sigma,&\qquad
  e_{k,l} \otimes e_{n,p} &\mapsto e_{n,p} \otimes e_{k-p,l}.
\end{alignat*}
We also describe the representations of \(\Cont(\T) \cong \Cst(\Z)\)
and \(\Cont_0(\Z) \cong \Cst(\T)\) on~\(\Hils[L]\) through a unitary
operator~\(\udrinf\) and a self-adjoint operator \(\Nhatdrinf\) with
spectrum~\(\Z\) and commuting with~\(\udrinf\):
\[
\udrinf(e_{k,l})\defeq e_{k+1,l}
\qquad
\Nhatdrinf(e_{k,l})\defeq l e_{k,l}.
\]
We define a closed operator \(\Upsilon = \Ph{\Upsilon}
\Mod{\Upsilon}\) on~\(\Hils[L]\) by
\[
\Ph{\Upsilon} e_{k,l}\defeq e_{k,l+1},\qquad
\Mod{\Upsilon} e_{k,l}\defeq q^{2k+l} e_{k,l},\qquad
\Upsilon e_{k,l}\defeq q^{2k+l} e_{k,l+1}.
\]
The operator~\(\Ph{\Upsilon}\) is unitary and~\(\Mod{\Upsilon}\) is
strictly positive with spectrum \(q^\Z\cup\{0\}\), and
\(\Ph{\Upsilon}\) and~\(\Mod{\Upsilon}\) satisfy the following
commutation relations:
\begin{equation}
  \label{eq:commutation_Upsilon}
  \left\{\begin{alignedat}{2}
    \Ph{\Upsilon}\Mod{\Upsilon}\Ph{\Upsilon}^* &=
    q^{-1}\Mod{\Upsilon},\\
    \udrinf \Ph{\Upsilon} &= \Ph{\Upsilon} \udrinf,
    &\qquad& \udrinf \Mod{\Upsilon} \udrinf^* = q^{-2} \Mod{\Upsilon},\\
    \Ph{\Upsilon} \Nhatdrinf\Ph{\Upsilon}^* &= \Nhatdrinf-1,
    &\qquad& \Mod{\Upsilon} \text{ and }
    \Nhatdrinf \text{ strongly commute.}
  \end{alignedat}\right.
\end{equation}
Thus \(\Upsilon^{-1}(e_{k,l}) = q^{-2k-l+1} e_{k,l-1}\).  The closed
operator
\[
X\defeq \Upsilon q^{-2\Nhatdrinf}\otimes\Upsilon^{-1},\qquad
e_{k,l} \otimes e_{n,p} \mapsto
q^{2(k-n) - (l+p) + 1} e_{k,l+1} \otimes e_{n,p-1},
\]
on \(\Hils[L] \otimes \Hils[L]\) is normal because \(\Mod{X}\colon
e_{k,l} \otimes e_{n,p} \mapsto q^{2(k-n) - (l+p) + 1} e_{k,l}
\otimes e_{n,p}\) commutes with its phase \(\Ph{X}\colon e_{k,l}
\otimes e_{n,p} \mapsto e_{k,l+1} \otimes e_{n,p-1}\) in the polar
decomposition.  The spectrum of~\(X\) is \(\cl{\C}_q \defeq
\C_q^\times \cup\{0\}\).  Both \(\Mod{X}\) and~\(\Ph{X}\) strongly
commute with \(\udrinf\otimes \udrinf\) and \(\Nhatdrinf \otimes 1
\dotplus 1 \otimes \Nhatdrinf\).  Thus~\(X\) is equivariant for the
tensor product representations \(\Corep{U} \tenscorep \Corep{U}\)
and \(\DuCorep{V} \tenscorep \DuCorep{V}\).  Hence any circle-valued
function \(F\colon \cl{\C}_q \to \T\) gives a unitary~\(F(X)\) on
\(\Hils[L] \otimes \Hils[L]\) that is equivariant with respect to
\(\Corep{U} \tenscorep \Corep{U}\) and \(\DuCorep{V} \tenscorep
\DuCorep{V}\).

We want to choose~\(F\) so that~\(F(X)\) satisfies the braided
pentagon equation.  Since the functional calculus is compatible with
conjugation by unitaries, the top-braided pentagon equation
for~\(F(X)\) says
\begin{multline}
  \label{eq:braided_pentagon_Upsilon}
  F(F(X_{23}) X_{12} F(X_{23})^*)
  = F(X_{12}) \Braiding{\Hils[L]}{\Hils[L]}_{23} F(X_{12})
  \Dualbraiding{\Hils[L]}{\Hils[L]}_{23}
  \\= F(X_{12}) F(Z_{23} X_{13} Z_{23}^*).
\end{multline}
We compute
\[
Z_{23} X_{13} Z_{23}^* (e_{k,l} \otimes e_{n,p} \otimes e_{r,s})
= q^{2(k-p-r) - (l+s) + 1} (e_{k,l+1} \otimes e_{n,p} \otimes e_{p,s-1}).
\]
Thus
\[
Z_{23} X_{13} Z_{23}^*
= \Upsilon q^{-2\Nhatdrinf}\otimes q^{-2\Nhatdrinf}\otimes\Upsilon^{-1}
= X_{12} \cdot X_{23}.
\]
Hence~\eqref{eq:braided_pentagon_Upsilon} becomes
\begin{equation}
  \label{eq:braided_pentagon_Upsilon2}
  F(F(X_{23}) X_{12} F(X_{23})) = F(X_{12}) F(X_{12} X_{23}).
\end{equation}

The quantum exponential function is defined
in~\cite{Woronowicz:Operator_eq_E2} by
\begin{equation}
  \label{eq:Qnt_exp}
  \brmultunit_q(z)\defeq
  \begin{cases}
    \displaystyle \prod_{k=1}^{\infty}
    \frac{1+q^{2k}\conj{z}}{1+q^{2k} z}
    & z\in\cl{\C}_q \setminus\{-q^{2\Z}\},\\
    -1 & \text{otherwise.}
  \end{cases}
\end{equation}
This product converges absolutely outside \(-q^{2\Z}\), and
\(\bigl\lvert\frac{1+q^{2k}\conj{z}}{1+q^{2k} z}\bigr\rvert = 1\)
for all \(z\neq -q^{2k}\).  Thus~\(\brmultunit_q\) is a unitary
multiplier of \(\Cont_0(\cl{\C_q})\), and
\[
\BrMultunit\defeq
\brmultunit_q(\Upsilon^{-1} q^{-2\Nhatdrinf}\otimes\Upsilon)
\]
is a unitary operator on~\(\Hils[L]\otimes\Hils[L]\).

\begin{theorem}
  \label{the:Br_mult_Qnt_pl}
  The triple \((\Corep{U},\DuCorep{V},\BrMultunit)\) is a manageable
  braided multiplicative unitary on~\(\Hils[L]\) relative
  to~\(\Multunit\).
\end{theorem}

The proof will occupy the rest of this section.  For
\((\Corep{U},\DuCorep{V},\BrMultunit)\) to be a braided
multiplicative unitary, it only remains to verify the braided
pentagon equation, which is equivalent
to~\eqref{eq:braided_pentagon_Upsilon}.  We shall use the properties of the
quantum exponential function established
in~\cite{Woronowicz:Operator_eq_E2}.

The operators
\[
R \defeq X_{12}
= \Upsilon q^{-2\Nhatdrinf}\otimes\Upsilon^{-1} \otimes 1,\qquad
S \defeq X_{12} X_{23}
= \Upsilon q^{-2\Nhatdrinf}\otimes q^{-2\Nhatdrinf}\otimes \Upsilon^{-1}
\]
are normal and satisfy the commutation relations
in~\cite{Woronowicz:Operator_eq_E2}*{(0.1)}, that is, their phases
and their absolute values strongly commute and
\[
\Ph{R}^* \abs{S} \Ph{R} = q \abs{S},\qquad
\Ph{S} \abs{R} \Ph{S}^* = q \abs{R}.
\]
Since \(R^{-1} S = X_{23}\) is also normal with
spectrum~\(\cl{\C}_q\), \cite{Woronowicz:Operator_eq_E2}*{Theorems
  2.1--2} apply and show that \(R\dotplus S\) is normal with
spectrum~\(\cl{\C}_q\) and
\[
\brmultunit_q(X_{23}) \cdot X_{12}\cdot \brmultunit_q(X_{23})^*
= \brmultunit_q(R^{-1} S) \cdot R\cdot \brmultunit_q(R^{-1} S)^*
= R \dotplus S.
\]
Moreover, \cite{Woronowicz:Operator_eq_E2}*{Theorem 3.1} gives
\[
\brmultunit_q(R) \brmultunit_q(S) = \brmultunit_q(R \dotplus S).
\]
Both results of~\cite{Woronowicz:Operator_eq_E2} together
give~\eqref{eq:braided_pentagon_Upsilon2} for \(F=\brmultunit_q\);
this is equivalent to the braided pentagon equation
for~\(\BrMultunit\).

Now we turn to braided manageability.  First we compute the
unitary~\(\widetilde{Z}\).  It is the unique unitary on
\(\conj{\Hils[L]}\otimes\Hils[L]\) that
satisfies~\eqref{eq:braiding-manag}.  The
contragradient~\(\widetilde{\corep{U}}{}^*\) of~\(\corep{U}\) is given
in the standard basis \(\conj{e_{k,l}} \otimes e_m\) of
\(\conj{\Hils[L]}\otimes\Hils\) by
\(\widetilde{\corep{U}}^*(\conj{e_{k,l}} \otimes e_m) =
\conj{e_{k,l}} \otimes e_{m-l}\).  Hence \(\widetilde{Z} \in
\U(\conj{\Hils[L]}\otimes\Hils[L])\) acts on the standard basis by
\[
\widetilde{Z}(\conj{e_{k,l}} \otimes e_{n,p}) =
\conj{e_{k,l}} \otimes e_{n-l,p}.
\]
Equivalently, \(\widetilde{Z} = (1\otimes
\udrinf)^{\Nhatdrinf^{\transpose}\otimes 1}\).

Next we define the operator~\(Q_{\Hils[L]}\) required by
Definition~\ref{def:braided_manageable}:
\[
Q_{\Hils[L]} e_{k,l} \defeq q^{-l} e_{k,l}.
\]
This is a strictly positive operator on~\(\Hils[L]\) with
spectrum~\(q^\Z\cup\{0\}\).  It commutes with \(\udrinf\)
and~\(\Nhatdrinf\) and therefore satisfies
\eqref{eq:br_manag_commute_U} and~\eqref{eq:br_manag_commute_V}.
The operator \(Q_{\Hils[L]}\otimes Q_{\Hils[L]}\), mapping \(e_{k,l}
\otimes e_{n,p}\mapsto q^{-(l+p)} e_{k,l} \otimes e_{n,p}\),
commutes with \(X= \Upsilon q^{-2\Nhatdrinf} \otimes \Upsilon^{-1}\)
and therefore with \(\BrMultunit = \brmultunit_q(X)\).
Thus~\eqref{eq:br_manag_commute_F} holds as well.

Finally, we need a unitary \(\widetilde{\BrMultunit}\in
\U(\conj{\Hils[L]}\otimes\Hils[L])\) that
satisfies~\eqref{eq:br_manag}.  It suffices to check this if the
vectors \(x,y,u,v\) involved are standard basis vectors
\(x=e_{k,l}\), \(y=e_{n,p}\), \(u=e_{a,b}\), \(v=e_{c,d}\)
for~\(\Hils[L]\).  Using our explicit formulas for~\(Z\)
and~\(\widetilde{Z}\), we may rewrite~\eqref{eq:br_manag} as
\[
(e_{k,l} \otimes e_{a-l,b} \mid \BrMultunit \mid e_{n,p} \otimes
e_{c,d}) =
(\conj{e_{n,p}}\otimes q^{-b} e_{a,b} \mid \widetilde{\BrMultunit}
\mid \conj{e_{k,l}} \otimes q^d e_{c+l,d}).
\]
Substituting \(\gamma= c+l\) and \(\BrMultunit = \brmultunit_q(X)\),
this becomes
\begin{equation}
  \label{eq:manageable_E2_concrete}
  (\conj{e_{n,p}}\otimes e_{a,b} \mid \widetilde{\BrMultunit}
  \mid \conj{e_{k,l}} \otimes e_{\gamma,d})
  = q^{b-d} (e_{k,l} \otimes e_{a-l,b} \mid \brmultunit_q(X) \mid
  e_{n,p} \otimes e_{\gamma-l,d})
\end{equation}
for \(a,b,\gamma,d,k,l,n,p\in\Z\).  So the issue is whether the
bilinear form~\(\widetilde{\BrMultunit}\) defined by this equation
is unitary.

To compute the right hand side in~\eqref{eq:manageable_E2_concrete},
we Fourier transform the restrictions of~\(\brmultunit_q\) to the
circles \(\abs{z}=q^n\), \(n\in\Z\), and write
\[
\brmultunit_q(z) = \sum_{m\in\Z} \brmultunit_m(\Mod{z}) \Ph{z}^m,
\]
see~\cite{Baaj:Regular-Representation-E-2} or
\cite{Woronowicz:Quantum_azb}*{Appendix A}.  The
scalars~\(\brmultunit_m(q^n)\) for \(m,n\in\Z\) are real and satisfy
\[
\brmultunit_m(q^n) = (-q)^m \brmultunit_{-m}(q^{n-m}).
\]
The vectors \(e_{n,p} \otimes e_{\gamma-l,d}\) and \(e_{k,l} \otimes
e_{a-l,b}\) are eigenvectors of~\(\Mod{X}\) with eigenvalues
\(q^{2(n-\gamma+l) - (p+d)+1}\) and \(q^{2(k-a+l) - (b+l)+1}\),
respectively.  And~\(\Ph{X}^m\) acts on these vectors by
\(e_{k,l}\otimes e_{a,b}\mapsto e_{k,l+m}\otimes e_{a,b-m}\).  Thus
\begin{align*}
  &\phantom{{}={}} q^{b-d} (e_{k,l} \otimes e_{a-l,b} \mid \brmultunit_q(X) \mid
  e_{n,p} \otimes e_{\gamma-l,d})\\
  &=\sum_{m\in\Z}  q^{b-d} (e_{k,l} \otimes e_{a-l,b} \mid
  \Ph{X}^m \brmultunit_m(\Mod{X})\mid e_{n,p} \otimes e_{\gamma-l,d})
  \\&= \sum_{m\in\Z}  q^{b-d} \cdot \brmultunit_m(q^{2(n-\gamma+l) - (p+d)+1})
  \delta_{k,n}\delta_{l,p+m} \delta_{a-l,\gamma-l} \delta_{b,d-m}
  \\&= \delta_{k,n} \delta_{a,\gamma} \delta_{p,l+b-d}\cdot
  q^{b-d} \cdot \brmultunit_{d-b}(q^{2 k- 2 a + l - b+1})
  \\&= \delta_{k,n} \delta_{a,\gamma} \delta_{p,l+b-d}\cdot
  (-1)^{b-d} \cdot \brmultunit_{b-d}(q^{2 k- 2 a + l - d+1}).
\end{align*}
Now we define an unbounded normal operator~\(\widetilde{X}\)
on~\(\conj{\Hils[L]} \otimes \Hils\) with spectrum~\(\cl{\C}_q\) by
\[
\Mod{\widetilde{X}}(\conj{e_{k,l}} \otimes e_{n,p}) =
q^{2(k-n) + l - p+1} \conj{e_{k,l}} \otimes e_{n,p},\quad
\Ph{\widetilde{X}}(\conj{e_{k,l}} \otimes e_{n,p}) =
-\conj{e_{k,l+1}} \otimes e_{n,p+1},
\]
so \(\widetilde{X}(\conj{e_{k,l}} \otimes e_{n,p}) = -1\cdot
q^{2(k-n) + l - p+1} \conj{e_{k,l+1}} \otimes e_{n,p+1}\).  We claim
that the unitary \(\widetilde{\BrMultunit} \defeq
\brmultunit_q(\widetilde{X})^*\) will do:
\begin{align*}
  (\conj{e_{n,p}}\otimes e_{a,b} \mid \widetilde{\BrMultunit}
  \mid \conj{e_{k,l}} \otimes e_{\gamma,d})
  &=\sum_{m\in\Z} (\conj{e_{k,l}} \otimes e_{\gamma,d} \mid
  \Ph{\widetilde{X}}^m \brmultunit_m(\Mod{\widetilde{X}})
  \mid \conj{e_{n,p}}\otimes e_{a,b})
  \\&= \sum_{m\in\Z} (-1)^m \brmultunit_m(q^{2(n-a) + p - b+1})
  \delta_{k,n}\delta_{l,p+m} \delta_{\gamma,a} \delta_{d,b+m}
  \\&= \delta_{k,n} \delta_{a,\gamma} \delta_{p,l+b-d}\cdot
  (-1)^{d-b} \cdot \brmultunit_{d-b}(q^{2 k- 2 a + l - d + 1}).
\end{align*}
This is equal to the result of the computation above,
so~\eqref{eq:manageable_E2_concrete} holds.  Thus our braided
multiplicative unitary is manageable, and
Theorem~\ref{the:Br_mult_Qnt_pl} is proved.

\appendix
\section{Some Manageability Techniques}
\label{app_sec:mang_bichar}

Let \(\Multunit[A]\in\U(\Hils_A\otimes\Hils_A)\) and
\(\Multunit[B]\in\U(\Hils_B\otimes\Hils_B)\) be manageable
multiplicative unitaries as in
\cite{Woronowicz:Mult_unit_to_Qgrp}*{Definition
  1.2}, which generate \(\Cst\)\nb-quantum groups \(\Qgrp{G}{A}\) and
\(\Qgrp{H}{B}\).

Let \(\bichar\in\U(\hat{A}\otimes B)\) be a bicharacter from \(\G\)
to~\(\G[H]\). Let \(\Bichar\in\U(\Hils_A\otimes\Hils_B)\) be the
concrete realisation of~\(\bichar\).  Then \(\Bichar\in
\U(\Hils_A\otimes \Hils_B)\) is adapted to~\(\Multunit[B]\) in the
sense of \cite{Woronowicz:Mult_unit_to_Qgrp}*{Definition 1.3} by
\cite{Meyer-Roy-Woronowicz:Homomorphisms}*{Lemma 3.2}.  Thus
\(\Bichar\in\U(\Hils_A\otimes\Hils_B)\) is manageable by
\cite{Woronowicz:Mult_unit_to_Qgrp}*{Theorem 1.6}; that is, there is a
unitary \(\widetilde{\Bichar}\in\U(\conj{\Hils}_A\otimes\Hils_B)\)
with
\begin{equation}
  \label{eq:second_leg_bichar_adapted}
  \bigl(x\otimes u\big| \Bichar\big| z\otimes y\bigr)
  = \bigl(\conj{z}\otimes Q_B u\big| \widetilde{\Bichar}\big|
  \conj{x}\otimes Q_B^{-1}y\bigr)
\end{equation}
for all \(x,z\in\Hils_A\), \(u\in\dom(Q_B)\) and
\(y\in\dom(Q_B^{-1})\); here \(Q_B\) is one of the operators in the
manageability condition for~\(\Multunit[B]\).

\begin{lemma}
  \label{lem:V_tilde_commute_Q}
  \(\Bichar (Q_A\otimes Q_B)\Bichar^* = Q_A\otimes Q_B\) and
  \(\widetilde{\Bichar} (Q_A^\transpose\otimes
  Q_B^{-1})\widetilde{\Bichar}^* = Q_A^\transpose\otimes Q_B^{-1}\).
\end{lemma}

\begin{proof}
  The scaling group \(\tau_t^A\colon \R\to\Aut(A)\) of~\(A\) acts
  through conjugation by~\(Q_A^{\ima t}\) by
  \cite{Woronowicz:Mult_unit_to_Qgrp}*{Theorem 1.5.5}, and similarly
  for~\(B\).  \cite{Meyer-Roy-Woronowicz:Homomorphisms}*{(20)} says
  that~\(V\) is fixed by \(\tau^A_t\otimes \tau^B_t\).  This means
  that~\(\Bichar\) commutes with \(Q_A\otimes Q_B\), as asserted.

  Hence the left-hand side of~\eqref{eq:second_leg_bichar_adapted}
  does not change if we replace \(x\), \(u\), \(z\) and~\(y\) by
  \(Q^{\ima t}_A(x)\), \(Q^{\ima t}_B(u)\), \(Q^{\ima t}_A(z)\)
  and~\(Q^{\ima t}_B(y)\), respectively, for any \(t\in\R\).  Thus the
  right-hand sides also remain the same, that is,
  \[
  \bigl(\conj{z}\otimes Q_Bu\big|
  \widetilde{\Bichar}\big|\conj{x}\otimes Q_B^{-1}y\bigr)
  = \Bigl(\bigl[Q_A^\transpose\bigr]^{-\ima t}\conj{z}\otimes
  Q_B^{\ima t}Q_Bu\Big| \widetilde{\Bichar}\Big|
  \bigl[Q_A^\transpose\bigr]^{-\ima t}\conj{x}\otimes
  Q_B^{\ima t} Q^{-1}_Bu\Bigr).
  \]
  Hence \(\widetilde{\Bichar} =
  \Bigl(\bigl[Q_A^\transpose\bigr]^{\ima t} \otimes Q^{-\ima
    t}_B\Bigr)
  \widetilde{\Bichar}\Bigl(\bigl[Q_A^\transpose\bigr]^{\ima t}
  \otimes Q^{-\ima t}_B\Bigr)\) for all \(t\in\R\).  This says
  that~\(\widetilde{\Bichar}\) commutes with \(Q_A^\transpose\otimes
  Q_B^{-1}\).
\end{proof}

\begin{lemma}
  \label{lem:nice_basis_Q}
  Let~\(Q\) be a self\nb-adjoint, strictly positive operator on a
  Hilbert space~\(\Hils\).  There is an orthonormal
  basis~\((e_i)_{i\in\N}\) in~\(\Hils\) with
  \(e_i\in\dom(Q)\cap\dom(Q^{-1})\) and strong convergence
  \begin{equation}
    \label{eq:nice_basis_Q}
    \sum_{i\in\N} \ket{Q^{-1}e_i}\bra{Q e_i} = \Id_{\Hils}.
  \end{equation}
\end{lemma}

\begin{proof}
  For \(n\in\Z\) and \(\lambda\in [2^{2n-1},2^{2n+1})\), let
  \(f(\lambda)=2^{-2n}\lambda\) and \(g(\lambda) =
  \lambda/f(\lambda)\).  The function
  \[
  \R_{>0} \ni \lambda\mapsto f(\lambda)\in [2^{-1},2)
  \]
  is piecewise linear and bounded with bounded inverse.  Hence the
  Borel functional calculus for self-adjoint operators gives
  \(Q'\defeq f(Q)\), which is self-adjoint and bounded with a bounded
  inverse.  We also get the self-adjoint operator \(Q'' = g(Q)\),
  which has countable spectrum~\(\{2^{2n}\mid n\in\Z\}\).
  Thus~\(\Hils\) is the orthogonal direct sum of the
  \(2^{2n}\)-eigenspaces of~\(\Hils\).  We choose orthonormal bases
  for all these eigenspaces and put them together to an orthonormal
  basis~\((e_i)_{i\in\N}\) of~\(\Hils\).  We have \(Q=Q'Q''\) and
  \(Q^{-1} = (Q')^{-1}(Q'')^{-1}\) by functional calculus.  Since the
  operators~\((Q')^{\pm1}\) are bounded and self-adjoint, \(Q\)
  and~\(Q''\) have the same domain.  Thus \(e_i\in \dom(Q)\cap
  \dom(Q^{-1})\) because it is an eigenvector of~\(Q''\) with some
  positive eigenvalue~\(2^{2n}\) for some \(n\in\Z\) depending
  on~\(i\).  Since~\((Q')^{\pm1}\) are bounded and \(Q''e_i = 2^{2n}
  e_i\), we may rewrite
  \begin{multline*}
    \ket{Q^{-1}e_i}\bra{Q e_i}
    = \ket{(Q')^{-1} (Q'')^{-1}e_i} \bra{Q' Q''e_i}
    = \ket{(Q')^{-1} 2^{-2n} e_i} \bra{Q' 2^{2n} e_i}
    \\= (Q')^{-1} \ket{e_i} \bra{e_i} Q'.
  \end{multline*}
  The sum \(\sum_{i\in\N} \ket{e_i}\bra{e_i}\) converges strongly to
  the identity on~\(\Hils\) because~\((e_i)_{i\in\N}\) is an
  orthonormal basis for~\(\Hils\).  Since~\((Q')^{\pm1}\) are
  bounded, the sum over \((Q')^{-1} \ket{e_i} \bra{e_i} Q'\)
  converges strongly to \((Q')^{-1} \cdot 1 \cdot Q' = 1\).
\end{proof}

The following lemma completes the proof of
Theorem~\ref{the:mang_mult_from_braided_and_standard}.

\begin{lemma}
  \label{lem:mang_mult_from_braided_and_standard}
  In the situation of the proof of
  Theorem~\textup{\ref{the:mang_mult_from_braided_and_standard}},
  \(\Multunit[C]\), \(Q\) and~\(\widetilde{\Multunit}{}^C\)
  verify~\eqref{eq:Multunit_manageable}.
\end{lemma}

\begin{proof}
  We continue in the notation of the proof of
  Theorem~\ref{the:mang_mult_from_braided_and_standard}.  It suffices
  to check~\eqref{eq:Multunit_manageable} when \(x,z,u,y\) are tensor
  monomials: \(x=x_1\otimes x_2\), \(z=z_1\otimes z_2\),
  \(u=u_1\otimes u_2\), \(y=y_1\otimes y_2\) with \(x_1,z_1\in\Hils\),
  \(x_2,z_2\in\Hils[L]\), \(u_1\in\dom(Q)\), \(u_2\in\dom(Q_{\Hils[L]})\),
  \(y_1\in\dom(Q^{-1})\), \(y_2\in\dom(Q_{\Hils[L]}^{-1})\).  This implies the
  assertion for all \(x,y,z,u\).

  Equation~\eqref{eq:braiding} is equivalent to
  \(\Corep{U}_{23}\DuCorep{V}_{34}^* = \DuCorep{V}_{34}^*
  \Corep{U}_{23} Z_{13}^*\), which gives
  \begin{equation}
    \label{eq:W_1234_equiv_form}
    \Multunit[C]_{1234} = \Multunit_{13} \DuCorep{V}^*_{34}
    \Corep{U}_{23} (Z^*_{24} \BrMultunit_{24}) \DuCorep{V}_{34}.
  \end{equation}

  We first concentrate on the part \(\Corep{U}_{23} (Z^*_{24}
  \BrMultunit_{24})\) in~\eqref{eq:W_1234_equiv_form}.  Let
  \((e_i)_{i\in\N}\) be an orthonormal basis of~\(\Hils[L]\).  Then
  \eqref{eq:Corep_U_tilde} and~\eqref{eq:br_manag} give
  \begin{align*}
    &\phantom{{}={}}\left<x_2\otimes u_1\otimes u_2 \middle| \Corep{U}_{12}
      Z^*_{13}\BrMultunit_{13} \middle| z_2\otimes y_1\otimes y_2\right>
    \\&= \sum \left<x_2\otimes u_1\otimes u_2 \middle| \Corep{U}_{12}
      \cdot \bigl(\ket{e_i}\bra{e_i} \otimes 1\otimes 1\bigr)
      \cdot Z^*_{13}\BrMultunit_{13} \middle| z_2\otimes y_1\otimes y_2\right>
    \\&= \sum
    \left< x_2\otimes u_1 \middle| \Corep{U} \middle|
      e_i\otimes y_1\right> \cdot
    \left<e_i\otimes u_2\middle| Z^*\BrMultunit \middle| z_2\otimes y_2\right>
    \\&= \sum
    \left<\conj{z_2}\otimes Q_{\Hils[L]}(u_2) \middle| \widetilde{\BrMultunit} \middle|
      \conj{e_i}\otimes Q_{\Hils[L]}^{-1}(y_2)\right>
    \left< \conj{e_i}\otimes Q(u_1) \middle| \widetilde{\Corep{U}}
      \middle| \conj{x_2} \otimes Q^{-1}(y_1)\right>
    \\&= \left<\conj{z_2}\otimes Q(u_1) \otimes Q_{\Hils[L]}(u_2) \middle|
      \widetilde{\BrMultunit}_{13} \widetilde{\Corep{U}}_{12} \middle|
      \conj{x_2}\otimes Q^{-1}(y_1)\otimes Q_{\Hils[L]}^{-1}(y_2)\right>
  \end{align*}

  Since~\(\DuCorep{V}\) commutes with~\(Q\otimes Q_{\Hils[L]}\), it also
  preserves the domains of \((Q \otimes Q_{\Hils[L]})^{-1}\), and we get an
  equivalent statement if we replace \(u_1\otimes u_2\) and
  \(y_1\otimes y_2\) above by \(\DuCorep{V}(u_1\otimes u_2)\) and
  \(\DuCorep{V}(y_1\otimes y_2)\), respectively.  This gives
  \begin{multline}
    \label{eq:aux_step_manag_W_1234_2}
    \left<x_2\otimes u_1\otimes u_2\middle| \DuCorep{V}_{23}^*
      \Corep{U}_{12} Z^*_{13} \BrMultunit_{13} \DuCorep{V}_{23}
      \middle| z_2\otimes y_1\otimes y_2\right>
    \\= \left<\conj{z_2}\otimes Q(u_1) \otimes Q_{\Hils[L]}(u_2) \middle|
      \DuCorep{V}_{23}^* \widetilde{\BrMultunit}_{13}
      \widetilde{\Corep{U}}_{12} \DuCorep{V}_{23} \middle|
      \conj{x_2}\otimes Q^{-1}(y_1)\otimes Q_{\Hils[L]}^{-1}(y_2)\right>.
  \end{multline}
  Now let \((\epsilon_j)_{j\in\N}\) be a basis of~\(\Hils\) as in
  Lemma~\ref{lem:nice_basis_Q}, that is,
  \[
  \sum_j \ket{Q^{-1}(\epsilon_j)}\bra{Q(\epsilon_j)} = \Id_{\Hils}.
  \]
  We compute
  \begin{align*}
    &\phantom{{}={}}
    \left<x_1\otimes x_2\otimes u_1\otimes u_2 \middle| \Multunit_{13}
      \DuCorep{V}_{34}^* \Corep{U}_{23} Z^*_{24}\BrMultunit_{24}
      \DuCorep{V}_{34} \middle| z_1\otimes z_2\otimes y_1\otimes y_2\right>
    \\&=
    \begin{multlined}[t][.96\linewidth]
      \sum_j
      \bigl<x_1\otimes x_2\otimes u_1\otimes u_2 \big|
      \Multunit_{13}\cdot
      \bigl(1\otimes 1\otimes \big|
      \ket{\epsilon_j}\bra{\epsilon_j} \otimes 1\bigr)
      \cdot \DuCorep{V}_{34}^* \Corep{U}_{23} Z^*_{24}
      \BrMultunit_{24} \DuCorep{V}_{34} \big|
      \\ \big| z_1\otimes z_2\otimes y_1\otimes y_2\bigr>
    \end{multlined}
    \\&= \sum_j
    \left<x_1\otimes u_1 \middle| \Multunit \middle|
      z_1\otimes \epsilon_i\right> \cdot
    \left<x_2\otimes \epsilon_i\otimes u_2 \middle|
      \DuCorep{V}_{23}^* \Corep{U}_{12} (Z^*\BrMultunit)_{13}
      \DuCorep{V}_{23} \middle| z_2\otimes y_1\otimes y_2\right>
    \\&=
    \begin{multlined}[t][.96\linewidth]
      \sum_j
      \left<\conj{z_1}\otimes Q(u_1) \middle| \widetilde{\Multunit}
        \middle| \conj{x_1}\otimes Q^{-1}(\epsilon_i)\right>
      \\\cdot
      \left<\conj{z_2}\otimes Q(\epsilon_i) \otimes Q_{\Hils[L]}(u_2) \middle|
        \DuCorep{V}_{23}^* \widetilde{\BrMultunit}_{13}
        \widetilde{\Corep{U}}_{12} \DuCorep{V}_{23}\middle|
        \conj{x_2}\otimes Q^{-1}(y_1)\otimes Q_{\Hils[L]}^{-1}(y_2)\right>
    \end{multlined}
    \\&=
    \begin{multlined}[t][.96\linewidth]
      \Bigl<\conj{z_1} \otimes \conj{z_2} \otimes Q(u_1) \otimes Q_{\Hils[L]}(u_2)\Big|
      \widetilde{\Multunit}_{13} \DuCorep{V}_{34}^*
      \widetilde{\BrMultunit}_{24} \widetilde{\Corep{U}}_{23}
      \DuCorep{V}_{34}\Big|
      \\ \Big|\conj{x_1} \otimes \conj{x_2} \otimes Q^{-1}(y_1)
      \otimes Q_{\Hils[L]}^{-1}(y_2)\Bigr>.
    \end{multlined}
  \end{align*}
  Thus~\eqref{eq:Multunit_manageable} holds for \(\Multunit[C]\),
  \(Q\) and~\(\widetilde{\Multunit}{}^C\).
\end{proof}

\begin{lemma}
  \label{lemm:brmanag_analysis}
  The unitary
  \(\widetilde{\BrMultunit}\in\U(\conj{\Hils[L]}\otimes\Hils[L])\)
  defined by~\eqref{eq:mang-br-analysis} satisfies the manageability
  condition~\eqref{eq:br_manag} for
  \(\BrMultunit\in\U(\Hils[L]\otimes\Hils[L])\)
  in Theorem~\textup{\ref{the:analysis_qg_projection}}.
\end{lemma}

\begin{proof}
  It suffices to check~\eqref{eq:br_manag} when~\(x,u,y,v\)
  are tensor monomials: \(x=\bar{x}_{1}\otimes x_{2}\),
  \(u=\bar{u}_{1}\otimes u_{2}\),
  \(y=\bar{y}_{1}\otimes y_{2}\),
  \(v=\bar{v}_{1}\otimes v_{2}\),
  with \(x_{1}, x_{2}, y_{1}, y_{2}\in\Hils\),
  \(v_{1}, u_{2}\in\dom(Q_{C})\)
  and \(u_{1}, v_{2}\in \dom(Q_{C}^{-1})\).

  First we focus on the part
  \(\Flip_{23}\widetilde{\DuProjBichar}_{23}\Flip_{23}
  \ProjBichar^{\transpose\otimes\transpose}_{13}\).
  Let \((e_i)_{i\in\N}\)
  be an orthonormal basis of~\(\Hils\)
  in Lemma~\ref{lem:nice_basis_Q}.  Then \((\bar{e}_i)_{i\in\N}\)
  is an orthonormal basis of~\(\conj{\Hils}\) with
  \begin{equation}
    \label{eq:nice_conj_basis}
    \sum_{i\in\N} \ket{\conj{Q e_i}}\bra{\conj{Q^{-1}e_i}} = \Id_{\conj{\Hils}}.
  \end{equation}
  Equation~\eqref{eq:Multunit_manageable} for \(\ProjBichar\)
  and~\(\DuProjBichar\) gives
  \begin{align*}
    & \left<\bar{x}_{1}\otimes x_{2}\otimes \bar{u}_{1}\middle|
      \Flip_{23}\widetilde{\DuProjBichar}_{23}\Flip_{23} \ProjBichar^{\transpose\otimes\transpose}_{13}
      \middle| \bar{y}_{1}\otimes y_{2}\otimes \bar{v}_{1}\right>\\
    &=\sum \left<\bar{x}_{1}\otimes x_{2}\otimes \bar{u}_{1}\middle|
      \Flip_{23}\widetilde{\DuProjBichar}_{23}\Flip_{23}
      \cdot  \bigl(1\otimes 1\otimes \ket{e_i}\bra{e_i} \bigr) \cdot
      \ProjBichar^{\transpose\otimes\transpose}_{13}
      \middle| \bar{y}_{1}\otimes y_{2}\otimes \bar{v}_{1}\right>\\
    &=\sum \left<x_{2}\otimes\bar{u}_{1}\middle|
      \Flip\widetilde{\DuProjBichar}\Flip\middle|
      y_{2}\otimes\bar{e}_{i}\right>
      \left<\bar{x}_{1}\otimes\bar{e}_{i}\middle|
      \ProjBichar^{\transpose\otimes\transpose}\middle|
      \bar{y}_{1}\otimes \bar{v}_{1}\right>\\
    &=\sum \left<\bar{u}_{1}\otimes x_{2}\middle|
      \widetilde{\DuProjBichar}\middle|
      \bar{e}_{i}\otimes y_{2}\right>
      \left<y_{1}\otimes v_{1}\middle|
      \ProjBichar \middle|
      x_{1}\otimes e_{i}\right>\\
    &=\sum \left<e_{i}\otimes Q_{C}^{-1}(x_{2})\middle|
      \DuProjBichar\middle|
      u_{1}\otimes Q_{C}(y_{2})\right>
      \left<\bar{x}_{1}\otimes Q_{C}(v_{1})\middle|
      \widetilde{\ProjBichar}\middle|
      \bar{y}_{1}\otimes Q_{C}^{-1}(e_{i})\right>.
  \end{align*}
  Lemma~\ref{lem:V_tilde_commute_Q} shows that~\(Q_C\otimes C_C\)
  commutes with~\(\DuProjBichar\).  Hence
  \begin{align*}
    & \left<\bar{x}_{1}\otimes x_{2}\otimes \bar{u}_{1}\middle|
      \Flip_{23}\widetilde{\DuProjBichar}_{23}\Flip_{23} \ProjBichar^{\transpose\otimes\transpose}_{13}
      \middle| \bar{y}_{1}\otimes y_{2}\otimes \bar{v}_{1}\right>\\
    &=\sum \left<Q_{C}(e_{i})\otimes x_{2}\middle|
      \DuProjBichar\middle|
      Q_{C}^{-1}(u_{1})\otimes y_{2}\right>
      \left<\bar{x}_{1}\otimes Q_{C}(v_{1})\middle|
      \widetilde{\ProjBichar}\middle|
      \bar{y}_{1}\otimes Q_{C}^{-1}(e_{i})\right>\\
    &=\sum \left<x_{2}\otimes Q_{C}(e_{i})\middle|
      \ProjBichar^*\middle|
      y_{2}\otimes Q_{C}^{-1}(u_{1})\right>
      \left< y_{1}\otimes \conj{Q_{C}^{-1}(e_{i})}\middle|
      \widetilde{\ProjBichar}{}^{\transpose\otimes\transpose}\middle|
      x_{1}\otimes \conj{Q_{C}(v_{1})}\right>\\
    &=\sum \left<\bar{y}_{2}\otimes \conj{Q_{C}^{-1}(u_{1})}\middle|
      (\ProjBichar^*)^{\transpose\otimes\transpose} \middle|
      \bar{x}_{2}\otimes \conj{Q_{C}(e_{i})}\right>
      \left< y_{1}\otimes \conj{Q_{C}^{-1}(e_{i})}\middle|
      \widetilde{\ProjBichar}{}^{\transpose\otimes\transpose}\middle|
      x_{1}\otimes \conj{Q_{C}(v_{1})}\right>.
  \end{align*}
  Now~\eqref{eq:nice_conj_basis} gives
  \begin{multline}
    \label{eq:br-man_aux1}
    \left<\bar{x}_{1}\otimes x_{2}\otimes \bar{u}_{1}\middle|
      \Flip_{23}\widetilde{\DuProjBichar}_{23}\Flip_{23} \ProjBichar^{\transpose\otimes\transpose}_{13}
      \middle| \bar{y}_{1}\otimes y_{2}\otimes \bar{v}_{1}\right>
    \\=  \left<y_{1}\otimes \bar{y}_{2}\otimes \conj{Q_{C}^{-1}(u_{1})}
      \middle|
      (\ProjBichar^*)^{\transpose\otimes\transpose}_{23}
      \widetilde{\ProjBichar}{}^{\transpose\otimes\transpose}_{13}
      \middle|
      x_{1}\otimes \bar{x}_{2}\otimes \conj{Q_{C}(v_{1})}\right>.
  \end{multline}
  A similar computation gives
  \begin{multline}
    \label{eq:br-man_aux2}
    \left<\bar{x}_{1}\otimes x_{2}\otimes u_{2}\middle|
      \Multunit[C]_{23}(\widetilde{\Multunit}{}^C_{13})^*\middle|
      \bar{y}_{1}\otimes y_{2}\otimes v_{2}\right>
    \\= \left<y_{1}\otimes \bar{y}_{2}\otimes Q_{C}(u_{2})\middle|
      \widetilde{\Multunit}{}^C_{23} (\Multunit[C]_{13})^*\middle|
      x_{1}\otimes \bar{x}_{2}\otimes Q_{C}^{-1}(v_{2})\right>.
  \end{multline}
  Let \((e_{j})_{j\in\N}\)
  be an orthonormal basis of~\(\Hils\).
  Equations \eqref{eq:br-man_aux1} and~\eqref{eq:br-man_aux2} imply
  \begin{align*}
    &\phantom{{}={}}
      \left<\bar{x}_{1}\otimes x_{2}\otimes\bar{u}_{1}\otimes u_{2}\middle|
      \Flip_{23}\widetilde{\DuProjBichar}_{23}\Flip_{23}
      \ProjBichar^{\transpose\otimes\transpose}_{13} \Multunit[C]_{24}
      (\widetilde{\Multunit}{}^{C}_{14})^* \middle|
      \bar{y}_{1}\otimes y_{2}\otimes \bar{v}_{1}\otimes v_{2}\right>
    \\&=
        \begin{multlined}[t][.96\linewidth]
          \sum_{j,k} \Bigl<\bar{x}_{1}\otimes x_{2}\otimes\bar{u}_{1}\otimes
          u_{2}\Big|
          \Flip_{23}\widetilde{\DuProjBichar}_{23}\Flip_{23}
          \ProjBichar^{\transpose\otimes\transpose}_{13}
          \cdot \bigl(\ket{\bar{e}_{j}}\bra{\bar{e}_{j}}\otimes \ket{e_{k}}\bra{e_{k}}\otimes
          1 \otimes 1\bigr)
          \\\cdot
          \Multunit[C]_{24} (\widetilde{\Multunit}{}^{C}_{14})^* \Big|
          \bar{y}_{1}\otimes y_{2}\otimes \bar{v}_{1}\otimes v_{2}\Bigr>
        \end{multlined}
    \\&=
        \begin{multlined}[t][.96\linewidth]
          \sum_{j,k} \Bigl<\bar{x}_{1}\otimes x_{2}\otimes\bar{u}_{1}\Big|
          \Flip_{23}\widetilde{\DuProjBichar}_{23}\Flip_{23}
          \ProjBichar^{\transpose\otimes\transpose}_{13} \Big|
          \bar{e}_{j}\otimes e_{k}\otimes \bar{v}_{1}\Bigr>
          \\
          \Bigl<\bar{e}_{j}\otimes e_{k}\otimes u_{2} \Big|
          \Multunit[C]_{23} (\widetilde{\Multunit}{}^{C}_{13})^* \Big|
          \bar{y}_{1}\otimes y_{2}\otimes v_{2}\Bigr>
        \end{multlined}
    \\&=\begin{multlined}[t][.96\linewidth]
      \sum_{j,k} \left<y_{1}\otimes \bar{y}_{2}\otimes Q_{C}(u_{2})\middle|
            \widetilde{\Multunit}{}^{C}_{23} (\Multunit[C]_{13})^*\middle|
            e_{j}\otimes \bar{e}_{k}\otimes Q_{C}^{-1}(v_{2})\right>
          \\
          \left<e_{j}\otimes \bar{e}_{k}\otimes \conj{Q_{C}^{-1}(u_{1})}
            \middle|
            (\ProjBichar^*)^{\transpose\otimes\transpose}_{23}
            \widetilde{\ProjBichar}{}^{\transpose\otimes\transpose}_{13}
            \middle|
            x_{1}\otimes \bar{x}_{2}\otimes \conj{Q_{C}(v_{1})}\right>
        \end{multlined}
    \\&=
        \begin{multlined}[t][.96\linewidth]
          \Bigl< y_{1}\otimes \bar{y}_{2}\otimes \conj{Q_{C}^{-1}(u_{1})}
          \otimes Q_{C}(u_{2})
          \Bigm|
          \widetilde{\Multunit}{}^{C}_{24} (\Multunit[C]_{14})^*
          (\ProjBichar^*)^{\transpose\otimes\transpose}_{23}
          \widetilde{\ProjBichar}{}^{\transpose\otimes\transpose}_{13}
          \Bigm|\\
          x_{1}\otimes\bar{x}_{2}\otimes \conj{Q_{C}(v_{1})}\otimes
          Q_{C}^{-1}(v_{2}) \Bigr>
        \end{multlined}
  \end{align*}
  This is the equation we have to check.
\end{proof}

\end{document}